\documentclass[10pt]{amsart} 
\usepackage{amscd,amssymb}
\usepackage{enumerate} 
\usepackage[all]{xypic} 
\usepackage{multirow}
\usepackage{hhline}
\usepackage{verbatim} 
\usepackage{comment}

\theoremstyle{plain} 
\newtheorem{Thm}[equation]{Theorem}
\newtheorem*{Thm*}{Theorem} 
\newtheorem*{MainThm*}{Main Theorem} 
\newtheorem{Cor}[equation]{Corollary}
\newtheorem{Prop}[equation]{Proposition}
\newtheorem{Lem}[equation]{Lemma} 
\newtheorem{Rmk}[equation]{Remark}

\newtheorem*{hypo1}{Hypothesis ($\ast$)}
\newtheorem*{hypo2}{Hypothesis ($\ast\ast$)}
\newtheorem*{hypo3}{Hypothesis ($\ast\ast\ast$)}

\numberwithin{equation}{section}

\newcommand{\Hom}{\operatorname{Hom}}

\newcommand{\GL}{\operatorname{GL}}

\newcommand{\Aut}{\operatorname{Aut}}
\newcommand{\SL}{\operatorname{SL}}

\newcommand{\Ind}{\operatorname{Ind}}
\newcommand{\cInd}{\operatorname{c-Ind}}

\newcommand{\diag}{\operatorname{diag}}

\newcommand{\GLt}{\widetilde{\operatorname{GL}}}
\newcommand{\SLt}{\widetilde{\operatorname{SL}}}

\newcommand{\Pt}{\widetilde{P}}
\newcommand{\Qt}{\widetilde{Q}}

\newcommand{\MQt}{\widetilde{M_Q}}

\newcommand{\GLtt}{\widetilde{\operatorname{GL}}^{(2)}}

\newcommand{\Mtt}{\widetilde{M}^{(2)}}

\newcommand{\GLtn}{\widetilde{\operatorname{GL}}^{(n)}}
\newcommand{\Mtn}{\widetilde{M}^{(n)}}
\newcommand{\MMtn}{\widetilde{M'}^{(n)}}
\newcommand{\Mtnn}{(\widetilde{M})^{(n)}}
\newcommand{\Mtnni}{(\widetilde{M_i})^{(n)}}
\newcommand{\Mn}{M^{(n)}}

\newcommand{\sigGLt}{{^{\sigma}\widetilde{\operatorname{GL}}}}

\newcommand{\cMt}{{^c\widetilde{M}}}
\newcommand{\cMtn}{{^c\widetilde{M}^{(n)}}}

\newcommand{\Zt}{\widetilde{Z}} 
\newcommand{\Mt}{\widetilde{M}}
\newcommand{\MMt}{\widetilde{M'}}
\newcommand{\Ht}{\widetilde{H}}

\newcommand{\Vt}{\widetilde{V}}

\newcommand{\timest}{\widetilde{\times}}
\newcommand{\otimest}{\widetilde{\otimes}}

\newcommand{\varphit}{\widetilde{\varphi}}

\newcommand{\pin}{\pi^{(n)}}

\newcommand{\gt}{\tilde{g}}
\newcommand{\mt}{\tilde{m}}

\newcommand{\OF}{\mathcal{O}_F} 
\newcommand{\OFv}{\mathcal{O}_{F_v}}

\newcommand{\G}{\mathbb{G}}

\newcommand{\C}{\mathbb C}
\newcommand{\A}{\mathbb{A}} 
\newcommand{\Q}{\mathbb{Q}}

\newcommand{\Z}{\mathbb{Z}}
\newcommand{\R}{\mathbb{R}}

\newcommand{\ZZ}{\mathcal{Z}}

\newcommand{\M}{\mathfrak{M}}

\newcommand{\supp}{\operatorname{supp}}
\newcommand{\Det}{\operatorname{Det}}

\newcommand{\longhookrightarrow}
{\ensuremath{\lhook\joinrel\relbar\joinrel\rightarrow}}

\newcommand{\s}{\mathbf{s}}
\newcommand{\ttt}{\mathbf{t}}
\newcommand{\sss}{\mathfrak{s}}
\newcommand{\ie}{{\em i.e. }}

\newcommand{\tth}{{\text{th}}}

\title[Metaplectic tensor products]{Metaplectic tensor products for
automorphic representations of $\GLt(r)$} 
\author{Shuichiro Takeda}
\address{Shuichiro Takeda: Mathematics Department, University of Missouri, Columbia, 202
  Math Sciences Building, Columbia, MO, 65211}
\email{takedas@missouri.edu}

\begin{document}

\maketitle

\begin{abstract} 
Let
$M=\GL_{r_1}\times\cdots\times\GL_{r_k}\subseteq\GL_r$ be a Levi
subgroup of $\GL_r$, where $r=r_1+\cdots+r_k$, and $\Mt$ its metaplectic preimage
in the $n$-fold metaplectic cover $\GLt_r$ of $\GL_r$. For automorphic
representations $\pi_1,\dots,\pi_k$ of $\GLt_{r_1}(\A),\dots,\GLt_{r_k}(\A)$, 
we construct (under a certain
technical assumption, which is always satisfied when $n=2$) an
automorphic representation $\pi$
of $\Mt(\A)$ which can be considered as the ``tensor product'' of the
representations $\pi_1,\dots,\pi_k$. This is
the global analogue of the metaplectic tensor product
defined by P. Mezo in the sense that locally at each place $v$,
$\pi_v$ is equivalent to the local metaplectic tensor product of
$\pi_{1,v},\dots,\pi_{k,v}$ defined by Mezo. Then we show that if all
of $\pi_i$ are cuspidal (resp. square-integrable modulo center), then
the metaplectic tensor product is  cuspidal (resp. square-integrable
modulo center). We also show that (both
locally and globally) the metaplectic tensor product behaves in the
expected way under the action of a Weyl group element, and show the
compatibility with parabolic inductions.
\end{abstract}


\section{\bf Introduction}


Let $F$ be either a local field of characteristic 0 or a number field,
and $R$ be $F$ if $F$ is local and the ring of adeles $\A$ if $F$ is global.
Consider the group $\GL_r(R)$. For a
partition $r=r_1+\cdots+r_k$ of $r$, one has the Levi subgroup 
\[
M(R):=\GL_{r_1}(R)\times\cdots\times\GL_{r_k}(R)\subseteq\GL_r(R).
\]
Let $\pi_1,\dots,\pi_k$ be irreducible admissible
(resp. automorphic) representations of
$\GL_{r_1}(R),\dots,\GL_{r_k}(R)$ where $F$ is local (resp. $F$ is global). Then it is a trivial construction to obtain the
representation $\pi_1\otimes\cdots\otimes\pi_k$, which is an
irreducible admissible (resp. automorphic) representation of the Levi
$M(R)$. Though highly trivial, this construction is of great importance in
the representation theory of $\GL_r(R)$.

Now if one considers the metaplectic $n$-fold cover $\GLt_r(R)$
constructed by Kazhdan and Patterson in \cite{KP}, the analogous
construction turns out to be far from trivial. Namely for the
metaplectic preimage $\Mt(R)$ of $M(R)$ in $\GL_r(R)$ and
representations $\pi_1,\dots,\pi_k$  of the metaplectic $n$-fold
covers $\GLt_{r_1}(R),\dots,\GLt_{r_k}(R)$, one cannot construct a
representation of $\Mt(R)$ simply by taking the tensor product
$\pi_1\otimes\cdots\otimes\pi_k$. This is simply because $\Mt(R)$ is not the
direct product of $\GLt_{r_1}(R),\dots,\GLt_{r_k}(R)$, namely
\[
\Mt(R)\ncong\GLt_{r_1}(R)\times\dots\times\GLt_{r_k}(R),
\]
and even worse there is no natural map between them.

When $F$ is a local field, for irreducible admissible
representations $\pi_1,\dots,\pi_k$  of
$\GLt_{r_1}(F),\dots,\GLt_{r_k}(F)$, P. Mezo (\cite{Mezo}), whose
work, we believe, is based on the work by Kable \cite{Kable2},
constructed an irreducible admissible representation of the Levi
$\Mt(F)$, which can be called the ``metaplectic tensor product'' of
$\pi_1,\dots,\pi_k$, and characterized it uniquely up to certain
character twists. (His construction will be reviewed and expanded further in
Section \ref{S:Mezo}.)

The theme of the paper is to carry out a construction analogous to
Mezo's when $F$ is a number field, and our main theorem is
\begin{MainThm*}
Let $M=\GL_{r_1}\times\cdots\times\GL_{r_k}$ be a Levi subgroup of $\GL_r$,
and let $\pi_1,\dots,\pi_k$ be unitary automorphic subrepresentations of
$\GLt_{r_1}(\A),\dots,\GLt_{r_k}(\A)$. Assume that $M$ and $n$ are
such that Hypothesis ($\ast$) (see section \ref{S:abelian}) is
satisfied, which is always the case if  $n=2$.  Then there exists an
automorphic representation $\pi$ of $\Mt(\A)$ such that
\[
\pi\cong\otimest'_v\pi_v,
\]
where each $\pi_v$ is the local metaplectic tensor product of
Mezo. Moreover, if $\pi_1,\dots,\pi_k$ are cuspidal
(resp. square-integrable modulo center), then $\pi$ is cuspidal
(resp. square-integrable modulo center).
\end{MainThm*}

In the above theorem, $\otimest_v'$ indicates the metaplectic restricted
tensor product, the meaning of which will be explained later in the
paper. The existence and the local-global compatibility in the main
theorem are proven in
Theorem \ref{T:main}, and the cuspidality and square-integrability are
proven in Theorem \ref{T:cuspidal} and Theorem
\ref{T:square_integrable}, respectively.

Let us note that by unitary, we mean that $\pi_i$ is equipped with a
Hermitian structure invariant under the action of the group. Also
we require $\pi_i$ be an automorphic subrepresentation, so that it is
realized in a subspace of automorphic forms and hence each element in
$\pi_i$ is indeed an automorphic form. (Note that usually an
automorphic representation is a subquotient.) We need those two
conditions for technical reasons, and they are satisfied if $\pi_i$ is
in the discrete spectrum, namely either cuspidal or residual.

Also we should emphasize that if $n>2$, we do not know if our
construction works unless we impose a technical assumption as in
Hypothesis ($\ast$). We will show in Appendix \ref{A:topology} that this assumption is
always satisfied if $n=2$, and if $n>2$ it is satisfied, for example,
if $\gcd(n, r-1+2cr)=1$, where $c$ is the parameter to be
explained. We hope that even for $n>2$ it is always satisfied, though
at this moment we do not know how to prove it.

As we will see, strictly speaking the metaplectic tensor product of $\pi_1,\dots,\pi_k$
might not be unique even up to equivalence but is dependent on a
character $\omega$ on the
center $Z_{\GLt_r}$ of $\GLt_r$. Hence we write
\[
\pi_\omega:=(\pi_1\otimest\cdots\otimest\pi_k)_\omega
\]
for the metaplectic tensor product to emphasize the dependence on $\omega$. \\

Also we will establish a couple of important properties of the metaplectic
tensor product both locally and globally. The first one is that the
metaplectic tensor product behaves in the expected way under the
action of the Weyl group. Namely

\quad\\
\noindent{\bf Theorem \ref{T:Weyl_group_local}
and \ref{T:Weyl_group_global}.} {\it Let $w\in W_M$ be a Weyl group
element of $\GL_r$ that only permutes the $\GL_{r_i}$-factors of
$M$. Namely for each
$(g_1,\dots,g_k)\in\GL_{r_1}\times\cdots\times\GL_{r_k}$, we have $w
(g_1,\dots,g_k)w^{-1}=(g_{\sigma(1)},\dots,g_{\sigma(k)})$ for a
permutation $\sigma\in S_k$ of $k$ letters. Then both locally and
globally, we have
\[
^{w}(\pi_1\otimest\cdots\otimest\pi_k)_\omega
\cong(\pi_{\sigma(1)}\otimest\cdots\otimest\pi_{\sigma(k)})_\omega,
\]
where the left hand side is the twist of
$(\pi_1\otimest\cdots\otimest\pi_k)_\omega$ by $w$.
}  \quad\\

The second important property we establish is the compatibility of the
metaplectic tensor product with parabolic inductions. Namely 

\quad\\
\noindent{\bf Theorem \ref{T:induction_local} and \ref{T:induction_global}.}
 {\it 
Both locally and globally, let $P=MN\subseteq\GL_r$ be the standard
parabolic subgroup whose Levi
part is $M=\GL_{r_1}\times\cdots\times\GL_{r_k}$. Further for each
$i=1,\dots,k$ let $P_i=M_iN_i\subseteq\GL_{r_i}$ be the standard
parabolic of $\GL_{r_i}$ whose Levi part is
$M_i=\GL_{r_{i,1}}\times\cdots\times\GL_{r_{i, l_i}}$. For each $i$, we are
given a representation 
\[
\sigma_i:=(\tau_{i,1}\,\otimest\cdots\otimest\,\tau_{i,l_i})_{\omega_i}
\]
of $\Mt_i$, which is given as the metaplectic tensor product of the
representations $\tau_{i,1},\dots,\tau_{i,l_i}$ of
$\GLt_{r_{i,1}},\dots,\GLt_{r_{i, l_i}}$. Assume that $\pi_i$ is an
irreducible constituent of the induced representation
$\Ind_{\Pt_i}^{\GLt_{r_i}}\sigma_i$. Then the metaplectic tensor
product
\[
\pi_\omega:=(\pi_1\,\otimest\cdots\otimest\,\pi_k)_\omega
\]
is an irreducible constituent of the induced representation
\[
\Ind_{\Qt}^{\Mt}(\tau_{1,1}\,\otimest\cdots\otimest\,\tau_{1, l_1}\,\otimest\cdots\otimest
\,\tau_{k,1}\,\otimest\cdots\otimest\,\tau_{k,l_k})_\omega,
\]
where $Q$ is the standard parabolic subgroup of $M$ whose Levi part is
$M_1\times\cdots\times M_k$.
}  \quad\\

In the above two theorems, it is implicitly assumed that if $n>2$ and
$F$ is global, the metaplectic tensor products in the theorems exist
in the sense that Hypothesis ($\ast$) is satisfied for
the relevant Levi subgroups.

Finally at the end, we will discuss the behavior of the global
metaplectic tensor product when restricted to a smaller Levi. Namely
for each automorphic form
$\varphi\in(\pi_1\otimest\cdots\otimest\pi_k)_\omega$ in the
metaplectic tensor product, we would like to know which space the
restriction $\varphi|_{\Mt_2}$ belongs to, where
$M_2=\{I_{r_1}\}\times\GL_{r_2}\times\cdots\times\GL_{r_k}\subset M$,
viewed as a subgroup of $M$,
is the Levi for the smaller group $\GL_{r-r_1}$. Somehow similarly to
the non-metaplectic case, the restriction
$\varphi|_{\Mt_2}$  belongs to the metaplectic tensor product of
$\pi_2, \dots,\pi_k$. But the precise statement is a bit more
subtle. Indeed, we will prove

\quad\\
\noindent{\bf Theorem \ref{T:restriction}.}
 {\it Assume Hypothesis ($\ast\ast$) (see section \ref{S:restriction}) is satisfied, which is
   always the case if $n=2$ or
   $\gcd(n,r-1+2cr)=\gcd(n,r-r_1-1+2c(r-r_1))=1$. Then there exists a
   realization of the metaplectic tensor product
   $\pi_\omega=(\pi_1\otimest\cdots\otimest\pi_k)_{\omega}$
   such that, if we let
\[
\pi_\omega\|_{\Mt_2(\A)}=\{\varphit|_{\Mt_2(\A)}:\varphit\in \pi_\omega\},
\]
then
\[
\pi_\omega\|_{\Mt_2(\A)}\subseteq \bigoplus_\delta m_\delta(\pi_2\otimest\cdots\otimest\pi_k)_{\omega_\delta},
\]
as a representation of $\Mt_2(\A)$,
where
$(\pi_2\otimest\cdots\otimest\pi_k)_{\omega_\delta}$
is the metaplectic tensor product of $\pi_2,\dots,\pi_k$,
$\omega_\delta$ is a certain character twisted by 
$\delta$ which runs through a finite subset of $\GL_{r_1}(F)$ and
$m_\delta\in\Z^{\geq 0}$ is a multiplicity. 
}
\quad\\

The precise meanings of the notations will be explained in
Section \ref{S:restriction}.

\quad

Even though the theory of metaplectic groups is an important
subject in representation theory and automorphic forms and used in
various important literatures such as \cite{Banks2, F, BBL,
  BFH, BH, Suzuki} to name a few, and most importantly for the purpose
of this paper, \cite{BG} which concerns the symmetric square
$L$-function on $\GL(r)$ , it has an
unfortunate history of numerous technical errors and as a result published
literatures in this area are often marred by those
errors which compromise their reliability. As is pointed out in
\cite{BLS}, this is probably due to the deep and subtle nature of the
subject. At any rate, this has made people who work in the area particularly wary
of inaccuracies in new works. For this reason, especially considering
the foundational nature of this paper, we tried
to provide detailed proofs for most of our assertions at the expense of the length
of the paper. Furthermore, for large part, we
rely only on the two fundamental works, namely the work on the
metaplectic cocycle by Banks, Levy and Sepanski (\cite{BLS}) and the
local metaplectic tensor product by Mezo (\cite{Mezo}), both of which
are written carefully enough to be reliable. 

\quad

Finally, let us mention that the result of this paper will be used in our
forthcoming \cite{Takeda2}, which will improve the main result of
\cite{Takeda1}.

\quad

\begin{center}{\bf Notations}\end{center}

Throughout the paper, $F$ is a local field of characteristic zero or
a number field. If $F$ is a number field, we denote the ring of adeles
by $\A$. As we did in the introduction we often use the notation
\[
R=\begin{cases} F\quad\text{if $F$ is local}\\
\A\quad\text{if $F$ is global}.
\end{cases}
\]
The symbol $R^\times$ has the usual meaning and we set
\[
R^{\times n}=\{a^n: a\in R^\times\}.
\]
Both locally and globally, we denote by $\OF$ the ring of integers of
$F$. For each algebraic group $G$ over a global $F$, and $g\in G(\A)$,
by $g_v$ we mean the $v^{\tth}$ component of $g$, and so $g_v\in
G(F_v)$.

For a positive integer $r$, we denote by $I_r$ the $r\times r$
identity matrix. Throughout we fix an integer $n\geq 2$, and we let
$\mu_n$ be the group of $n^{\text{th}}$ roots of unity in the
algebraic closure of the prime field. We always assume that
$\mu_n\subseteq F$, where $F$ is either local or global. So in
particular if $n\geq 3$, for archimedean $F$, we have
$F=\C$, and for global $F$, $F$ is totally complex.

The symbol $(-,-)_F$ denotes the $n^{\text{th}}$ order Hilbert
symbol of $F$ if $F$ is local, which is a bilinear map
\[
(-,-)_F:F^\times\times F^\times\rightarrow\mu_n.
\]
If $F$ is global, we let $(-,-)_\A:=\prod_v(-,-)_{F_v}$, where the
product is finite. We sometimes write simply $(-,-)$ for
$(-,-)_R$ when there is no danger of confusion. Let us recall that
both locally and globally the
Hilbert symbol has the following properties:
\begin{gather*}
(a,b)^{-1}=(b,a)\\
(a^n,b)=(a, b^n)=1\\
(a, -a)=1
\end{gather*}
for $a, b\in R^\times$. Also for the global Hilbert symbol, we have
the product formula $(a,b)_\A=1$ for all $a, b\in F^\times$.

We fix a partition $r_1+\cdots+r_k=r$ of $r$, and we let
\[
M=\GL_{r_1}\times\cdots\times\GL_{r_k}\subseteq\GL_r
\]
and assume it is embedded diagonally as usual. We often denote each
element $m\in M$ by
\[
m=\begin{pmatrix}g_1&&\\ &\ddots&\\ &&g_k\end{pmatrix},
\quad\text{or}\quad
m=\diag(g_1,\dots,g_k)
\]
or sometimes simply $m=(g_1,\dots,g_k)$, where $g_i\in\GL_{r_i}$.

For $\GL_r$, we let $B=TN_B$ be the Borel subgroup with the unipotent
radical $N_B$ and the maximal torus $T$.

If $\pi$ is a representation of a group $G$, we denote the space of
$\pi$ by $V_{\pi}$, though we often confuse $\pi$ with $V_\pi$ when
there is no danger of confusion. We say $\pi$ is unitary if $V_\pi$ is
equipped with a Hermitian structure invariant under the action of
$G$, but we do not necessarily assume that the space $V_\pi$ is
complete. Now assume that the
space $V_\pi$ is a space of functions or maps on the group $G$ and
$\pi$ is the representation of $G$ on $V_\pi$ defined by right translation. (This
is the case, for example, if $\pi$ is an automorphic subrepresentation.)
Let $H\subseteq G$ be a subgroup. Then we define $\pi\|_H$ to be the
representation of $H$ realized in the space
\[
V_{\pi\|_H}:=\{f|_H: f\in V_\pi\}
\]
of restrictions of $f\in V_\pi$ to $H$, on which $H$ acts by right
translation. Namely $\pi\|_H$ is the representation obtained by
restricting the functions in $V_\pi$. Occasionally, we confuse
$\pi\|_H$ with its space when there is no danger of confusion. Note
that there is an $H$-intertwining surjection
$\pi|_H\rightarrow\pi\|_H$, where $\pi|_H$ is the (usual) restriction of
$\pi$ to $H$.

For any group $G$ and elements $g, h\in G$, we define $^gh=ghg^{-1}$.
For a subgroup $H\subseteq G$ and a representation $\pi$ of $H$, we
define $^g\pi$ to be the representation of $gHg^{-1}$ defined by
$^g\pi(h')=\pi(g^{-1}h'g)$ for $h'\in gHg^{-1}$.

We let $W$ be the set of all $r\times r$ permutation
matrices, so for each element $w\in W$ each row and each column has
exactly one 1 and  all the other entries are 0. The Weyl group of
$\GL_r$ is identified with $W$. Also for our Levi $M$, we let
$W_M$ be the subset of $W$ that only permutes the
$\GL_{r_i}$-blocks of $M$. Namely  $W_M$ is the collection of block
matrices 
\[
W_M:=\{(\delta_{\sigma(i),j}I_{r_j})\in W : \sigma\in S_k\},
\]
where $S_k$ is the permutation group of $k$ letters. Though $W_M$
is not a group in general, it is in bijection with $S_k$. Note that if
$w\in W_M$ corresponds to
$\sigma\in S_k$, we have
\[
^w\diag(g_1,\dots,g_k)=w\diag(g_1,\dots,g_k)w^{-1}
=\diag(g_{\sigma^{-1}(1)},\dots, g_{\sigma^{-1}(k)}).
\]
In addition to $W$, in order to use various results from \cite{BLS},
which gives a detailed description of the 2-cocycle $\sigma_r$
defining our metaplectic group $\GLt_r$,
one sometimes needs to use another set of representatives of the Weyl group
elements, which we as well as \cite{BLS} denote by $\M$. The set $\M$
is chosen to be such that for each element
$\eta\in\mathfrak{M}$ we have $\det(\eta)=1$. To be
more precise, each $\eta$ with length $l$ is written as
\[
\eta=w_{\alpha_1}\cdots w_{\alpha_l}
\]
where $w_{\alpha_i}$ is a simple root refection corresponding to a
simple root $\alpha_i$ and is the matrix of the
form
\[
w_{\alpha_i}=\begin{pmatrix}
\ddots&&&\\
&&-1&\\
&1&&\\
&&&\ddots
\end{pmatrix}.
\]
Though the set $\mathfrak{M}$ is not a group, it has the advantage
that we can compute the cocycle $\sigma_r$ in a systematic way as one
can see in \cite{BLS}. For each $w\in W$, we denote by $\eta_w$ the
corresponding element in $\M$. If $w\in W_M$, one can see that
$\eta_w$ is of the form $(\varepsilon_j\delta_{\sigma(i),j}I_{r_j})$
for $\varepsilon_j\in\{\pm 1\}$. Namely $\eta_w$ is a $k\times k$
block matrix in which the non-zero entries are either $I_{r_j}$ or
$-I_{r_j}$.

\quad

\begin{center}{\bf Acknowledgements}\end{center}
The author would like to thanks Paul Mezo for reading an early draft
and giving him helpful comments, and Jeff Adams for sending him the
preprint \cite{Adams} and explaining the construction of the
metaplectic tensor product for the real case.
The author is partially supported by NSF grant DMS-1215419. Also
part of this research was done when he was visiting the I.H.E.S. in
the summer of 2012 and
he would like to thank their hospitality.

\quad


\section{\bf The metaplectic cover $\GLt_r$ of 
$\GL_r$}\label{S:metaplectic_cover}


In this section, we review the theory of the metaplectic $n$-fold cover $\GLt_r$ of
$\GL_r$ for both local and global cases, which was originally
constructed by Kazhdan and Patterson in \cite{KP}.


\subsection{\bf The local metaplectic cover $\GLt_r(F)$}


Let $F$ be a (not necessarily non-archimedean) local field of
characteristic $0$ which contains all the $n^\text{th}$ roots of unity.
In this paper, by the metaplectic $n$-fold cover
$\GLt_r(F)$ of $\GL_r(F)$ with a fixed parameter $c\in\{0,\dots,n-1\}$, we
mean the central extension of $\GL_r(F)$ by $\mu_n$ as constructed
by Kazhdan and Patterson in \cite{KP}. To be more specific, let us first
recall that the $n$-fold cover $\SLt_{r+1}(F)$ of $\SL_{r+1}(F)$ was
constructed by Matsumoto in \cite{Matsumoto}, and there is an
embedding
\begin{equation}\label{E:embedding0}
l_0:\GL_r(F)\rightarrow\SL_{r+1}(F),\quad 
g\mapsto\begin{pmatrix}\det(g)^{-1}&\\ &g\end{pmatrix}.
\end{equation}
Our metaplectic
$n$-fold cover $\GLt_r(F)$ with $c=0$ is the preimage of $l_0(\GL_r(F))$ via the
canonical projection $\SLt_{r+1}(F)\rightarrow\SL_{r+1}(F)$. 
Then $\GLt_r(F)$ is defined by a 2-cocycle
\[
\sigma_r:\GL_r(F)\times\GL_r(F)\rightarrow\mu_n.
\]
For arbitrary parameter $c\in\{0,\dots, n-1\}$,
we define the twisted cocycle $\sigma_r^{(c)}$ by
\[
\sigma_r^{(c)}(g, g')=\sigma_r(g, g')(\det(g),\det(g'))^c_F
\]
for $g, g'\in\GL_r(F)$, where recall from the notation section that
$(-,-)_F$ is the $n^{\text{th}}$ order Hilbert symbol for $F$. The
metaplectic cover with a parameter $c$ is defined by this
cocycle. In \cite{KP}, the metaplectic cover with parameter
$c$ is denoted by $\GLt_r^{(c)}(F)$ but we avoid this notation. This
is because later we will introduce the notation $\GLtn_r(F)$, which
has a completely different meaning. Also we suppress the superscript
$(c)$ from the notation of
the cocycle and always agree that the parameter $c$ is fixed
throughout the paper. 

By carefully studying Matsumoto's construction, Banks, Levy, and Sepanski
(\cite{BLS}) gave an explicit description of the 2-cocycle $\sigma_r$
and shows that their 2-cocycle is
``block-compatible'' in the following sense:  For the standard
$(r_1,\dots,r_k)$-parabolic of $\GL_r$, so that its Levi $M$ is of the
form
$\GL_{r_1}\times\cdots\times\GL_{r_k}$ which is embedded diagonally into
$\GL_r$, we have
\begin{align}\label{E:compatibility}
&\sigma_r(\begin{pmatrix}g_1&&\\ &\ddots&\\ &&g_k\end{pmatrix},
\begin{pmatrix}g'_1&&\\ &\ddots&\\ &&g'_k\end{pmatrix})\\
=&\prod_{i=1}^k\sigma_{r_i}(g_i,g_i')
\prod_{1\leq i<j\leq k}(\det(g_i), \det(g_j'))_F
\prod_{i\neq j}(\det(g_i), \det(g_j'))_F^c,\notag
 \end{align}
for all $g_i, g_i'\in\GL_{r_i}(F)$. (See \cite[Theorem 11,
\S3]{BLS}. Strictly speaking in \cite{BLS} only the case $c=0$ is
considered but one can derive the above formula using
the bilinearity of the Hilbert symbol.) This 2-cocycle generalizes the
well-known cocycle given by Kubota \cite{Kubota} for the case
$r=2$. Also we should note that if $r=1$, this cocycle is
trivial. Note that $\GLt_r(F)$ is not the $F$-rational points of an
algebraic group, but this notation seems to be standard. 

Let us list some other important properties of the cocycle
$\sigma_r$, which we will use in this paper.
\begin{Prop}\label{P:BLS}
Let $B=TN_B$ be the Borel subgroup of $\GL_r$ where $T$ is the maximal
torus and $N_B$ the unipotent radical. The cocycle $\sigma_r$
satisfies the following properties:
\begin{enumerate}[(1)]
\vspace{.1in}
\item [(0)] $\sigma_r(g,g')\sigma_r(gg', g'')=\sigma_r(g,
  g'g'')\sigma_r(g',g'')$ for $g,g',g''\in\GL_r$.
\vspace{.1in}
\item $\sigma_r(ng, g'n')=\sigma_r(g, g')$ for $g, g'\in\GL_r$ and $n,
  n'\in N_B$, and so in particular $\sigma_r(ng, n')=\sigma_r(n, g'n')=1$.
\vspace{.1in}
\item $\sigma_r(gn, g')=\sigma_r(g, ng')$ for $g, g'\in\GL_r$ and
  $n\in N_B$.
\vspace{.1in}
\item $\sigma(\eta,
  t)=\underset{\substack{\alpha=(i,j)\in\Phi^+\\ \eta\alpha<0}}{\prod}(-t_j,
  t_i)$ for $\eta\in\M$ and $t=\diag(t_1,\dots,t_r)\in T$, where $\Phi^+$ is the set
  of positive roots and each root $\alpha\in\Phi^+$ is identified with a pair
  of integers $(i,j)$ with $1\leq i<j\leq r$ as usual.
\vspace{.1in}
\item $\sigma_r(t,t')=\underset{i<j}{\prod}(t_i, t'_j)(\det(t),\det(t'))^c$ for
  $t=\diag(t_1,\dots,t_r), t'=\diag(t'_1,\dots,t'_r)\in T$.
\vspace{.1in}
\item $\sigma_r(t,\eta)=1$ for $t\in T$ and $\eta\in\M$.
\end{enumerate}
\end{Prop}
\begin{proof}
The first one is simply the definition of 2-cocycle and all the others
are some of the properties of $\sigma_r$ listed in \cite[Theorem 7,
p.153]{BLS}.
\end{proof}

We need to recall how this cocycle is constructed. As mentioned
earlier, Matsumoto constructed $\SLt_{r+1}(F)$. It is shown in
\cite{BLS} that $\SLt_{r+1}(F)$ is defined by a cocycle
$\sigma_{\SL_{r+1}}$ which satisfies the block-compatibility in a much
stronger sense as in \cite[Theorem 7, \S2, p145]{BLS}. (Note that our
$\SL_{r+1}$ corresponds to $\G^\flat$ of \cite{BLS}.) Then the cocycle
$\sigma_r$ is defined by
\[
\sigma_r(g, g')=\sigma_{\SL_{r+1}}(l(g), l(g'))(\det(g), \det(g'))_F (\det(g), \det(g'))^c_F,
\]
where $l$ is the embedding defined by
\begin{equation}\label{E:embedding}
l:\GL_r(F)\rightarrow\SL_{r+1}(F),\quad 
g\mapsto\begin{pmatrix}g&\\ &\det(g)^{-1}\end{pmatrix}.
\end{equation}
See \cite[p.146]{BLS}. (Note the difference between this
embedding and the one in (\ref{E:embedding0}). This is the reason we
have the extra Hilbert symbol in the definition of $\sigma_r$.)

Since we would like to emphasize the cocycle being used, we denote
$\GLt_r(F)$ by $\sigGLt_r(F)$ when the cocycle $\sigma$ is
used. Namely $\sigGLt_r(F)$ is the group whose underlying set is
\[ 
\sigGLt_r(F) =\GL_r(F)\times\mu_n=\{(g,\xi):g\in\GL_r(F), \xi\in\mu_n\},
\] 
and the group law is defined by
\[
(g,\xi)\cdot (g',\xi')=(gg',\sigma_r(g, g')\xi\xi').
\]

To use the block-compatible 2-cocycle of \cite{BLS} has obvious
advantages. In particular, it has been explicitly computed and, of course, it is
block-compatible. Indeed, when we consider purely local problems, we
always assume that the cocycle $\sigma_r$ is used. 

However it does not allow us to construct the global
metaplectic cover $\GLt_r(\A)$. Namely one
cannot define the adelic block-combatible 2-cocycle simply by taking
the product of the local block-combatible 2-cocycles over all the
places. Namely for $g, g'\in\GL_r(\A)$, the product
\[
\prod_v\sigma_{r,v}(g_v,g_v')
\]
is not necessarily finite. This can be already observed for the
case $r=2$. (See \cite[p.125]{F}.)

For this reason, we will use a different 2-cocycle $\tau_r$ which
works nicely with the global metaplectic cover $\GLt_r(\A)$. To construct such
$\tau_r$, first assume $F$ is non-archimedean. It is known that an
open compact subgroup $K$
splits in $\GLt_r(F)$, and moreover if $|n|_F=1$, we have
$K=\GL_r(\OF)$. (See \cite[Proposition 0.1.2]{KP}.) Also for
$k,k'\in K$, a property of the Hilbert symbol gives  $(\det(k),\det(k'))_F=1$.
Hence one has a
continuous map $s_r:\GL_r(F)\rightarrow\mu_n$  such that
$\sigma_r(k,k')s_r(k)s_r(k')=s_r(kk')$ for all $k,k'\in
K$. Then define our 2-cocycle $\tau_r$ by
\begin{equation}\label{E:tau_sigma}
\tau_r(g, g'):=\sigma_r(g, g')\cdot\frac{s_r(g)s_r(g')}{s_r(gg')}
\end{equation}
for $g, g'\in\GL_r(F)$. If $F$ is archimedean, we set
$\tau_r=\sigma_r$. 

The choice of $s_r$ and hence $\tau_r$ is not unique. However when
$|n|_F=1$, there is a canonical choice with
respect to the splitting of $K$ in the following sense: Assume that
$F$ is such that $|n|_F=1$. Then the Hilbert symbol $(-,-)_F$ is
trivial on $\OF^\times\times\OF^\times$, and hence, when restricted to $\GL_r(\OF)\times\GL_r(\OF)$,
the cocycle $\sigma_r$ is the restriction of $\sigma_{\SL_{r+1}}$ to the image of the embedding
$l$. Now it is known that the compact group $\SL_{r+1}(\OF)$ also splits in
$\SLt_{r+1}(F)$, and hence there is a map $\sss_r:\SL_{r+1}(F)\rightarrow\mu_n$ such
that the section $\SL_{r+1}(F)\rightarrow\SLt_{r+1}(F)$ given by $(g,\sss_r(g))$
is a homomorphism on $\SL_{r+1}(\OF)$. (Here we are assuming $\SLt_{r+1}(F)$ is
realized as $\SL_{r+1}(F)\times\mu_n$ as a set and the group structure
is defined by the cocycle $\sigma_{\SL_{r+1}}$.) Moreover $\sss_r|_{\SL_{r+1}(\OF)}$ is
determined up to twists by the elements in $H^1(\SL_{r+1}(\OF),
\mu_n)=\Hom(\SL_{r+1}(\OF), \mu_n)$. But $\Hom(\SL_{r+1}(\OF),
\mu_n)=1$ because $\SL_{r+1}(\OF)$ is a perfect group and $\mu_n$ is
commutative. Hence $\sss_r|_{\SL_{r+1}(\OF)}$ is unique. (See also \cite[p. 43]{KP} for this
matter.)  We choose $s_r$ so that
\begin{equation}\label{E:canonical_section}
s_r|_{\GL_r(\OF)}=\sss_r|_{l(\GL_r(\OF))}.
\end{equation} 
With this choice, we have the commutative diagram
\begin{equation}\label{E:canonical_diagram}
\xymatrix{
\sigGLt_r(\OF)\ar[r]&\SLt_{r+1}(\OF)\\
K\ar[r]\ar[u]^{k\mapsto (k,\;s_r(k))}&\SL_{r+1}(\OF),\ar[u]_{k\mapsto(k,\;\sss_r(k))}
}
\end{equation}
where the top arrow is $(g,\xi)\mapsto (\l(g),\xi)$, the bottom arrow
is $l$, and  all the arrows can be seen to be
homomorphisms. This choice of $s_r$ will be crucial for constructing the
metaplectic tensor product of automorphic representations. Also note
that the left vertical arrow in the above diagram is what is called
the canonical lift in \cite{KP} and denoted by $\kappa^\ast$
there. (Although we do not need this fact in this
paper, if $r=2$ one can show that $\tau_r$ can be chosen to be block
compatible, which is the cocycle used in \cite{F}.)

Using $\tau_r$, we realize $\GLt_r(F)$ as
\[ 
\GLt_r(F)=\GL_r(F)\times\mu_n,
\] 
as a set and the group law is given by
\[ 
(g,\xi)\cdot(g',\xi')=(gg', \tau_r(g,g')\xi\xi').
\] 
Note that we have the exact sequence
\[
\xymatrix{
0\ar[r]&\mu_n\ar[r]&\GLt_r(F)\ar[r]^{p}&\GL_r(F)\ar[r]&
0
}
\]
given by the obvious maps, where we call $p$ the canonical projection.

We define a set theoretic section
\[ 
\kappa:\GL_r(F)\rightarrow\GLt_r(F),\; g\mapsto (g,1).
\] 
Note that $\kappa$ is not a homomorphism. But by our
construction of the
cocycle $\tau_r$, $\kappa|_K$ is a homomorphism if $F$ is
non-archimedean and $K$ is a sufficiently small open compact subgroup,
and moreover if $|n|_F=1$, one has
$K=\GL_r(\OF)$.

Also we define another set theoretic section
\[
\s_r:\GL_r(F)\rightarrow\GLt_r(F),\; g\mapsto (g,s_r(g)^{-1})
\]
where $s_r(g)$ is as above, and then we have the isomorphism
\[
\GLt_r(F)\rightarrow\sigGLt_r(F),\quad (g,\xi)\mapsto (g,s_r(g)\xi), 
\]
which gives rise to the commutative diagram
\[
\xymatrix{
\GLt_r(F)\ar[rr]&&\sigGLt_r(F)\\
&\GL_r(F)\ar[ul]^{\s_r}\ar[ur]_{g\mapsto (g,1)}&
}
\]
of set theoretic maps. Also note
that the elements in the image $\s_r(\GL_r(F))$ ``multiply via
$\sigma_r$'' in the sense that for $g,g'\in\GL_r(F)$, we have
\begin{equation}\label{E:convenient}
(g,s_r(g)^{-1}) (g',s_r(g')^{-1})=(gg', \sigma_r(g,g')s_r(gg')^{-1}).
\end{equation}
 
Let us mention
\begin{Lem}
Assume $F$ is non-archimedean with $|n|_F=1$. We have
\begin{equation}\label{E:kappa_and_s}
\kappa|_{T\cap K}=\s_r|_{T\cap K},\quad \kappa|_{W}=\s_r|_{W},\quad
\kappa|_{N_B\cap K}=\s_r|_{N_B\cap K},
\end{equation}
where $W$ is the Weyl group and $K=\GL_r(\OF)$. 
In particular, this implies $s_r|_{T\cap K}=s_r|_{W}=s_r|_{N_B\cap K}=1$.
\end{Lem}
\begin{proof}
See \cite[Proposition 0.I.3]{KP}.
\end{proof}

\begin{Rmk}
Though we do not need this fact in this paper, it should be noted that $\s_r$
splits the Weyl group $W$ if and only if $(-1,-1)_F=1$. So in
particular it splits $W$ if $|n|_F=1$. See \cite[\S 5]{BLS}. 
\end{Rmk}

If $P$ is a parabolic subgroup of $\GL_r$ whose Levi is
$M=\GL_{r_1}\times\cdots\times\GL_{r_k}$, we often write
\[ 
\Mt(F)=\GLt_{r_1}(F)\timest\cdots\timest\GLt_{r_k}(F)
\] 
for the metaplectic preimage of $M(F)$. Next let
\[ 
\GL_r^{(n)}(F)=\{g\in\GL_r(F):\det g\in F^{\times n}\}, 
\] 
and $\GLtn_r(F)$ its metaplectic preimage. Also we define
\[ 
M^{(n)}(F)=\{(g_1,\dots,g_k)\in M(F): \det g_i\in F^{\times n}\}
\] 
and often denote its preimage by
\[ 
\Mtn(F)=\GLtn_{r_1}(F)\timest\cdots\timest\GLtn_{r_k}(F).
\] 
The group $\Mtn(F)$ is a normal subgroup of finite index. Indeed, we
have the exact sequence
\begin{equation}\label{E:finite_quotient}
1\rightarrow\Mtn(F)\rightarrow\Mt(F)\rightarrow
\underbrace{F^{\times n}\backslash F^\times\times\cdots\times
  F^{\times n}\backslash F^\times}_{\text{$k$ times}}
\rightarrow 1,
\end{equation}
where the third map is given by
$(\diag(g_1,\dots,g_k),\xi)\mapsto(\det(g_1),\dots,\det(g_k))$.
We should mention the explicit isomorphism $F^{\times n}\backslash
F^\times\times\cdots\times F^{\times n}\backslash
F^\times\rightarrow\Mtn(F)\backslash\Mt(F)$ defined as follows: First 
for each $i\in\{1,\dots,k\}$, define a map
$\iota_i:F^\times\rightarrow\GL_{r_i}$ by
\begin{equation}\label{E:iota}
\iota_i(a)=\begin{pmatrix}a&\\ &I_{r_i-1}\end{pmatrix}.
\end{equation}
Then the map given by
\[
(a_1,\dots,a_k)\mapsto(\begin{pmatrix}\iota_1(a_1)&&\\ &\ddots& \\
  &&\iota_k(a_k)\end{pmatrix}, 1)
\]
is a homomorphism. Clearly the map is well-defined and 1-1. Moreover
this is surjective because each element $g_i\in\GL_{r_i}$ is written as
\[
g_i=g_i\iota_i(\det(g_i)^{n-1})\iota_i(\det(g_i)^{1-n})
\]
and $g_i\iota_i(\det(g_i)^{n-1})\in\GL_{r_i}^{(n)}$.
 
The following should be mentioned.
\begin{Lem}\label{L:closed_subgroup_local}
The groups $F^{\times n}, M^{(n)}(F)$ and $\Mtn(F)$ are closed
subgroups of $F^\times, M(F)$ and $\Mt(F)$, respectively.
\end{Lem}
\begin{proof}
It is well-known that $F^{\times n}$ is
closed and of finite index in $F^\times$. Hence the group 
$F^{\times n}\backslash F^\times\times\cdots\times
  F^{\times n}\backslash F^\times$ is discrete, in particular
  Hausdorff. But both $\Mtn(F)\backslash\Mt(F)$ and
  $M^{(n)}(F)\backslash M(F)$ are, as topological groups, isomorphic
  to this Hausdorff space. This completes the proof. 
\end{proof}

\begin{Rmk}\label{R:archimedean1}
If $F=\C$, clearly $\Mtn(F)=\Mt(F)$. If $F=\R$, then necessarily $n=2$
and $\GL_r^{(2)}(\R)$ consists of the elements of positive
determinants, which is usually denoted by $\GL_r^+(\R)$. Accordingly
one may denote $\GLtn_r(\R)$ and $\Mtn(\R)$ by $\GLt_r^+(\R)$ and
${\Mt}^+(\R)$ respectively. Both $\GLt_r^+(\R)$ and $\GLt_r(\R)$ share
the identity component, and hence they have the same Lie algebra. The
same applies to $\Mt^+(\R)$ and $\Mt(\R)$.
\end{Rmk}

\quad

Let us mention the following important fact. Let $Z_{\GL_r}(F)\subseteq\GL_r(F)$ be
the center of $\GL_r(F)$. Then its metaplectic preimage
$\widetilde{Z_{\GL_r}}(F)$ is not the center of $\GLt_r(F)$ in
general. (It might not be even commutative for $n>2$.) The
center, which we denote by $Z_{\GLt_r}(F)$, is
\begin{align}\label{E:center_GLt}
Z_{\GLt_r}(F)&=\{(aI_r, \xi):a^{r-1+2rc}\in F^{\times n},
\xi\in\mu_n\}\\
\notag&=\{(aI_r, \xi):a\in F^{\times \frac{n}{d}}, \xi\in\mu_n\},
\end{align}
where $d=\gcd(r-1+2c, n)$. (The second equality is proven in
\cite[Lemma 1]{GO}.) Note that $Z_{\GLt_r}(F)$ is a closed subgroup.

\quad

Let $\pi$ be an admissible representation of a subgroup
$\widetilde{H}\subseteq \GLt_r(F)$, where $\Ht$ is the metaplectic
preimage of a subgroup $H\subseteq\GL_r(F)$. We say $\pi$ is
``genuine'' if each element $(1,\xi)\in\widetilde{H}$ acts as
multiplication by $\xi$, where we view $\xi$ as an element of
$\C$ in the natural way.


\subsection{\bf The global metaplectic cover $\GLt_r(\A)$}\label{S:group}


In this subsection we consider the global metaplectic group. So we let
$F$ be a number field which contains all the $n^\text{th}$ roots of
unity and $\A$ the ring of adeles. Note that if $n>2$, then $F$ must
be totally complex. We shall define the
$n$-fold metaplectic cover $\GLt_r(\A)$ of $\GL_r(\A)$. (Just like the
local case, we write $\GLt_r(\A)$ even though it is not the adelic
points of an algebraic group.) The construction of $\GLt_r(\A)$ has
been done in various places such as \cite{KP, FK}. 

First define the adelic 2-cocycle $\tau_r$ by
\[
    \tau_r(g, g'):=\prod_v\tau_{r,v}({g}_v, g'_v),
\]
for $g, g'\in\GL_r(\A)$, where $\tau_{r,v}$ is the local
cocycle defined in the previous subsection. By definition of $\tau_{r,v}$, we
have $\tau_{r,v}(g_v, g'_v)=1$ for almost all $v$, and hence the
product is well-defined. 

We define $\GLt_r(\A)$ to be the group whose underlying set is
$\GL_r(\A)\times\mu_n$ and the group structure is defined via $\tau_r$
as in the local case, \ie
\[
    (g, \xi)\cdot(g', \xi')=(gg', \tau_r(g, g')\xi\xi'),
\]
for $g, g'\in\GL_r(\A)$, and $\xi, \xi'\in\mu_n$.  Just as the local case, we have
\[
\xymatrix{
0\ar[r]&\mu_n\ar[r]&\GLt_r(\A)\ar[r]^{p}&\GL_r(\A)\ar[r]&0,
}
\]
where we call $p$ the canonical projection. Define a set theoretic section
$\kappa:\GL_r(\A)\rightarrow\GLt_r(\A)$ by
$g\mapsto(g,1)$. 

It is well-known that $\GL_r(F)$ splits in $\GLt_r(\A)$. However the
splitting is not via $\kappa$. In what follows, we will see that the
splitting is via the product of all the local $\s_r$.

Let us start with the following ``product formula'' of $\sigma_r$.
\begin{Prop}\label{P:product_formula}
For $g, g'\in\GL_r(F)$, we have $\sigma_{r,v}(g, g')=1$ for almost
all $v$, and further
\[
\prod_v\sigma_{r,v}(g, g')=1.
\]
\end{Prop}
\begin{proof}
From the explicit description of the cocycle $\sigma_{r,v}(g, g')$
given at the end of $\S 4$ of \cite{BLS}, one can see that
$\sigma_{r,v}(g, g')$ is written as a product of Hilbert
symbols of the form $(t, t')_{F_v}$ for $t, t'\in F^\times$. This proves
the first part of the proposition. The second part follows from the
product formula for the Hilbert symbol.
\end{proof}

\begin{Prop}
If $g\in\GL_r(F)$, then we have $s_{r,v}(g)=1$ for almost all $v$, where
$s_{r,v}$ is the map $s_{r,v}:\GL(F_v)\rightarrow\mu_n$ defining the local
section $\s_r:\GL(F_v)\rightarrow\GLt_r(F_v)$.
\end{Prop}
\begin{proof}
By the Bruhat decomposition we have $g=bwb'$ for some $b, b'\in B(F)$ and
$w\in W$. Then for each place $v$
\begin{align*}
s_{r,v}(g)
&=s_{r,v}(bwb')\\
&=\sigma_{r,v}(b, wb')s_{r,v}(b)s_{r,v}(wb')/\tau_{r,v}(b,wb')\quad\text{by
  (\ref{E:tau_sigma})}\\
&=\sigma_{r,v}(b, wb')s_{r,v}(b) \sigma_{r,v}(w,
b')s_{r,v}(w)s_{r,v}(b')/\tau_{r,v}(w,b')\tau_{r,v}(b,wb') \quad\text{again by
  (\ref{E:tau_sigma})}.
\end{align*}
By the previous proposition, $\sigma_{r,v}(b, wb')=\sigma_{r,v}(w,
b')=1$ for almost all $v$.
By (\ref{E:kappa_and_s}) we know $s_{r,v}(b)=s_{r,v}(w)=s_{r,v}(b')=1$ for almost
all $v$. Finally by definition of $\tau_{r,v}$,
$\tau_{r,v}(w,b')=\tau_{r,v}(b,wb')=1$ for almost all $v$.
\end{proof}

This proposition implies that the expression 
\[
s_r(g):=\prod_vs_{r,v}(g)
\]
makes sense for all $g\in\GL_r(F)$, and one can define the map
\[
\s_r:\GL_r(F)\rightarrow\GLt_r(\A),\quad g\mapsto (g, s_r(g)^{-1}).
\]
Moreover, this is a homomorphism because of Proposition
\ref{P:product_formula} and (\ref{E:convenient}).

Unfortunately, however, the expression $\prod_vs_{r,v}(g_v)$ does not make
sense for every $g\in\GL_r(\A)$ because one does not know
whether $s_{r,v}(g_v)=1$ for almsot all $v$. Yet, we have
\begin{Prop}\label{P:s_split}
The expression $s_r(g)=\prod_vs_{r,v}(g_v)$ makes sense when $g$ is in
$\GL_r(F)$ or $N_B(\A)$, so $\s_r$ is defined on $\GL_r(F)$ and
$N_B(\A)$. Moreover, $\s_r$ is indeed a homomorphism on  $\GL_r(F)$ and
$N_B(\A)$. Also if $g\in\GL_r(F)$ and
$n\in N_B(\A)$, both $s_r(gn)$ and $s_r(ng)$ make sense and further we
have $\s_r(gn)=\s_r(g)\s_r(n)$ and $\s_r(ng)=\s_r(n)\s_r(g)$.
\end{Prop}
\begin{proof}
We already know $s_r(g)$ is defined and $\s_r$ is a homomorphism on
$\GL_r(F)$. Also $s_r(n)$ is defined thanks
to (\ref{E:kappa_and_s}) and $\s_r$ is a homomorphism on $N_B(\A)$ thanks
to Proposition \ref{P:BLS} (1). Moreover for all places $v$,
we have $\sigma_{r,v}(g_v, n_v)=1$ again by Proposition \ref{P:BLS} (1). Hence
for all $v$,
$s_{r,v}(gn_v)=s_{r,v}(g)s_{r,v}(n_v)/\tau_{r,v}(g,n_v)$. For
almost all $v$, the right hand side is $1$. Hence the global $s_r(gn)$
is defined. Also this equality shows that $\s_r(gn)=\s_r(g)\s_r(n)$. The
same argument works for $ng$.
\end{proof}

If $H\subseteq\GL_r(\A)$ is a subgroup on which $\s_r$ is not only defined but
also a group homomorphism, we write $H^\ast:=\s_r(H)$. In particular we have
\begin{equation}\label{E:star}
\GL_r(F)^\ast:=\s_r(\GL_r(F))\quad\text{and}\quad
N_B(\A)^\ast:=\s_r(N_B(\A)).
\end{equation}

\quad

We define the groups like $\GLtn_r(\A)$, $\Mt(\A)$, $\Mtn(\A)$, etc
completely analogously to the local case. Let us mention
\begin{Lem}\label{L:closed_subgroup_global}
The groups $\A^{\times n}, M^{(n)}(\A)$ and $\Mtn(\A)$ are closed
subgroups of $\A^\times, M(\A)$ and $\Mt(\A)$, respectively.
\end{Lem}
\begin{proof}
That $\A^{\times n}$ and $M^{(n)}(\A)$ are closed follows from the
following lemma together with Lemma
\ref{L:closed_subgroup_local}. Once one knows $M^{(n)}(\A)$ is closed,
one will know
$\Mtn(\A)$ is closed because it is the preimage of the closed
$M^{(n)}(\A)$ under the canonical projection, which is continuous.
\end{proof}

\begin{Lem}\label{L:closed_subgroup_local_global}
Let $G$ be an algebraic group over $F$ and $G(\A)$ its adelic
points. Let $H\subseteq G(\A)$ be a subgroup such that $H$ is written
as $H=\prod'_vH_v$ (algebraically) where for each place $v$, $H_v:=H\cap
G(F_v)$ is a closed subgroup of $G(F_v)$. Then $H$ is closed.
\end{Lem}
\begin{proof}
Let $(x_i)_{i\in I}$ be a net in $H$ that converges in $G(\A)$, where
$I$ is some index set. Let $g=\lim_{i\in I} x_i$. Assume $g\notin
H$. Then there exists a place $w$ such that $g_w\notin H_w$. Since
$H_w$ is closed, the set $U_w:=G(F_w)\backslash H_w$ is open. Then
there exists an open neighborhood $U$ of $g$ of the form $U=\prod_v
U_v$, where $U_v$ is some open neighborhood of $g_v$ and at $v=w$,
$U_v=U_w$. But for any $i\in I$, $x_i\notin U$ because $x_{i,
  w}\notin U_w$, which contradicts the assumption that  $g=\lim_{i\in
  I} x_i$. Hence $g\in H$, which shows $H$ is closed.
\end{proof}

Just like the local case, the preimage $\widetilde{Z_{\GL_r}}(\A)$ of
the center $Z_{\GL_r}(\A)$ of
$\GL_r(\A)$ is in general not the center of
$\GLt_r(\A)$ but the center, which we denote by $Z_{\GLt_r}(\A)$, is 
\begin{align*}
Z_{\GLt_r}(\A)&=\{(aI_r, \xi):a^{r-1+2rc}\in \A^{\times n}, \xi\in\mu_n\}\\
&=\{(aI_r, \xi):a\in \A^{\times \frac{n}{d}}, \xi\in\mu_n\},
\end{align*}
where $d=\gcd(r-1+2c, n)$. The center is a closed subgroup of $\GLt_r(\A)$.

\quad

We can also describe $\GLt_r(\A)$ as a quotient of a restricted direct
product of the groups $\GLt_r(F_v)$ as follows. Consider the
restricted direct product $\prod_v'\GLt_r(F_v)$ with respect to the
groups $\kappa(K_v)=\kappa(\GL_r(\mathcal{O}_{F_v}))$ for all $v$ with
$v\nmid n$ and $v\nmid\infty$. If we denote each element in this
restricted direct product by $\Pi'_v(g_v,\xi_v)$ so that $g_v\in
K_v$ and $\xi_v=1$ for almost all $v$, we have the
surjection
\begin{equation}\label{E:surjection}
    \rho:{\prod_v}'\GLt_r(F_v)\rightarrow\GLt_r(\A),\quad
    \Pi'_v(g_v,\xi_v)\mapsto (\Pi'_vg_v, \Pi_v\xi_v),
\end{equation}
where the product $\Pi_v\xi_v$ is literary the product inside $\mu_n$.
This is a group homomorphism because $\tau_r=\prod_v\tau_{r,v}$ and
the groups $\GLt_r(\A)$ and $\GLt_r(F_v)$ are defined, respectively,
by $\tau_r$ and $\tau_{r,v}$. We have
\[
    {\prod_v}'\GLt_r(F_v)/\ker\rho\cong \GLt_r(\A),
\]
where $\ker\rho$ consists of the elements of the form
$(1,\xi)$ with $\xi\in\prod'_v\mu_n$ and $\Pi_v\xi_v=1$.

\quad

Let $\pi$ be a representation of $\widetilde{H}\subseteq \GLt_r(\A)$
where $\Ht$ is the metaplectic preimage of a subgroup
$H\subseteq\GL_r(\A)$. Just like the local
case, we call $\pi$ genuine if $(1,\xi)\in\widetilde{H}(\A)$ acts as
multiplication by $\xi$ for all $\xi\in\mu_n$. Also we have the notion of automorphic
representation as well as automorphic form on $\GLt_r(\A)$ or
$\Mt(\A)$. In this paper, by an automorphic form, we mean a smooth
automorphic form instead of a $K$-finite one, namely an automorphic
form is $K_f$-finite, $\ZZ$-finite and of uniformly moderate
growth. (See \cite[p.17]{Cogdell}.) Hence if $\pi$ is an
automorphic representation of $\GLt_r(\A)$ (or $\Mt(\A)$), the
full group $\GLt_r(\A)$ (or $\Mt(\A)$) acts on $\pi$. An automorphic form $f$ on
$\GLt_r(\A)$ (or $\Mt(\A)$) is said to be genuine if $f(g,\xi)=\xi
f(g,1)$ for all $(g,\xi)\in\GLt_r(\A)$ (or $\Mt(\A)$). In particular
every automorphic form in the space of a genuine automorphic
representation is genuine.

\quad

Suppose we are given a collection of irreducible admissible
representations $\pi_v$ of $\GLt_r(F_v)$ such that $\pi_v$ is
$\kappa(K_v)$-spherical for almost all $v$. Then we can form an
irreducible admissible representation of $\prod_v'\GLt_r(F_v)$ by
taking a restricted tensor product $\otimes_v'\pi_v$ as usual. Suppose
further that $\ker\rho$ acts trivially on $\otimes_v'\pi_v$, which is
always the case if each $\pi_v$ is genuine. Then it
descends to an irreducible admissible representation of $\GLt_r(\A)$,
which we denote by $\otimest'_v\pi_v$, and call it the ``metaplectic
restricted tensor product''. Let us emphasize that the space for
$\otimest'_v\pi_v$ is the same as that for
$\otimes_v'\pi_v$. Conversely, if $\pi$ is an irreducible admissible
representation of $\GLt_r(\A)$, it is written as
$\otimest'_v\pi_v$ where $\pi_v$ is an irreducible admissible
representation of $\GLt_r(F_v)$, and for almost all $v$, $\pi_v$ is
$\kappa(K_v)$-spherical. (To see it, view $\pi$ as a representation of
the restricted product $\prod_v'\GLt_r(F_v)$ by pulling it back by
$\rho$ as in (\ref{E:surjection}) and apply the usual
tensor product theorem for the restricted direct product. This gives
the restricted tensor product
$\otimes_v'\pi_v$, where each $\pi_v$ is genuine, and hence it
descends to $\otimest_v'\pi_v$.) 

\quad

Finally in this section, let us mention that we define
\begin{equation}\label{L:Hasse}
\GL_r^{(n)}(F):=\GL_r(F)\cap\GL_r^{(n)}(\A),
\end{equation}
namely $\GL_r^{(n)}(F)=\{g\in\GL_r(F):\det g\in \A^{\times n}\}$. But
since $F$ contains $\mu_n$, one can easily show that
\[
\GL_r^{(n)}(F)=\{g\in\GL_r(F):\det g\in F^{\times n}\}.
\]
(See, for example, \cite[Chap. 9, Theorem 1]{AT}. Also for $n=2$, this
is a consequence of the Hasse-Minkowski theorem.) Similarly we define
\[
M^{(n)}(F)=M(F)\cap M^{(n)}(\A).
\]

\quad


\section{\bf The metaplectic cover $\Mt$ of the Levi $M$}\label{S:Levi}


Both locally and globally, one cannot show the cocycle
$\tau_r$ has the block-compatibility as in (\ref{E:compatibility})
(except when $r=2$). Yet, in order to define the metaplectic tensor
product, it seems to be necessary to have the block-compatibility of
the cocycle. To get round it, we will introduce another cocycle $\tau_M$, but this
time it is a cocycle only on the Levi $M$, and will show that $\tau_M$
is cohomologous to the restriction $\tau_r|_{M\times M}$ of $\tau_r$
to $M\times M$ both for the local and global cases.


\subsection{\bf The cocycle $\tau_M$}


 In this subsection, we assume that all the groups are over $F$ if $F$
 is local and over $\A$ if $F$ is global, and suppress it
 from our notation.

We define the cocycle
\[
\tau_M: M\times M\rightarrow\mu_n,
\]
by
\[
\tau_M(\begin{pmatrix}g_1&&\\ &\ddots&\\ &&g_k\end{pmatrix},
\begin{pmatrix}g'_1&&\\ &\ddots&\\ &&g'_k\end{pmatrix})
=\prod_{i=1}^k\tau_{r_i}(g_i,g_i')\prod_{1\leq i<j\leq k}(\det(g_i),
\det(g_j'))
\prod_{i\neq j}(\det(g_i), \det(g_j'))^c,
 \]
where $(-,-)$ is the local or global Hilbert symbol.
Note that the definition makes sense both locally and
globally. Moreover the global $\tau_M$ is the product of the local
ones.

We define the group $\cMt$ to be 
\[
\cMt=M\times\mu_n
\]
as a set and the group structure is
given by $\tau_M$. The superscript $^c$ is for
``compatible''. One advantage to work with $\cMt$ is that each
$\GLt_{r_i}$ embeds into $\cMt$ via the natural map
\[
(g_i,\xi)\mapsto(\begin{pmatrix}I_{r_1+\cdots+r_{i-1}}&&\\ &g_i&\\
    &&I_{r_{i+1}+\cdots+r_k}\end{pmatrix}, \xi).
\]
Indeed, the cocycle $\tau_M$ is so chosen that we have this
embedding.

Also recall our notation
\[
M^{(n)}=\GL_{r_1}^{(n)}\times\cdots\times\GL_{r_k}^{(n)},
\]
and
\[
\Mtn=\GLtn_{r_1}\timest\cdots\timest\GLtn_{r_k}.
\]
We define $\cMtn$ analogously to $\cMt$, namely the group structure of
$\cMtn$ is defined via the cocycle $\tau_M$. Of course, $\cMtn$ is a
subgroup of $\cMt$. Note that each
$\GLtn_{r_i}$ naturally embeds into $\cMtn$ as above.

\begin{Lem}
The subgroups $\GLtn_{r_i}$ and $\GLtn_{r_j}$ in $\cMtn$ commute
pointwise for $i\neq j$.
\end{Lem}
\begin{proof}
Locally or globally, it suffices to show
$\tau_M(g_i,g_j)=\tau_M(g_j,g_i)$ for $g_i\in\GL^{(n)}_{r_i}$ and
$g_j\in\GL^{(n)}_{r_j}$. But the block-compatibility of the 2-cocycle
$\tau_M$, we have $\tau_M(g_i,g_j)=\tau_{r_i}(g_i,
I_{r_j})\tau_{r_j}(I_{r_j}, g_j)=1$, and similarly have $\tau_M(g_j,g_i)=1$.
\end{proof}

\begin{Lem}
There is a surjection
\[
\GLtn_{r_1}\times\cdots\times\GLtn_{r_k}\rightarrow\; \cMtn
\]
given by the map
\[
((g_1,\xi_1),\dots,(g_k,\xi_k))\mapsto
(\begin{pmatrix}g_1&&\\ &\ddots&\\ &&g_k\end{pmatrix},
\xi_1\cdots\xi_k),
\]
whose kernel is
\[
\mathcal{K}_P:=\{((1,\xi_1),\dots,(1,\xi_k)):\xi_1\cdots\xi_k=1\},
\]
so that $\cMtn\cong
\GLtn_{r_1}\times\cdots\times\GLtn_{r_k}/\mathcal{K}_P$.
\end{Lem}
\begin{proof}
The block-compatibility of $\tau_M$
guarantees that the map is indeed a group homomorphism. The
description of the kernel is immediate.
\end{proof}


\subsection{\bf The relation between $\tau_M$ and $\tau_r$}


Note that for the group $\Mt$ (instead of $\cMt$), the group structure
is defined by the restriction of $\tau_r$ to $M\times M$, and hence
each $\GLt_{r_i}$ might not embed into $\GLt_r$ in
the natural way because of the possible failure of the
block-compatibility of $\tau_r$ unless $r=2$.
To make explicit the relation between $\cMt$ and $\Mt$, the
discrepancy between $\tau_M$ and 
$\tau_r|_{M\times M}$ (which we denote simply by $\tau_r$) has to be
clarified.

\quad

\noindent{\bf Local case:}

Assume $F$ is local. Then we have
\begin{align*}
&\tau_M(\begin{pmatrix}g_1&&\\ &\ddots&\\ &&g_k\end{pmatrix},
\begin{pmatrix}g'_1&&\\ &\ddots&\\ &&g'_k\end{pmatrix})\\
=&\sigma_r(\begin{pmatrix}g_1&&\\ &\ddots&\\ &&g_k\end{pmatrix},
\begin{pmatrix}g'_1&&\\ &\ddots&\\ &&g'_k\end{pmatrix})
\prod_{i=1}^k\frac{s_{r_i}(g_i)s_{r_i}(g_i')}{s_{r_i}(g_ig_i')},
\end{align*}
so $\tau_M$ and $\sigma_r|_{M\times M}$ are cohomologous via the function
$\prod_{i=1}^ks_{r_i}$. Here recall from Section \ref{S:group} that the map
$s_{r_i}:\GL_{r_i}\rightarrow\mu_n$ relates $\tau_{r_i}$ with
$\sigma_{r_i}$ by
\[
\sigma_{r_i}(g_i,g_i')=\tau_{r_i}(g_i,g_i')\cdot\frac{s_{r_i}(g_i,g_i')}{s_{r_i}(g_i)s_{r_i}(g_i')},
\]
for $g_i,g_i'\in\GL_{r_i}$. Moreover if $|n|_F=1$, $s_{r_i}$ is chosen to
be ``canonical'' in the sense that (\ref{E:canonical_section}) is satisfied.

The block-compatibility of $\sigma_r$ implies
\[
\tau_r(m, m')\cdot\frac{s_r(mm')}{s_r(m)s_r(m')}
=\sigma_r(m,m')
=\tau_M(m,m')\cdot\prod_{i=1}^k\frac{s_{r_i}(g_ig_i')}{s_{r_i}(g_i)s_{r_i}(g_i')},
\]
for $m=\begin{pmatrix}g_1&&\\ &\ddots&\\ &&g_k\end{pmatrix}$ and
$m'=\begin{pmatrix}g_1'&&\\ &\ddots&\\ &&g_k'\end{pmatrix}$. Hence if
we define $\hat{s}_M:M\rightarrow\mu_n$ by
\begin{equation}\label{E:s_hat_M}
\hat{s}_M(m)=\frac{\prod_{i=1}^ks_{r_i}(g_i)}{s_r(m)}, 
\end{equation}
we have
\begin{equation}\label{E:s_hat}
\tau_M(m, m')=\tau_r(m, m')\cdot
\frac{\hat{s}_M(m)\hat{s}_M(m')}{\hat{s}_M(mm')},
\end{equation}
namely $\tau_r$ and $\tau_M$ are cohomologous via $\hat{s}_M$.
Therefore we have the isomorphism
\[
\alpha_M:\cMt\rightarrow\Mt,\quad (m,\xi)\mapsto (m, \hat{s}_M(m)\xi).
\]

The following lemma will be crucial later for showing that the global
$\tau_M$ is also cohomologous to $\tau_r|_{M(\A)\times M(\A)}$.

\begin{Lem}\label{L:s_hat_M}
Assume $F$ is such that
$|n|_F=1$. Then for all $k\in M(\OF)$, we have $\hat{s}_M(k)=1$.
\end{Lem}
\begin{proof}
First note that if $k,k'\in M(\OF)$, then $\tau_r(k,k')=\tau_M(k,
k')=1$ and so by (\ref{E:s_hat}) we have
\[
\hat{s}_M(kk')=\hat{s}_M(k)\hat{s}_M(k'),
\]
\ie $\hat{s}_M$ is a homomorphism on $M_M(\OF)$. Hence it suffices to
prove the lemma only for the elements $k\in M(\OF)$ of the form
\[
k=\begin{pmatrix} I_{r_1+\cdots+r_{i-1}}&&\\ &k_i&\\ && I_{r_{i+1}+\cdots+r_k}\end{pmatrix}
\]
where $k_i\in\GL_{r_i}$ is in the $i^{\text{th}}$ place on the
diagonal. Namely we need to prove
\[
\frac{s_{r_i}(k_i)}{s_r(k)}=1.
\]

In what follows, we will show that this follows from the
``canonicality'' of $s_r$ and $s_{r_i}$, and the fact that the
cocycle for $\SL_{r+1}$ is block-compatible in a very strong sense as in
\cite[Lemma 5, Theorem 7 \S 2, p.145]{BLS}. Recall
from (\ref{E:canonical_section}) that $s_r$ has been chosen to satisfy
$s_r|_{\GL_r(\OF)}={\sss_r}|_{l(\GL_r(\OF))}$, where $\sss_r$ is the map on
$\SL_{r+1}(F)$ that makes the diagram (\ref{E:canonical_diagram})
commute. Similarly for $s_{r_i}$ with $r$ replaced by $r_i$. Let us
write 
\[
l_i:\GL_{r_i}(F)\rightarrow \SL_{r_i+1}(F),\quad
g_i\mapsto\begin{pmatrix}g_i&\\ &\det(g_i)^{-1}\end{pmatrix}
\]
for the embedding that is used to define the cocycle $\sigma_{r_i}$. Define the embedding
\[
F:\SL_{r_i+1}(F)\rightarrow\SL_{r+1}(F),\quad
\begin{pmatrix} A&b\\ c&d\end{pmatrix}\mapsto 
\begin{pmatrix} I_{r_1+\cdots+r_{i-1}}&&&\\
& A&&b\\ &&I_{r_{i+1}+\cdots+r_{k}}&\\
&c&&d\end{pmatrix},
\]
where $A$ is a $r_i\times r_i$-block and accordingly $b$ is $r_i\times
1$, $c$ is $1\times r_i$ and d is $1\times 1$. Note that this
embedding is chosen so that we have 
\begin{equation}\label{E:F_and_l}
F(l_i(k_i))=l(k).
\end{equation}

By the block compatibility of $\sigma_{\SL_{r+1}}$
we have 
\[
\sigma_{\SL_{r+1}}|_{F(\SL_{r_i+1})\times
  F(\SL_{r_i+1})}=\sigma_{\SL_{r_i+1}}.
\]
This is nothing but \cite[Lemma 5, \S2]{BLS}. (The reader has to
be careful in that the image $F(\SL_{r_i+1})$ is not a standard
subgroup in the sense defined in \cite[p.143]{BLS} if one chooses the
set $\Delta$ of simple roots of $\SL_{r+1}$ in the usual way. One can, however,
choose $\Delta$ differently so that $F(\SL_{r_i+1})$ is indeed a
standard subgroup. And all the results of \cite[\S 2]{BLS} are totally
independent of the choice of $\Delta$.)
This implies the map $(g_i, \xi)\mapsto (F(g_i),\xi)$ for $(g_i,\xi)\in\SLt_{r_i+1}$ is a
homomorphism. Hence the canonical section
$\SL_{r+1}(\OF)\rightarrow \SLt_{r+1}(F)$, which is given by $g\mapsto (g,
\sss_r(g))$, restricts to the canonical section
$\SL_{r_i+1}(\OF)\rightarrow \SLt_{r_i+1}(F)$, which is given by
$g_i\mapsto (g_i,\sss_{r_i}(g_i))$. Namely we have the commutative
diagram
\[
\xymatrix{
\SLt_{r_i+1}(\OF)\ar[rrr]^{(g,\;\xi)\rightarrow( F(g),\; \xi)}&&&\SLt_{r+1}(\OF)\\
\SL_{r_i+1}(\OF)\ar[rrr]^F\ar[u]^{g_i\mapsto(g_i,\;
  \sss_{r_i}(g_i))}&&&\SL_{r+1}(\OF)\ar[u]_{g\mapsto(g,\;\sss_r(g))},
}
\]
where all the maps are homomorphisms.
In particular, we have 
\begin{equation}\label{E:sss}
\sss_r(F(g_i))=\sss_{r_i}(g_i),
\end{equation}
for all $g_i\in\SL_{r_i+1}(\OF)$. Thus 
\begin{align*}
s_r(k)&=\sss_r(l(k))\quad\text{by (\ref{E:canonical_section})}\\
&=\sss_r(F(l_i(k_i)))\quad\text{by (\ref{E:F_and_l})}\\
&=\sss_{r_i}(l_i(k_i))\quad\text{by (\ref{E:sss})}\\
&=s_{r_i}(k_i)\quad\text{by (\ref{E:canonical_section}) with $r$
  replaced by $r_i$}.
\end{align*}
The lemma has been proven.
\end{proof}

\quad

\noindent{\bf Global case:}

Assume $F$ is a number field. We define $\hat{s}_M:
M(\A)\rightarrow\mu_n$ by
\[
\hat{s}_M(\prod_vm_v):=\prod_v{\hat{s}_{M_v}}(m_v)
\]
for $\prod_vm_v\in M(\A)$. The product is finite thanks to Lemma
\ref{L:s_hat_M}. Since both of the cocycles $\tau_r$ and $\tau_M$ are
the products of the corresponding local ones, one can
see that the relation (\ref{E:s_hat}) holds globally as well.

Thus analogously to the local case, we have
the isomorphism
\[
\alpha_M:\; \cMt(\A)\rightarrow\Mt(\A),\quad
(m,\xi)\mapsto (m, \hat{s}_M(m)\xi).
\]

\quad

\begin{Lem}\label{L:splitting_cMPt}
The splitting of $M(F)$ into $\cMt(\A)$ is given by
\[
\s_M:M(F)\rightarrow\;\cMt(\A),\quad
\begin{pmatrix}g_1&&\\ &\ddots&\\ &&g_k\end{pmatrix}\mapsto 
(\begin{pmatrix}g_1&&\\ &\ddots&\\ &&g_k\end{pmatrix},\;
\prod_{i=1}^k s_i(g_i)^{-1}).
\]
\end{Lem}
\begin{proof}
For each $i$ the splitting
$\s_{r_i}:\GL_{r_i}(F)\rightarrow\GLt_{r_i}(\A)$ is given by
$g_i\mapsto(g_i,\;s_{r_i}(g_i)^{-1})$, where $\GLt_{r_i}(\A)$ is
defined via the cocycle $\tau_{r_i}$. The lemma follows by the
block-compatibility of $\tau_M$ and the product formula for the
Hilbert symbol.
\end{proof}

Just like the case of $\GLt_r(\A)$, the section
$\s_M$ as in this lemma cannot be defined on all of $M(\A)$ even
set theoretically because the expression $\prod_is_{r_i}(g_i)$ does
not make sense to all $\diag(g_1,\dots,g_k)\in M(\A)$. So we only have
a partial set theoretic section
\[
\s_M:M(\A)\rightarrow\cMt(\A).
\]
But analogously to Proposition \ref{P:s_split}, we have
\begin{Prop}\label{P:s_split_M}
The partial section $\s_M$ is defined on both
$M(F)$ and $N_M(\A)$, where $N_M(\A)$ is the unipotent radical of the
Borel subgroup of $M$, and moreover it gives rise to a group
homomorphism on each of these subgroups. Also for $m\in M(F)$ and $n\in
N_M(\A)$, both $\s_M(mn)$ and $\s_M(nm)$ are defined and further
$\s_M(mn)=\s_M(m)\s_M(n)$  and $\s_M(nm)=\s_M(n)\s_M(m)$.
\end{Prop}
\begin{proof}
This follows from Proposition \ref{P:s_split} applied to each
$\GLt_{r_i}(\A)$ together with the block-compatibility of the cocycle
$\tau_M$. (Note that one also needs to use the fact that for all
$g, g'$ in the subgroup generated by $M(F)$ and $N_M(\A)$, we have
$(\det(g), \det(g'))_\A=1$.)
\end{proof}

This splitting is related to the splitting
$\s_r:\GL_r(F)\rightarrow\GL_r(\A)$ by
\begin{Prop}\label{P:diagram}
We have the following commutative diagram:
\[
\xymatrix{\cMt(\A)\; \ar@{^{(}->}[r]^{\alpha_M}&\GLt_r(\A)\\
M(F)\;\ar@{^{(}->}[r]\; \ar[u]^{\s_M}&\GL_r(F)\ar[u]_{\s_r}.
}
\]
\end{Prop}
\begin{proof}
For $m=\begin{pmatrix}g_1&&\\ &\ddots&\\ &&g_k\end{pmatrix}
\in M(F)$, we have
\[
\alpha_M(\s_M(m))=\alpha_M(m, \prod_{i=1}^ks_{r_i}(g_i)^{-1})
=(m, \hat{s}_M(m)\prod_{i=1}^ks_{r_i}(g_i)^{-1})
=(m, s_r(m)^{-1})=\s_r(m),
\]
where for the elements in $M(F)$, all of $s_{r_i}$ and $s_r$
are defined globally, and the second equality follows from the
definition of $\hat{s}_M$ as in (\ref{E:s_hat_M}).
\end{proof}

This proposition implies
\begin{Cor} \label{C:diagram}
Assume $\pi$ is an automorphic subrepresentation of
$\cMt(\A)$. The representation of
$\Mt(\A)$ defined by $\pi\circ\alpha_M^{-1}$ is
also automorphic.
\end{Cor}
\begin{proof}
If $\pi$ is realized in a space $V$ of automorphic forms on
$\cMt(\A)$, then $\pi\circ\alpha_M^{-1}$ is realized in the
space of functions of the form $f\circ\alpha_M^{-1}$ for $f\in
V$. The automorphy follows from the commutativity of the diagram
in the above lemma.
\end{proof}

\quad

The following remark should be kept in mind for the rest of the paper.
\begin{Rmk}
The results of this subsection essentially show that we may identify
$\cMt$ (locally or globally) with $\Mt$. We may even ``pretend''
that the cocycle $\tau_r$ has the block-compatibility
property. We need to make the distinction between $\cMt$ and $\Mt$
only when we would like to view the group $\Mt$ as a subgroup of
$\GLt_r$. For most part of this paper, however, we will not have to view $\Mt$
as a subgroup of $\GLt_r$. Hence we suppress the superscript $^c$ from
the notation and always denote $\cMt$ simply by $\Mt$, when there is
no danger of confusion. Accordingly,
we denote the partial section $\s_M$ simply by $\s$.
\end{Rmk}

\quad


\subsection{\bf The center $Z_{\Mt}$ of $\Mt$}


In this subsection $F$ is either local or global, and accordingly we
let $R=F$ or $\A$ as in the notation section. And all the groups are over
$R$. 

For any
group $H$ (metaplectic or not), we denote its center by $Z_H$. In
particular for each group $\widetilde{H}\subseteq\GLt_r$, we let
\[
Z_{\widetilde{H}}=\text{center of $\widetilde{H}$}.
\]

For the Levi part
$M=\GL_{r_1}\times\cdots\times\GL_{r_2}\subseteq\GL_r$, 
we of course have
\[
Z_{M}=\{\begin{pmatrix}a_1I_{r_1}&&\\ &\ddots&\\
  &&a_{k}I_{r_k}\end{pmatrix}:a_i\in R^\times\}.
\]
But for the center $Z_{\Mt}$ of $\Mt$, we have
\[
Z_{\Mt}\subsetneq \widetilde{Z_{M}},
\]
in general, and indeed $\widetilde{Z_{M}}$ might not be even commutative.

In what follows, we will describe $Z_{\Mt}$ in detail. For this
purpose, let us start with
\begin{Lem}\label{L:to_compute_center}
Assume $F$ is local. Then for each $g\in\GL_r(F)$ and $a\in F^\times$,
we have
\[
\sigma_r(g, aI_r)\sigma_r(aI_r, g)^{-1}=(\det(g), a^{r-1+2cr}).
\]
\end{Lem}
\begin{proof}
First let us note that if we write $\sigma_r=\sigma_r^{(c)}$ to
emphasize the parameter $c$, then 
\[
\sigma_r^{(c)}(g, aI_r)\sigma_r^{(c)}(aI_r, g)^{-1}
=\sigma_r^{(0)}(g, aI_r)\sigma_r^{(0)}(aI_r, g)^{-1}(\det(g), a^r)^{2c}
\]
because $(a^r,\det(g))^{-1}=(\det(g), a^r)$. Hence it suffices to
show the lemma for the case $c=0$.

But this can be done by using the recipe provided by
\cite{BLS}. Namely let $g=nt\eta n'$ for $n, n'\in N_B$, $t\in T$ and
$\eta\in\M$. Then
\begin{align*}
\sigma_r(g, aI_r)&=\sigma_r(nt\eta n', aI_r)\\
&=\sigma_r(t\eta, n' aI_r)\quad\text{by Proposition \ref{P:BLS} (1)
  and (2)}\\
&=\sigma_r(t\eta, aI_r)\quad\text{by $n'aI_r=aI_rn'$ and Proposition
  \ref{P:BLS} (1)}\\
&=\sigma_r(t, \eta aI_r)\sigma_r(\eta, aI_r)\sigma_r(t,
\eta)^{-1}\quad\text{by Proposition \ref{P:BLS} (0)}\\
&=\sigma_r(t, aI_r\eta)\sigma_r(\eta, aI_r)\quad\text{by Proposition
  \ref{P:BLS} (5)}\\
&=\sigma_r(taI_r, \eta)\sigma_r(t, aI_r)\sigma_r(aI_r, \eta)^{-1}
\sigma_r(\eta, aI_r)\quad\text{by Proposition
  \ref{P:BLS} (0)}\\
&=\sigma_r(t, aI_r)\sigma_r(\eta, aI_r)\quad\text{by Proposition
  \ref{P:BLS} (5)}.
\end{align*}
Now by Proposition \ref{P:BLS} (3), $\sigma(\eta, aI_r)$ is a product
of $(-a, a)$'s, which is $1$. Hence by using Proposition \ref{P:BLS}
(4), we have
\[
\sigma_r(g, aI_r)=\sigma_r(t, aI_r)
=\prod_{i=1}^{r}(t_i, a)^{r-i}.
\]
By an analogous computation, one can see
\[
\sigma_r(aI_r, g)=\sigma_r(aI_r, t)=\prod_{i=1}^{r}(a, t_i)^{i-1}.
\]
Using $(a, t_i)^{-1}=(t_i, a)$, one can see
\[
\sigma_r(g, aI_r) \sigma_r(aI_r, g)^{-1}=\prod_{i=1}^r(t_i, a)^{r-1}.
\]
But this is equal to $(\det(g), a^{r-1})$ because
$\det(g)=\prod_{i=1}^rt_i$. 
\end{proof}

Note that this lemma immediately implies that the center $Z_{\GLt_r}$ of $\GLt_r$
is indeed as in (\ref{E:center_GLt}), though a different proof is
provided in \cite{KP}. 

Also with this lemma, we can prove
\begin{Prop}
Both locally and globally, the center $Z_{\Mt}$ is described as
\[
Z_{\Mt}=\{\begin{pmatrix}a_1I_{r_1}&&\\ &\ddots&\\
  &&a_kI_{r_k}\end{pmatrix}:
a_i^{r-1+2cr}\in R^{\times n}\;\text{ and }\;a_1\equiv\cdots\equiv
a_r\mod{R^{\times n}}\}.
\]
\end{Prop}
\begin{proof}
First assume $F$ is local. Let $m=\diag(g_1,\dots,g_k)\in M$ and
$a=\diag(a_1I_{r_1},\dots,a_kI_{r_k})$. It suffices to show
$\sigma_r(m,a)\sigma_r(a, m)^{-1}=1$ if and only if all $a_i$ are as in the
proposition. But
\begin{align*}
&\sigma_r(m,a) \sigma_r(a, m)^{-1}\\
=&\prod_{i=1}^r\sigma_{r_i}(g_i, a_iI_{r_i})\sigma_{r_i}(a_iI_{r_i},
g_i)^{-1}\prod_{1\leq i<j\leq r}(\det(g_i), a_j^{r_j})\prod_{i\neq
  j}(\det(g_i), a_j^{r_j})^c\\
&\qquad\qquad\qquad\qquad\qquad\qquad\qquad
\cdot\prod_{1\leq i<j\leq r}(a_i^{r_i},\det(g_j))^{-1}\prod_{i\neq
  j}(a_i^{r_i},\det(g_j))^{-c}\\
=&\prod_{i=1}^r\sigma_{r_i}(g_i, a_iI_{r_i})\sigma_{r_i}(a_iI_{r_i},
g_i)^{-1}\prod_{i\neq
  j}(\det(g_i), a_j^{r_j})^{1+2c}\\
=&\prod_{i=1}^r (\det(g_i), a_i^{r_i-1+2cr_i})\prod_{i\neq
  j}(\det(g_i), a_j^{r_j+2cr_j})\\
=&\prod_{i=1}^r (\det(g_i),\, a_i^{-1}\prod_{j=1}^ra_j^{r_j+2cr_j}),
\end{align*}
where for the third  equality we used the above lemma with $r$ replaced
by $r_i$. 

Now assume $a$ is such that $(a,1)\in Z_{\Mt}$. Then the above product
must be 1 for any $m$. In particular, choose $m$ so that $g_j=1$ for all
$i\neq j$. Then we must have $(\det(g_i),\,
a_i^{-1}\prod_{j=1}^ra_j^{r_j+2cr_j})=1$ for all $g_i\in
\GL_{r_i}$. This implies 
\[
a_i^{-1}\prod_{j=1}^ra_j^{r_j+2cr_j}\in F^{\times n}
\]
for all $i$. Since this holds for all $i$, one can see
$a_i^{-1}a_j\in
F^{\times n}$ for all $i\neq j$, which implies $\;a_1\equiv\cdots\equiv
a_r\mod{F^{\times n}}$. But if $\;a_1\equiv\cdots\equiv
a_r\mod{F^{\times n}}$, then
\begin{align*}
\prod_{i=1}^r (\det(g_i),\, a_i^{-1}\prod_{j=1}^ra_j^{r_j+2cr_j})
=&\prod_{i=1}^r(\det(g_i), a_i^{-1}\prod_{j=1}^ra_i^{r_j+2cr_j})\\
=&\prod_{i=1}^r(\det(g_i), a_i^{r-1+2cr}).
\end{align*}
This must be equal to 1 for any choice of $g_i$, which gives
$a_i^{r-1+2cr}\in F^{\times n}$.

Conversely if $a$ is of the form as in the proposition, one can see
that $\sigma_r(m,a) \sigma_r(a, m)^{-1}=\prod_{i=1}^r (\det(g_i),\,
a_i^{-1}\prod_{j=1}^ra_j^{r_j+2cr_j})=1$ for any $m$.

The global case follows from the local one because locally by using
(\ref{E:tau_sigma}) and $am=ma$, one can see 
$\sigma_r(m,a) \sigma_r(a, m)^{-1}=1$ if and only if $\tau_r(m,a)
\tau_r(a, m)^{-1}=1$, and the global $\tau_r$ is the product of local ones.
\end{proof}

Lemma \ref{L:to_compute_center} also implies
\begin{Lem}\label{L:center_GLtt}
Both locally and globally, $\widetilde{Z_{\GL_r}}$ commutes with
$\GLtn_r$ pointwise.
\end{Lem}
\begin{proof}
The local case is an immediate corollary of Lemma
\ref{L:to_compute_center} because if
$g\in\GL_r^{(n)}$ the lemma implies $\sigma_r(g, aI_r)=\sigma_r(aI_r,
g)$. Hence by (\ref{E:tau_sigma}), 
locally $\tau_r(g, aI_r)=\tau_r(aI_r, g)$ for all $g\in\GL_r^{(n)}$
and $a\in F^\times$. Since the global $\tau_r$
is the product of the local ones, the global case also follows.
\end{proof}
Let us mention that in particular, if $n=2$ and $r=\text{even}$, then
$\widetilde{Z_{\GL_r}}\subseteq\GLtn_r$
and $\widetilde{Z_{\GL_r}}$ is the center of $\GLtn_r$. This fact is used
crucially in \cite{Takeda1}.

\quad

It should be mentioned that this description of the center $Z_{\Mt}$ easily
implies 
\begin{equation}\label{E:center=center}
Z_{\GLt_r}\Mtn=Z_{\Mt}\Mtn.
\end{equation}

Also we have
\begin{Prop}\label{P:Z_M_commute_Mtn}
Both locally and globally, the groups $\widetilde{Z_M}$ and $\Mtn$ commute pointwise, which gives
\begin{equation}\label{E:center_Mtn}
Z_{\Mtn}=\widetilde{Z_{M}}\cap {\Mtn},
\end{equation}
and hence
\begin{equation}\label{E:center_Mtn2}
Z_{\GLt_r}Z_{\Mtn}=Z_{\GLt_r}(\widetilde{Z_{M}}\cap
{\Mtn})=\widetilde{Z_{M}}\cap( Z_{\GLt_r} \Mtn).
\end{equation}
\end{Prop}
\begin{proof}
By the block compatibility of the cocycle $\tau_M$, one can see
that an element of the form $(\begin{pmatrix}a_1I_{r_1}&&\\ &\ddots&\\
  &&a_{k}I_{r_k}\end{pmatrix}, \xi)$ commutes with all the elements in
$\Mtn$ if and only if each $(a_iI_{r_i},\xi)$ commutes with all the
elements in $\GLtn_{r_i}$. But this is always the case by the above
lemme (with $r$ replaced by $r_i$). This proves the proposition.
\end{proof}

\quad

If $F$ is global, we define
\[
Z_{\Mt}(F)=Z_{\Mt}(\A)\cap\s(M(F)),
\]
where recall that $\s:M(F)\rightarrow\Mt(\A)$ is the section that
splits $M(F)$. Similarly we define groups like $Z_{\GLt_r}(F),
\Mtn(F)$, etc. Namely in general for any subgroup $\Ht\subseteq
\Mt(\A)$, we define the ``$F$-rational points'' $\Ht(F)$ of $\Ht$ by
\begin{equation}\label{E:H(F)}
\Ht(F):=\Ht\cap\s(M(F)).
\end{equation}

\quad


\subsection{\bf The abelian subgroup $A_{\Mt}$}\label{S:abelian}


Again in this subsection, $F$ is local or global, and $R=F$ or $\A$.
As we mentioned above, the preimage $\widetilde{Z_{M}}$ of the center
$Z_M$ of the Levi $M$ might not be even commutative. For later
purposes, we let $A_{\Mt}$ be a closed abelian subgroup of $\widetilde{Z_M}$
containing the center $Z_{\GLt_r}$. Namely
$A_{\Mt}$ is a closed abelian subgroup such that
\[
Z_{\GLt_r}\subseteq A_{\Mt}\subseteq \widetilde{Z_{M}}.
\]
We let
\[
A_M:=p(A_{\Mt}),
\]
where $p$ is the canonical projection.  If $F$ is global, we always
assume $A_{\Mt}(\A)$ is chosen
compatibly with the local $A_{\Mt}(F_v)$ in the sense that we have
\[
A_{M}(\A)={\prod_v}'A_{M}(F_v).
\]
Note that if $A_{M}(F_v)$ (hence $A_{\Mt}(F_v)$) is closed, then
$A_{M}(\A)$ (hence $A_{\Mt}(\A)$) is closed by Lemma
\ref{L:closed_subgroup_local_global}.

Of course there are many different choices for $A_{\Mt}$. But we would
like to choose $A_{\Mt}$ so that the following hypothesis is
satisfied:
\begin{hypo1}
Assume $F$ is global. The image of $M(F)$ in the quotient $A_M(\A)M^{(n)}(\A)\backslash
M(\A)$ is discrete in the quotient topology. 
\end{hypo1}

The author does not know if one can always find such $A_{\Mt}$ for
general $n$. But at least we have
\begin{Prop}\label{P:hypothesis}
If $n=2$, the above hypothesis is satisfied for a suitable choice of
$A_{\Mt}$. For $n>2$, if
$d=\gcd(n, r-1+2cr)$ is such that $n$ divides $nr_i/d$ for all
$i=1,\dots,k$, (which is the case, for example, if $d=1$,) then the
above hypothesis is satisfied with $A_{\Mt}=Z_{\Mt}$.
\end{Prop}
\begin{proof}
This is proven in Appendix \ref{A:topology}.
\end{proof}

We believe that for any reasonable choice of
$A_{\Mt}$ the above hypothesis is always satisfied, but the author
does not know how to prove it at this moment. This is a bit
unfortunate in that this subtle technical issue makes the main theorem
of the paper conditional when $n>2$. However if $n=2$, our main
results are complete, and this is the only case we need for our applications to
symmetric square $L$-functions in \cite{Takeda1, Takeda2}, which is
the main motivation for the present work.

Let us mention that the group $A_M(\A)M^{(n)}(\A)$ (for any
choice of $A_M$) is a normal subgroup of $M(\A)$, and hence the
quotient $A_M(\A)M^{(n)}(\A)\backslash M(\A)$ is a group. Accordingly,
if the hypothesis is satisfied,  the image of $M(F)$ in the quotient
is a discrete subgroup and hence closed. 

Also we have
\[
A_{\Mt}(F)=A_{\Mt}(\A)\cap\s(M(F)).
\]
following the convention as in (\ref{E:H(F)}), and we set
\[
A_M(F)=p(A_{\Mt}(F)).
\]

\quad


\section{\bf On the local metaplectic tensor product}\label{S:Mezo}


In this section we first review the local metaplectic tensor product of Mezo
\cite{Mezo} and then extend his theory further, first by proving that the
metaplectic tensor product behaves in the expected way under the Weyl
group action, and second by establishing
the compatibility of the metaplectic tensor product with
parabolic inductions. Hence in this
section, all the groups are over a local (not necessarily non-archimedean)
field $F$ unless otherwise stated. Accordingly, we assume that our metaplectic group is
defined by the block-compatible cocycle $\sigma_r$ of \cite{BLS}, and
hence by $\GLt_r$ we actually mean $\sigGLt_r$.


\subsection{\bf Mezo's metaplectic tensor product}\label{SS:Mezo}


Let $\pi_1,\cdots,\pi_k$ be irreducible genuine
representations of $\GLt_{r_1},\dots,\GLt_{r_k}$, respectively. The
construction of the metaplectic tensor product takes several
steps. First of all,
for each $i$, fix an irreducible constituent $\pin_i$ of the restriction
$\pi_i|_{\GLtn_{r_i}} $ of $\pi_i$ to $\GLtn_{r_i}$. Then we have
\[
\pi_i|_{\GLtn_{r_i}}=\sum_{g}m_i\,^g(\pin_i),
\]
where $g$ runs through a finite subset of $\GLt_{r_i}$, $m_i$ is a
positive multiplicity and $^g(\pin_i)$ is the representation twisted
by $g$. Then we construct the tenor product representation
\[
\pin_1\otimes\cdots\otimes\pin_k
\]
of the group $\GLtn_{r_1}\times\cdots\times\GLtn_{r_k}$. Note that
this group
is merely the direct product of the groups $\GLtn_{r_i}$. The
genuineness of the representations $\pin_1,\dots,\pin_k$ implies that
this tensor product representation descends to a representation of the group
$\GLtn_{r_1}\timest\cdots\timest\GLtn_{r_k}$, \ie the representation
factors through the natural surjection
\[
\GLtn_{r_1}\times\cdots\times\GLtn_{r_k}\twoheadrightarrow 
\GLtn_{r_1}\timest\cdots\timest\GLtn_{r_k}=\Mtn.
\]
We denote this representation of $\Mtn$ by
\[
\pin:=\pin_1\,\otimest\cdots\otimest\,\pin_k,
\]
and call it the metaplectic tensor product of $\pin_1,\dots,\pin_k$. Let us note that
the space $V_{\pin}$ of $\pin$ is simply the tensor product
$V_{\pin_1}\otimes\cdots\otimes V_{\pin_k}$ of the spaces of
$\pin_i$. Let $\omega$ be a character on $Z_{\GLt_r}$ such that for
all $(aI_{r},\xi)\in Z_{\GLt_r}\cap\Mtn$ where $a\in F^\times$ we have
\[
\omega(aI_r,\xi)=\pin(aI_r,\xi)=\xi\pin_1(aI_{r_1},1)\cdots\pin_k(aI_{r_k},1).
\]
Namely $\omega$ agrees with $\pin$ on the intersection
$Z_{\GLt_r}\cap\Mtn$. We can extend $\pin$ to the representation
\[
\pin_\omega:=\omega\pin
\] 
of $Z_{\GLt_r}\Mtn$ by letting $Z_{\GLt_r}$ act by $\omega$. Now extend the representation
$\pin_\omega$ to a representation $\rho_\omega$ of a subgroup $\Ht$ of
$\Mt$ so that $\rho_\omega$ satisfies Mackey's irreducibility
criterion and so the induced representation
\begin{equation}\label{E:Mezo_tensor}
\pi_\omega:=\Ind_{\Ht}^{\Mt}\rho_\omega
\end{equation}
is irreducible. It is always possible to find such $\Ht$ and moreover
$\Ht$ can be chosen to be normal. Mezo shows
in \cite{Mezo} that $\pi_\omega$ is dependent only on $\omega$
and is independent of the other
choices made throughout, namely the choices of $\pin_i$, $\Ht$ and
$\rho_\omega$. We write
\[
\pi_\omega=(\pi_1\,\otimest\cdots\otimest\,\pi_k)_\omega
\]
and call it the metaplectic tensor product of $\pi_1,\dots,\pi_k$ with
the character $\omega$.

Mezo also shows that the metaplectic tensor product $\pi_\omega$ is
unique up to twist. Namely
\begin{Prop}\label{P:local_uniqueness}
Let $\pi_1,\dots,\pi_k$ and $\pi'_1,\dots,\pi'_k$ be representations
of $\GLt_{r_1},\dots,\GLt_{r_k}$. They give rise to isomorphic
metaplectic tensor products with a character $\omega$, \ie
\[
(\pi_1\,\otimest\cdots\otimest\,\pi_k)_\omega\cong
(\pi'_1\,\otimest\cdots\otimest\,\pi'_k)_\omega,
\]
if and only if for each $i$ there exists a character $\omega_i$ of
$\GLt_{r_i}$ trivial on $\GLtn_{r_i}$ such that
$\pi_i\cong\omega_i\otimes\pi'_i$.
\end{Prop}
\begin{proof}
This is \cite[Lemma 5.1]{Mezo}.
\end{proof}

\begin{Rmk}\label{R:dependence_on_omega}
Though the metaplectic tensor product generally depends on the choice of
$\omega$, if the center $Z_{\GLt_r}$ is already contained in $\Mtn$,
we have $\pin_\omega=\pin$ and hence there is no actual choice for $\omega$
and the metaplectic tensor product is canonical. This is the case, for
example, when $n=2$ and $r$ is even, which is one of the important
cases we consider in our applications in \cite{Takeda1, Takeda2}.
\end{Rmk}

\begin{Rmk}
The equality (\ref{E:center=center}) implies that extending a
representation $\pin$ of $\Mtn$ to $\pin_\omega$ multiplying the
character $\omega$ on $Z_{\GLt_r}$ is the same
as extending it by multiplying an appropriate character on
$Z_{\Mt}$. 
\end{Rmk}

Let us mention the following, which is not explicitly mentioned in
\cite{Mezo}.
\begin{Lem}\label{L:always_tensor_product}
Let $\pi_\omega$ be an irreducible admissible representation of $\Mt$
where $\omega$ is the character on $Z_{\GLt_r}$ defined by
$\omega=\pi_\omega|_{Z_{\GLt_r}}$. Then there exist irreducible
admissible representations $\pi_1,\dots\pi_k$ of
$\GLt_{r_1},\dots,\GLt_{r_k}$, respectively, such that
\[
\pi_\omega=(\pi_1\,\otimest\cdots\otimest\,\pi_k)_\omega.
\]
Namely a representation of $\Mt$ is always a metaplectic tensor product.
\end{Lem}
\begin{proof}
The restriction $\pi_\omega|_{Z_{\GLt_r}\Mtn}$ contains a
representation of the form $\omega
(\pin_1\,\otimest\cdots\otimest\,\pin_k)$ for some representations
$\pin_i$ of $\GLt_{r_i}$. Let $\pi_i$ be an irreducible constituent of
$\Ind_{\GLtn_{r_i}}^{\GLt_{r_i}}\pin_i$. Then one can see that
$\pi_\omega$ is $(\pi_1\,\otimest\cdots\otimest\,\pi_k)_\omega$.
\end{proof}

From Mezo's construction, one can tell that essentially the
representation theory of the group $\Mt$ is determined by that of
$Z_{\GLt_r}\Mtn$. Let us briefly explain why this is so. Let $\pi$ be
an irreducible admissible representation of $\Mt$, and
$\chi_\pi:\Mt\rightarrow\C$ be the distribution character. If $\pi$ is
genuine, so is $\chi_\pi$. Namely
$\chi_\pi((1,\xi)\mt)=\xi\chi_\pi(\mt)$ for all $\xi\in\mu_n$ and
$\mt\in\Mt$. But if $\mt\in\Mt$ is a regular
element but not in $Z_{\GLt_r}\Mtn$, then one can find $\xi\in\mu_n$
with $\xi\neq 1$
such that $(1,\xi)\mt$ is conjugate to $\mt$. This is proven in the same
way as \cite[Proposition 0.1.4]{KP}. (The only modification one needs
is to
choose $A\subset M_r(F)$ in their proof so that $A\subset
M_{r_1}(F)\times\cdots\times M_{r_k}(F)$.) Therefore for such $\mt$, one
has $\chi_\pi(\mt)=0$. Namely, the support of $\chi_\pi$ is contained
in $Z_{\GLt_r}\Mtn$. (Indeed, this argument by the
distribution character is crucially used in \cite[Lemma 4.2]{Mezo}. ) This
explains why $\pi$ is essentially determined by the restriction
$\pi|_{Z_{\GLt_r}\Mtn}$.  

This idea can be observed in 
\begin{Lem}\label{L:equivalent_tensor_product}
Let $\pi$ and $\pi'$ be irreducible admissible representations of
$\Mt$. Then $\pi$ and $\pi'$ are equivalent if and only if
$\pi|_{Z_{\GLt_r}\Mtn}$ and $\pi'|_{Z_{\GLt_r}\Mtn}$ have an
equivalent constituent.
\end{Lem}
\begin{proof}
This follows from Proposition \ref{P:local_uniqueness} and Lemma
\ref{L:always_tensor_product}.
\end{proof}

Also let us mention
\begin{Prop}\label{P:Mezo}
we have
\[
\Ind_{Z_{\GLt_r}\Mtn}^{\Mt}\pin_\omega=m\pi_\omega
\]
for some finite multiplicity $m$, so every constituent of
$\Ind_{Z_{\GLt_r}\Mtn}^{\Mt}\pin_\omega$ is isomorphic to $\pi_\omega$.
\end{Prop}
\begin{proof}
By inducting in stages, we have
$\Ind_{Z_{\GLt_r}\Mtn}^{\Mt}\pin_\omega
=\Ind_{\Ht}^{\Mt}\Ind_{Z_{\GLt_r}\Mtn}^{\Ht}\pin_\omega$, where $\Ht$
is as in (\ref{E:Mezo_tensor}), 
and by \cite[Lemma 4.1]{Mezo} we have
\[
\Ind_{Z_{\GLt_r}\Mtn}^{\Ht}\pin_\omega=\bigoplus_{\chi}\chi\otimes\rho_\omega
\]
where $\chi$ runs over the finite set of characters of $\Ht$ that are
trivial on $Z_{\GLt_r}\Mtn$. Moreover it is shown in \cite[Lemma 4.1]{Mezo} that 
any extension of $\pin_\omega$ to
$\Ht$ is of the form $\chi\otimes\rho_\omega$ and
$\Ind^{\Mt}_{\Ht}\chi\otimes\rho_\omega=\pi_\omega$ for all $\chi$ by
\cite[Lemma 4.2]{Mezo}. Hence we have
\[
\Ind_{Z_{\GLt_r}\Mtn}^{\Mt}\pin_\omega
=\bigoplus_\chi \Ind_{\Ht}^{\Mt}\chi\otimes\rho_\omega
=m\pi_\omega.
\]
\end{proof}

Let $\omega$ be as above and $A_{\Mt}$ as in Section
\ref{S:abelian}. The restriction $\pin|_{A_{\Mt}\cap\Mtn}$
gives a character on $A_{\Mt}\cap\Mtn$ because
${A_{\Mt}\cap\Mtn}$ is contained in the center of $\Mtn$ by
(\ref{E:center_Mtn}). The product $\omega
(\pin|_{A_{\Mt}\cap\Mtn})$ of $\omega$ and $\pin|_{A_{\Mt}\cap\Mtn} $ defines a
character on $Z_{\GLt_r}(A_{\Mt}\cap\Mtn)$ because the two
characters agree on $Z_{\GLt_r}\cap (A_{\Mt}\cap\Mtn)$. Since the Pontryagin
dual is an exact functor, one can extend
it to a character on $A_{\Mt}$, which we denote again by
$\omega$. Namely $\omega$ is a character on $A_{\Mt}$ extending
$\omega$ such that $\omega(a)=\pin(a)$ for all
$a\in A_{\Mt}\cap\Mtn$. With this said, we have

\begin{Cor}\label{C:local_tensor}
Let $\omega$ be the character on $A_{\Mt}$ described above, and let
$\pin_{\omega}:=\omega\pin$ be the representation of $A_{\Mt}\Mtn$ extending
$\pin$ by letting $A_{\Mt}$ act as $\omega$. Then
\[
\Ind_{A_{\Mt}\Mtn}^{\Mt}\pin_{\omega}=m'\pi_\omega
\]
where $m'$ is some finite multiplicity.
\end{Cor}
\begin{proof}
This follows from the previous proposition because we have the
inclusion $\Ind_{A_{\Mt}\Mtn}^{\Mt}\pin_{\omega}\hookrightarrow
\Ind_{Z_{\GLt_r}\Mtn}^{\Mt}\pin_{\omega}$.
\end{proof}
\quad


\subsection{\bf The archimedean case}
\label{S:archimedean}


Let us make some remarks when
$F$ is archimedean. Strictly speaking, Mezo assumes that the field $F$
is non-archimedean. If $F=\C$, then $\Mtn=\Mt$. Indeed,
$\Mt(\C)=M(\C)\times\mu_n$ (direct product), and the metaplectic tensor product is
obtained simply by taking the tensor product $\pi_1\otimes\cdots\otimes\pi_k$
and descending it to $\Mt(\C)$. Hence there is essentially no
discrepancy between the metaplectic case and the non-metaplectic one.

If $F=\R$ (so necessarily $n=2$), one
can trace the argument of Mezo and make sure the construction works
for this case as well, with the proviso that equivalence has to be
considered as infinitesimal equivalence. However, it has been
communicated to the author by J. Adams that for this case, the induced
representation $\Ind_{Z_{\GLt_r}\Mtn}^{\Mt}\pin_\omega$ is always
irreducible. (See \cite{Adams}). Hence one can simply define the
metaplectic tensor product to be this induced representation.

\quad


\subsection{\bf Twists by Weyl group elements}
\label{S:Weyl_group_local}


As in the notation section, we let $W_M$ be the subset of the Weyl
group $W_{\GL_r}$ consisting of
only those elements which permute the $\GL_{r_i}$-factors of
$M=\GL_{r_1}\times\cdots\times\GL_{r_k}$. Though $W_M$ is not a group
in general, it is identified with the group $S_k$ of permutations of
$k$ letters. Assume $w\in W_M$ is such that
\[
M':=wMw^{-1}=\GL_{r_{\sigma(1)}}\times\cdots\times\GL_{r_{\sigma(k)}}
\]
for a permutation $\sigma\in S_k$, and so
$w(g_1,\dots,g_k)w^{-1}=(g_{\sigma(1)},\dots,g_{\sigma(k)})$ for each
$(g_1,\dots,g_k)\in M$. Namely $w$ corresponds to the permutation
$\sigma^{-1}$. Then we have
\[
\MMt=\s(w)\Mt\s(w)^{-1}.
\]

Let $\pi=(\pi_1\,\otimest\cdots\otimest\,\pi_k)_\omega$ be an irreducible
admissible representation of $\Mt$. As in the notation section one can define the
twist $^{\s(w)}\pi$ of $\pi$ by $\s(w)$ to be the representation of $\MMt$ on
the space $V_\pi$ given by $^{\s(w)}\pi(\mt')=\pi(\s(w)^{-1}\mt'\s(w))$ for
$\mt'\in\MMt$. To ease the notation we simply write
\[
^w\pi:=\,^{\s(w)}\pi.
\]
Actually since $\mu_n\subseteq\Mt$ is in the center, for any preimage
$\tilde{w}$ of $w$, we have $\,^{\s(w)}\pi=\,^{\tilde{w}}\pi$, and
hence the notation $^w\pi$ is not ambiguous. 

The goal of this subsection is to show that the
metaplectic tensor product behaves in the expected way under the Weyl
group action. Namely, we will prove
\begin{Thm}\label{T:Weyl_group_local}
With the above notations, we have
\begin{equation}\label{E:Weyl_group_local}
^w(\pi_1\,\otimest\cdots\otimest\,\pi_k)_\omega
\cong(\pi_{\sigma(1)}\,\otimest\cdots\otimest\,\pi_{\sigma(k)})_\omega.
\end{equation}
\end{Thm}

To prove this, we first need
\begin{Lem}
For each $(m,1)\in\Mtn$ and $w\in W_M$, where $m\in M^{(n)}$, we have
\[
\s(w)(m,1)\s(w)^{-1}=(wmw^{-1},1),
\]
namely $\s(w)\s(m)\s(w)^{-1}=\s(wmw^{-1})$.
\end{Lem}
\begin{proof}
Note that $\s(w)=(w,1)$ and
$\s(w)^{-1}=(w^{-1},\sigma_r(w,w^{-1})^{-1})$ because we are using
$\sigGLt_r$, and hence
\[
\s(w)(m,1)\s(w)^{-1}
=(wmw^{-1},\sigma_r(w, mw^{-1})\sigma_r(m, w^{-1}) \sigma_r(w,w^{-1})^{-1}).
\]
Let
\[
\varphi_w(m):=\sigma_r(w, mw^{-1})\sigma_r(m, w^{-1}) \sigma_r(w,w^{-1})^{-1}.
\]
 We need to show $\varphi_w(m)=1$ for all $m\in M^{(n)}$. Let us first
 show that the map $m\mapsto\varphi_w(m)$ is a
 homomorphism on $M^{(n)}$. To see it, for $m, m'\in M^{(n)}$ we have
\begin{align*}
\s(w)(m,1)(m',1)\s(w)^{-1}&=\s(w)(mm',\sigma_r(m,m'))\s(w)^{-1}\\
&=(wmm'w^{-1},\sigma_r(m,m')\varphi_w(mm')).
\end{align*}
On the other hand, we have
\begin{align*}
\s(w)(m,1)(m',1)\s(w)^{-1}&=\s(w)(m,1)\s(w)^{-1}\s(w)(m',1)\s(w)^{-1}\\
&=(wmw^{-1},\varphi_w(m))(wm'w^{-1},\varphi_w(m'))\\
&=(wmm'w^{-1}, \sigma_r(wmw^{-1},wm'w^{-1})
\varphi_w(m)\varphi_w(m'))\\
&=(wmm'w^{-1}, \sigma_r(m,m')
\varphi_w(m)\varphi_w(m')),
\end{align*}
where the last equality follows because
$\sigma_r(wmw^{-1},wm'w^{-1})=\sigma_r(m,m')$ by the
block-compatibility of $\sigma_r$. Hence by comparing those two, one
obtains $\varphi_w(mm')=\varphi_w(m)\varphi_w(m')$. Therefore to show
$\varphi_w(m)=1$, it suffices to show it for the elmenets of the form
\begin{equation}\label{E:form_of_m}
m=\diag(I_{r_1},\dots,I_{r_{i-1}},g_i,I_{r_{i+1}},\dots, I_{r_k})
\end{equation}
for $g_i\in\GL_{r_i}^{(n)}$.

Then one can rewrite $\varphi_w(m)$ as follows:
\begin{align*}
\varphi_w(m)&=\sigma_r(w, mw^{-1})\sigma_r(m, w^{-1})
\sigma_r(w,w^{-1})^{-1}\\
&=\sigma_r(w, w^{-1}wmw^{-1})\sigma_r(m, w^{-1})
\sigma_r(w,w^{-1})^{-1}\\
&=\sigma_r(ww^{-1}, wmw^{-1})\sigma_r(w, w^{-1})\sigma_r(w^{-1},
wmw^{-1})^{-1}\sigma_r(m, w^{-1})\sigma_r(w,w^{-1})^{-1}\\
&=\sigma_r(w^{-1}, wmw^{-1})^{-1}\sigma_r(m, w^{-1}),
\end{align*}
where for the third equality we used Proposition \ref{P:BLS} (0).
So we only have to show
\begin{equation}\label{E:cocycle_computation1}
\sigma_r(w^{-1}, wmw^{-1})^{-1}\sigma_r(m, w^{-1})=1.
\end{equation}

This can be shown by using the algorithm computing the cocycle
$\sigma_r$ given by \cite{BLS}. To use the results of \cite{BLS}, it
should be mentioned that one needs to use the set $\M$ for a set of
representatives of the Weyl group of $\GL_r$ as defined in the
notation section. Also let us recall the following notation from
\cite{BLS}: For each
$g\in\GL_r$, the ``torus part function'' $\ttt:\GL_r\rightarrow T$ is
the unique map such that
\[
\ttt(nt\eta n')=t,
\] 
where $n, n'\in N_B$, $t\in T$ and $\eta\in\mathfrak{M}$ when $\GL_r$
is written as
\[
\GL_r=\coprod_{\eta\in\mathfrak{M}}N_BT\eta N_B
\]
by the Bruhat decomposition. Namely $\ttt(g)$ is the ``torus part'' of
$g$.

Using this language, each $w\in W_M$ is written as
\[
w=\ttt(w)\eta_w
\]
where $\eta_w\in\mathfrak{M}$, and
$\ttt(w)\in\GL_{r_{\sigma(1)}}\times\cdots\times\GL_{r_{\sigma(k)}}$ is
of the form
\[
\ttt(w)=(\varepsilon_{\sigma(1)}I_{\sigma(1)},\dots,\varepsilon_{\sigma(k)}I_{\sigma(k)}),
\]
where $\varepsilon_i\in\{\pm 1\}$.

We are now ready to carry out our cocycle computations for
(\ref{E:cocycle_computation1}). Let us deal with $\sigma_r(m, w^{-1})$
first. Write $m=nt\eta n'$ by the Bruhat decomposition, so
$\ttt(m)=t$. But recall that we are assuming $m$ is of the form
as in (\ref{E:form_of_m}), so the decomposition $nt\eta n'$ takes place essentially
inside the $\GL_{r_i}$-block.  In particular, we can write
\[
m=\diag(I_{r_1},\dots,I_{r_{i-1}},n_it_i\eta_i n_i',I_{r_{i+1}}\dots,I_{r_k}),
\]
where $t_i\in\GL^{(n)}_{r_i}$. (Note that $\det(t_i)\in F^{\times n}$.)
Then one can compute $\sigma_r(m,
w^{-1})$ as follows:
\begin{align*}
\sigma_r(m, w^{-1})&=\sigma_r(nt\eta n', w^{-1})\\
&=\sigma_r(t\eta, n'w^{-1})\quad\text{by Proposition \ref{P:BLS} (1), (2)}\\
&=\sigma_r(t\eta, w^{-1}wn'w^{-1})\\
&=\sigma_r(t\eta, w^{-1})\quad\text{because $wn'w^{-1}\in N_B$ and 
by Proposition \ref{P:BLS} (1)}\\
&=\sigma_r(t\eta, \ttt(w^{-1})\eta_{w^{-1}}).
\end{align*}
Now since $\eta$ is essentially inside the $\GL_{r_i}$-factor of $M$
and $\eta_{w^{-1}}$ only permutes the $\GL_{r_j}$-factors of $M$, we
have $l(\eta\eta_{w^{-1}})=l(\eta)+l(\eta_{w^{-1}})$, where $l$ is the
length function. Hence by applying \cite[Lemma 10, p.155]{BLS}, we
have
\begin{equation}\label{E:cocycle_computation2}
\sigma_r(t\eta, \ttt(w^{-1})\eta_{w^{-1}})
=\sigma_r(t,\eta\ttt(w^{-1})\eta^{-1})\sigma_r(\eta, \ttt(w^{-1})).
\end{equation}
Here note that $\ttt(w^{-1})\in
M=\GL_{r_1}\times\cdots\times\GL_{r_k}$ is of the form
$(\varepsilon_1 I_{r_1},\dots, \varepsilon_k I_{r_k})$ and $\eta$ is
in the $\GL_{r_i}$-block. Hence
$\eta\ttt(w^{-1})\eta^{-1}=\ttt(w^{-1})$. Thus by the
block-compatibility of $\sigma_r$, (\ref{E:cocycle_computation2}) is
written as
\[
\sigma_{r_i}(t_i, \varepsilon_i I_{r_i})\sigma_{r_i}(\eta_i, \varepsilon_i I_{r_i}).
\]
Clearly, if $\varepsilon_i=1$, then both $\sigma_{r_i}(t_i,
\varepsilon_i I_{r_i})$ and $\sigma_{r_i}(\eta_i, \varepsilon_i
I_{r_i})$ are $1$. If $\varepsilon_i=-1$, then by Proposition
\ref{P:BLS} (3), one can see that $\sigma_{r_i}(\eta_i,
\varepsilon_i I_{r_i})=1$. Hence in either case, one has
\begin{equation}\label{E:cocycle_computation3}
\sigma_r(m,w^{-1})
=\sigma_{r_i}(t_i, \varepsilon_i I_{r_i}).
\end{equation}

Next let us deal with $\sigma_r(w^{-1}, wmw^{-1})$ in
(\ref{E:cocycle_computation1}). First by the analogous computation to
what we did for $\sigma(m, w^{-1})$, one can write
\begin{equation}\label{E:cocycle_computation4}
\sigma_r(w^{-1}, wmw^{-1})
=\sigma_r(w^{-1}, wt\eta w^{-1})
=\sigma_r(\ttt(w^{-1})\eta_{w^{-1}}, wt\eta w^{-1}).
\end{equation}
Since $w$ corresponds to the permutation $\sigma^{-1}$, if we let
\[
\tau_i:\GL_{r_i}\rightarrow
wMw^{-1}=\GL_{r_{\sigma(1)}}\times\cdots\times\GL_{r_{\sigma(k)}}
\]
be the embedding of $\GL_{r_i}$ into the corresponding $\GL_{r_i}$-factor of
$wMw^{-1}$, then (\ref{E:cocycle_computation4}) is written as
\[
\sigma_r(\ttt(w^{-1})\eta_{w^{-1}}, \tau_i(t_i)\tau_i(\eta_i)).
\]
Note that $\tau_i(\eta_i)\in\mathfrak{M}$ and
$l(\eta_{w^{-1}}\tau_i(\eta_i))=l(\eta_{w^{-1}})+l(\tau_i(\eta_i))$. Hence
by using \cite[Lemma 10, p.155]{BLS}, this is written as
\begin{equation}\label{E:cocycle_computation5}
\sigma_r(\ttt(w^{-1}), \eta_{w^{-1}}\tau_i(t_i)\eta_{w^{-1}}^{-1})
\sigma_r(\eta_{w^{-1}}, \tau_i(t_i)).
\end{equation}
By the block compatibility of $\sigma_r$, one can see
\[
\sigma_r(\ttt(w^{-1}), \eta_{w^{-1}}\tau_i(t_i)\eta_{w^{-1}}^{-1})
=\sigma_{r_i}(\varepsilon_iI_{r_i}, t_i).
\]
Also to compute $\sigma_r(\eta_{w^{-1}}, \tau_i(t_i))$, one needs to
use Proposition \ref{P:BLS} (3). For this purpose, let us write
\[
t_i=\begin{pmatrix}a_1&&\\ &\ddots&\\ &&a_{r_i}\end{pmatrix}\in\GL_{r_i}
\]
where $\det(t_i)=a_1\cdots a_{r_i}\in F^{\times n}$. By looking at the
formula in Proposition \ref{P:BLS} (3), one can see that
$\sigma_r(\eta_{w^{-1}}, \tau_i(t_i))$ is a power of 
\[
(-1, a_1)\cdots(-1,a_{r_i}), 
\]
which is equal to $(-1, a_1\cdots a_{r_i})=1$ because $\det(t_i)=a_1\cdots
a_{r_i}\in F^{\times n}$. Hence (\ref{E:cocycle_computation5}), which
is the same as (\ref{E:cocycle_computation4}), becomes
\[
\sigma_{r_i}(\varepsilon_iI_{r_i}, t_i).
\]
Hence the left hand side of (\ref{E:cocycle_computation1}) is written
as
\[
\sigma_{r_i}(\varepsilon_iI_{r_i}, t_i)^{-1}\sigma_{r_i}(t_i,\varepsilon_iI_{r_i}).
\]
We need to show this is $1$. But clearly this is the case if
$\varepsilon_i=1$. So let us assume $\varepsilon_i=-1$. Namely we will
show $\sigma_{r_i}(-I_{r_i}, t_i)^{-1}\sigma_{r_i}(t_i,-I_{r_i})=1$. But
by Proposition \ref{P:BLS} (4), one can compute
\[
\sigma_{r_i}(-I_{r_i}, t_i)=(-1, a_2)(-1, a_3)^2(-1,a_4)^3\cdots (-1, a_r)^{r-1+2c}
\]
and 
\[
\sigma_{r_i}(t_i, -I_{r_i})=(a_1, -1)^{r-1}(a_2,-1)^{r-2}(a_3,-1)^{r-3}\cdots (-1, a_{r-1}).
\]
Noting that $(-1,a_i)^{-1}=(a_i, -1)$, we have
\[
\sigma_{r_i}(-I_{r_i}, t_i)^{-1}\sigma_{r_i}(t_i,
-I_{r_i})=\prod_{i=1}^r(a_i, -1)^{r-1+2c}=(\prod_{i=1}^ra_i, -1)^{r-1+2c}=1,
\]
where the last equality follows because
$\det(t_i)=\prod_{i=1}^ra_i\in F^{\times n}$. This completes the proof.
\end{proof}

Now we are ready to prove Theorem \ref{T:Weyl_group_local}.

\begin{proof}[Proof of Theorem \ref{T:Weyl_group_local}]
By restricting to $\MMtn$, one can see that the left hand side of (\ref{E:Weyl_group_local})
contains the representation $^w(\pin_1\,\otimest\cdots\otimest\,\pin_k)$
and the right hand side of (\ref{E:Weyl_group_local}) contains
$\pin_{\sigma(1)}\,\otimest\cdots\otimest\,\pin_{\sigma(k)}$, where
$^w(\pin_1\,\otimest\cdots\otimest\,\pin_k)$ is the representation of $\MMtn=\s(w)\Mtn
\s(w)^{-1}$ whose space is the space of
$\pin_1\,\otimest\cdots\otimest\,\pin_k$. Hence by Lemma
\ref{L:equivalent_tensor_product}, it suffices to show that
\[
^w(\pin_1\,\otimest\cdots\otimest\,\pin_k)\cong
\pin_{\sigma(1)}\,\otimest\cdots\otimest\,\pin_{\sigma(k)}.
\]
But this can be seen from the commutative diagram
\[
\xymatrix{ 
&\GLtn_{r_{\sigma(1)}}\times\cdots\times\GLtn_{r_{\sigma(k)}}\ar[ddl]\ar[ddr]\ar[d]&\\ 
&\MMtn\ar[dl]\ar[dr]&\\ 
\Aut(V_{\pin_1}\otimes\cdots\otimes V_{\pin_k})\ar[rr]^{\sim}&&
\Aut(V_{\pin_{\sigma(1)}}\otimes\cdots\otimes V_{\pin_{\sigma(k)}}),
}
\]
where the left most arrow is the representation of
$\GLtn_{r_\sigma(1)}\times\cdots\times\GLtn_{r_\sigma(k)}$ (direct
product) acting on the space of $\pin_1\otimes\cdots\otimes\pin_k$ by
permuting each factor by $\sigma^{-1}$, which descends to the
representation $^w(\pin_1\,\otimest\cdots\otimest\,\pin_k)$ of $\Mtn$. To see this
indeed descends to $^w(\pin_1\,\otimest\cdots\otimest\,\pin_k)$, one needs
the above lemma.
\end{proof}

\quad


\subsection{\bf Compatibility with parabolic induction}
\label{S:parabolic_induction_local}


We will show the compatibility of the metaplectic tensor product with
parabolic induction. Hence we consider the standard parabolic subgroup
$P=MN\subseteq\GL_r$
where $M$ is the Levi part and $N$ the unipotent radical. 

First let us mention
\begin{Lem}\label{L:normalizer_local}
The image $N^\ast$ of the unipotent radical $N$ via the section $\s:\GL_r\rightarrow\GLt_r$
is normalized by the metaplectic preimage $\Mt$ of the Levi part $M$.
\end{Lem}
\begin{proof}
This is known not only for $\GLt_r$ but for
any covering group. See \cite[Appendix I]{MW}.  But we will give a
simple proof for $\GLt_r$. Let
$\mt\in\Mt$ and $(n,1)\in N^\ast$, where $n\in N$. (Note that
since we are assuming the group $\GLt_r$ is defined by $\sigma_r$, each element
in $N^\ast$ is written as $(n,1)$.) We may assume $\mt=(m,1)$ for
$m\in M$. Noting that $\mt^{-1}=(m^{-1},\sigma_r(m,m^{-1})^{-1})$, we
compute
\begin{align*}
\mt(n,1)\mt^{-1}&=(m,1)(n,1)(m^{-1}, \sigma_r(m,m^{-1})^{-1})\\
&=(mn, \sigma_r(m,n)) (m^{-1}, \sigma_r(m,m^{-1})^{-1})\\
&=(mnm^{-1},\sigma_r(mn, m^{-1})\sigma_r(m,n) \sigma_r(m,m^{-1})^{-1}).
\end{align*}
By Proposition \ref{P:BLS} (1), $\sigma_r(m,n)=1$. Also since
$mnm^{-1}\in N$, we have $\sigma_r(mn,
m^{-1})=\sigma_r(mnm^{-1}m, m^{-1})=\sigma_r(m,m^{-1})$ again by
Proposition \ref{P:BLS} (1).  Thus we have
$\mt(n,1)\mt^{-1}=(mnm^{-1},1)\in N^\ast$.
\end{proof}

By this lemma, we can write
\[
\Pt=\Mt N^\ast
\]
where $\Mt$ normalizes $N^\ast$ and hence for a representation $\pi$
of $\Mt$ one can form the induced representation
\[
\Ind_{\Mt N^\ast}^{\GLt_r}\pi
\]
by letting $N^\ast$ act trivially.

\begin{Thm}\label{T:induction_local}
Let $P=MN\subseteq\GL_r$ be the standard parabolic subgroup whose Levi
part is $M=\GL_{r_1}\times\cdots\times\GL_{r_k}$. Further for each
$i=1,\dots,k$ let $P_i=M_iN_i\subseteq\GL_{r_i}$ be the standard
parabolic of $\GL_{r_i}$ whose Levi part is
$M_i=\GL_{r_{i,1}}\times\cdots\times\GL_{r_{i, l_i}}$. For each $i$, we are
given a representation 
\[
\sigma_i:=(\tau_{i,1}\,\otimest\cdots\otimest\,\tau_{i,l_i})_{\omega_i}
\]
of $\Mt_i$, which is given as the metaplectic tensor product of the
representations $\tau_{i,1},\dots,\tau_{i,l_i}$ of
$\GLt_{r_{i,1}},\dots,\GLt_{r_{i, l_i}}$. Assume that $\pi_i$ is an
irreducible constituent of the induced representation
$\Ind_{\Pt_i}^{\GLt_{r_i}}\sigma_i$. Then the metaplectic tensor
product
\[
\pi_\omega:=(\pi_1\,\otimest\cdots\otimest\,\pi_k)_\omega
\]
is an irreducible constituent of the induced representation
\[
\Ind_{\Qt}^{\Mt}(\tau_{1,1}\,\otimest\cdots\otimest\,\tau_{1, l_1}\,\otimest\cdots\otimest
\,\tau_{k,1}\,\otimest\cdots\otimest\,\tau_{k,l_k})_\omega,
\]
where $Q$ is the standard parabolic of $M$ whose Levi part is
$M_1\times\cdots\times M_k$.
\end{Thm}

First we need 
\begin{Lem}
For a genuine representation $\pi$ of a Levi part $\Mt$, the map
\[
\Ind_{\Mt N^\ast}^{\GLt_r}\pi\rightarrow\Ind_{\Mtnn N^\ast}^{\GLtn_r}\pi|_{\Mtnn}
\]
given by the restriction $\varphi\mapsto \varphi|_{\GLtn_r}$ for
$\varphi\in \Ind_{\Mt N^\ast}^{\GLt_r}\pi$ is an isomorphism, where
\[
\Mtnn=\Mt\cap\GLtn_r.
\]
Hence in particular
\[
\left(\Ind_{\Mt N^\ast}^{\GLt_r}\pi\right)|_{\GLtn_r}\cong
\left(\Ind_{\Mt N^\ast}^{\GLt_r}\pi\right)\|_{\GLtn_r}\cong\Ind_{\Mtnn N^\ast}^{\GLtn_r}\pi|_{\Mtnn}
\]
as representations of $\GLtn_r$.
\end{Lem}
\begin{proof}
To show it is one-to-one, assume $\varphi|_{\GLtn_r}=0$. We need to show
$\varphi=0$. But for any $g\in\GL_r$, one can write $g=\begin{pmatrix}\det
  g^{-n+1}&\\ & I_{r-1}\end{pmatrix}\begin{pmatrix}\det
  g^{n-1}&\\ & I_{r-1}\end{pmatrix}g$, where $\begin{pmatrix}\det
  g^{-n+1}&\\ & I_{r-1}\end{pmatrix}\in M$ and $\begin{pmatrix}\det
  g^{n-1}&\\ & I_{r-1}\end{pmatrix}g\in\GL_r^{(n)}$. Hence any
$\gt\in\GLt_r$ is written as $\gt=\mt\gt'$ for some $\mt\in\Mt$ and
$\gt'\in\GLtn_r$. Hence $\varphi(\gt)=\pi(\mt)\varphi(\gt')$. But
$\varphi(\gt')=0$. Hence $\varphi(\gt)=0$.

To show it is onto, let $\varphi\in\Ind_{\Mtnn
  N^\ast}^{\GLtn_r}\pi|_{\Mtnn}$. Define
$\widetilde{\varphi}:\GLt_r\rightarrow\pi$ by
\[
\widetilde{\varphi}(g,\xi)=\xi\pi(\begin{pmatrix}\det
  g^{-n+1}&\\ & I_{r-1}\end{pmatrix}, \eta)\varphi(\begin{pmatrix}\det
  g^{n-1}&\\ & I_{r-1}\end{pmatrix}g, 1),
\]
where $\eta$ is chosen to be such that $(\begin{pmatrix}\det
  g^{-n+1}&\\ & I_{r-1}\end{pmatrix}, \eta)(\begin{pmatrix}\det
  g^{n-1}&\\ & I_{r-1}\end{pmatrix}g, 1)=(g,1)$. Namely $\eta$ is
given by the cocycle as
\[
\eta=\sigma_r(\begin{pmatrix}\det
  g^{-n+1}&\\ & I_{r-1}\end{pmatrix}, \begin{pmatrix}\det
  g^{n-1}&\\ & I_{r-1}\end{pmatrix}g)^{-1}.
\]
That $\widetilde{\varphi}|_{\GLtn_r}=\varphi$ follows because if $g\in\GL_r^{(n)}$
then $(\begin{pmatrix}\det
  g^{-n+1}&\\ & I_{r-1}\end{pmatrix}, \eta)\in\Mtnn$. Also one can
check $\widetilde{\varphi}\in \Ind_{\Mt N^\ast}^{\GLt_r}\pi$ as follows: We
need to check $\varphi(\mt(g,\xi))=\pi(\mt)\varphi(g,\xi)$ for all
$\mt\in\Mt$. But since $\pi$ (and hence $\varphi$) is genuine, we may assume
$\mt$ is of the form $(m,1)$ for $m\in M$ and $\xi=1$. Then
\begin{align}
\notag\widetilde{\varphi}((m,1)(g,1))&=\widetilde{\varphi}(mg,\sigma_r(m,g))\\
\label{E:F_tilde}&=\sigma_r(m,g)\pi(\begin{pmatrix}\det
  (mg)^{-n+1}&\\ & I_{r-1}\end{pmatrix}, \eta_1)\varphi(\begin{pmatrix}\det
  (mg)^{n-1}&\\ & I_{r-1}\end{pmatrix}mg, 1),
\end{align}
where 
\[
\eta_1=\sigma_r(\begin{pmatrix}\det
  (mg)^{-n+1}&\\ & I_{r-1}\end{pmatrix}, \begin{pmatrix}\det
  (mg)^{n-1}&\\ & I_{r-1}\end{pmatrix}mg)^{-1}.
\]
Now 
\begin{align*}
&(\begin{pmatrix}\det(mg)^{n-1}&\\ & I_{r-1}\end{pmatrix}mg, 1)\\ 
=&(\begin{pmatrix}\det(mg)^{n-1}&\\ &
  I_{r-1}\end{pmatrix}m\begin{pmatrix}\det g^{-n+1}&\\ &
  I_{r-1}\end{pmatrix}, \eta_2)
(\begin{pmatrix}\det
  g^{n-1}&\\ & I_{r-1}\end{pmatrix}g, 1),
\end{align*}
where 
\[
\eta_2=\sigma_r(\begin{pmatrix}\det(mg)^{n-1}&\\ &
  I_{r-1}\end{pmatrix}m\begin{pmatrix}\det g^{-n+1}&\\ &
  I_{r-1}\end{pmatrix}, \begin{pmatrix}\det
  g^{n-1}&\\ & I_{r-1}\end{pmatrix}g)^{-1}.
\]
Since $(\begin{pmatrix}\det(mg)^{n-1}&\\ &
  I_{r-1}\end{pmatrix}m\begin{pmatrix}\det g^{-n-1}&\\ &
  I_{r-1}\end{pmatrix}, \eta_2)\in\Mtnn$, the right hand side of
(\ref{E:F_tilde}) becomes
\begin{align*}
&\sigma_r(m,g)\pi\Big((\begin{pmatrix}\det
  (mg)^{-n+1}&\\ & I_{r-1}\end{pmatrix}, \eta_1) (\begin{pmatrix}\det(mg)^{n-1}&\\ &
  I_{r-1}\end{pmatrix}m\begin{pmatrix}\det g^{-n+1}&\\ &
  I_{r-1}\end{pmatrix}, \eta_2)\Big)\\
&\qquad\qquad \varphi(\begin{pmatrix}\det
  g^{n-1}&\\ & I_{r-1}\end{pmatrix}g, 1)\\
&=\sigma_r(m,g)\pi(m\begin{pmatrix}\det g^{-n+1}&\\ &
  I_{r-1}\end{pmatrix},\eta_1\eta_2\eta_3) \varphi(\begin{pmatrix}\det
  g^{n-1}&\\ & I_{r-1}\end{pmatrix}g, 1)\\
&=\sigma_r(m,g)\pi(m,\eta_1\eta_2\eta_3\eta_4)\pi(\begin{pmatrix}\det g^{-n+1}&\\ &
  I_{r-1}\end{pmatrix}, 1) \varphi(\begin{pmatrix}\det
  g^{n-1}&\\ & I_{r-1}\end{pmatrix}g, 1),
\end{align*}
where
\[
\eta_3=\sigma_r(\begin{pmatrix}\det
  (mg)^{-n+1}&\\ & I_{r-1}\end{pmatrix},
\begin{pmatrix}\det(mg)^{n-1}&\\ &
  I_{r-1}\end{pmatrix}m\begin{pmatrix}\det g^{-n+1}&\\ &
  I_{r-1}\end{pmatrix}),
\]
and 
\[
\eta_4=\sigma_r(m, \begin{pmatrix}\det g^{-n+1}&\\ &
  I_{r-1}\end{pmatrix})^{-1}.
\]
Then one can compute
\[
\sigma_r(m,g)\eta_1\eta_2\eta_3\eta_4=\eta
\]
by using Proposition \ref{P:BLS} (0). Hence (\ref{E:F_tilde}) is
written as
\[
\pi(m,1)\pi(\begin{pmatrix}\det g^{-n+1}&\\ &
  I_{r-1}\end{pmatrix}, \eta) \varphi(\begin{pmatrix}\det
  g^{n-1}&\\ & I_{r-1}\end{pmatrix}g, 1)
=\pi(m,1)\widetilde{\varphi}(g,1).
\]
This completes the proof.
\end{proof}

With this lemma, one can prove the theorem.
\begin{proof}[Proof of Theorem \ref{T:induction_local}]
Let $\pi_i^{(n)}$ be an irreducible constituent of the restriction
$\pi_i|_{\GLtn_{r_i}}$. By the above lemma, it is an irreducible
constituent of 
\[
\Ind_{\Mtnni N_i^\ast}^{\GLtn_{r_i}}\sigma_i|_{\Mtnni}.
\]
Noting that $\Mtn_i\subseteq\Mtnni$, we have the inclusion
\[
\Ind_{\Mtnni N_i^\ast}^{\GLtn_{r_i}}\sigma_i|_{\Mtnni}
\hookrightarrow \Ind_{\Mtn_i N_i^\ast}^{\GLtn_{r_i}}\sigma_i|_{\Mtn_i}.
\]
But since $\sigma_i$ is a metaplectic tensor product of
$\tau_{i,1},\dots,{\tau_{i,l_i}}$, the restriction $\sigma_i|_{\Mtn_i}$
is a sum of representations of the form
\[
\tau_{i,1}^{(n)}\,\otimest\cdots\otimest\,\tau_{i,l_i}^{(n)}
\]
where each $\tau_{i,t}^{(n)}$ is an irreducible constituent of the
restriction $\tau_{i,t}|_{\GLtn_{^ir_t}}$ of $\tau_{i,t}$ to
$\GLtn_{r_{i,t}}$. Note that this is a metaplectic tensor product
representation of $\Mtn_i$. Hence the metaplectic tensor product
\[
\pi^{(n)}:=\pin_1\,\otimest\cdots\otimest\,\pin_k
\]
is an irreducible constituent of
\begin{equation}\label{E:tensor_of_tensor}
\widetilde{\bigotimes}_{i=1}^k\Ind_{\Mtn_i
  N_i^\ast}^{\GLtn_{r_i}}\,\tau_{i,1}^{(n)}\,\otimest\cdots\otimest\,
{\tau_{i,l_i}^{(n)}}.
\end{equation}
Note that the metaplectic tensor product for the group $\Mtn$ can be
defined for reducible representations, and hence
$\widetilde{\bigotimes}_{i=1}^k$ is defined and the space of the
representation is the same as the one for the usual tensor
product. In particular, the space of the representation
(\ref{E:tensor_of_tensor}) is the usual tensor product. Then one can
see that (\ref{E:tensor_of_tensor}) is equivalent to
\begin{equation}\label{E:tensor_of_tensor2}
\Ind_{\Mtn_1\timest\cdots\timest\Mtn_k (N_1\times\cdots\times
  N_k)^\ast}^{\Mtn}\otimest_{i=1}^k
\,\tau_{i,1}^{(n)}\,\otimest\cdots\otimest\,
{\tau_{i,l_i}^{(n)}}.
\end{equation}
(To see this one can define a map from (\ref{E:tensor_of_tensor}) to
(\ref{E:tensor_of_tensor2}) by
$\varphi_1\,\otimes\cdots\otimes\,\varphi_k\mapsto
\varphi_1\cdots\varphi_k$ where
$\varphi_1\cdots\varphi_k$ is the product of
functions which can be naturally viewed as a function on $\Mtn$.)

Now let $\omega$ be a character on $Z_{\GLt_r}$ that agrees with
$\pin$ on $Z_{\GLt_r}\cap\Mtn$, so that the product
\[
\pin_\omega:=\omega\cdot\pin_n
\]
is a well-defined representation of $Z_{\GLt_r}\Mtn$. Now all the
constituents of the representation (\ref{E:tensor_of_tensor2}) have the
same central character, and hence $\omega$ agrees with
(\ref{E:tensor_of_tensor2}) on $Z_{\GLt_r}\cap\Mtn$, and hence
$\pin_\omega$ is a constituent of 
\[
\Ind_{Z_{\GLt_r}\Mtn_1\timest\cdots\timest\Mtn_k (N_1\times\cdots\times
  N_k)^\ast}^{Z_{\GLt_r}\Mtn}\omega\cdot\otimest_{i=1}^k
\,\tau_{i,1}^{(n)}\,\otimest\cdots\otimest\,{\tau_{i,l_i}^{(n)}}.
\]
Recall that the metaplectic tensor product $\pi_\omega$ is a constituent of
\[
\Ind_{Z_{\GLt_r}\Mtn}^{\Mt}\pin_\omega
\]
and hence a constituent of
\[
\Ind_{Z_{\GLt_r}\Mtn}^{\Mt}\Ind_{Z_{\GLt_r}\Mtn_1\timest\cdots\timest\Mtn_k (N_1\times\cdots\times
  N_k)^\ast}^{Z_{\GLt_r}\Mtn}\omega\cdot\otimest_{i=1}^k
\,\tau_{i,1}^{(n)}\,\otimest\cdots\otimest\,{\tau_{i,l_i}^{(n)}},
\]
which is 
\begin{equation}\label{E:tensor_of_tensor3}
\Ind_{\Qt}^{\Mt}\Ind_{Z_{\GLt_r}\Mtn_1\timest\cdots\timest\Mtn_k (N_1\times\cdots\times
  N_k)^\ast}^{\Qt}\omega\cdot\otimest_{i=1}^k
\,\tau_{i,1}^{(n)}\,\otimest\cdots\otimest\,{\tau_{i,l_i}^{(n)}}
\end{equation}
by inducing in stages. 

Now one can see that the inner induced representation in
(\ref{E:tensor_of_tensor3}) is equal to
\begin{equation}\label{E:tensor_of_tensor4}
\Ind_{Z_{\GLt_r}\MQt^{(n)}}^{\MQt}\omega\cdot\otimest_{i=1}^k
\,\tau_{i,1}^{(n)}\,\otimest\cdots\otimest\,{\tau_{i,l_i}^{(n)}}
\end{equation}
where the unipotent group $(N_1\times\cdots\times N_k)^\ast$ acts
trivially and $\MQt$ is the Levi part of $\Qt$, namely 
\[
\MQt=\Mt_1\timest\cdots\timest\Mt_k.
\]
By Proposition \ref{P:Mezo} applied to the Levi $\MQt$, the representation
(\ref{E:tensor_of_tensor4}) is a sum of the metaplectic tensor
product
\[
(\tau_{1,1}^{(n)}\,\otimest\cdots\otimest\,
{\tau_{1,l_1}^{(n)}}\,\otimest\,\cdots\otimest\, \tau_{k,1}^{(n)}\,\otimest\cdots\otimest\,
{\tau_{k,l_k}^{(n)}})_\omega.
\]
Hence $\pi_\omega$ is a constituent of $\Ind_{\Qt}^{\Mt}(\tau_{1,1}^{(n)}\,\otimest\cdots\otimest\,
{\tau_{1,l_1}^{(n)}}\,\otimest\cdots\otimest\,\tau_{k,1}^{(n)}\,\otimest\cdots\otimest\,
{\tau_{k,l_k}^{(n)}})_\omega$ as claimed.
\end{proof}

\begin{Rmk}\label{R:induction_local}
In the statement of Theorem \ref{T:induction_local}, one can replace
``constituent'' by ``irreducible subrepresentation'' or ``irreducible
quotient'', and the analogous statement is still true. Namely if each
$\pi_i$ is an irreducible subrepresentation (resp. quotient) of the induced
representation in the theorem, then the metaplectic tensor product
$(\pi_1\,\otimes\cdots\otimes\,\pi_k)_\omega$ is also an irreducible
subrepresentation (resp. quotient) of the corresponding induced
representation. To prove it, one can simply replace all the
occurrences of ``constituent'' by ``irreducible subrepresentation'' or
``irreducible quotient'' in the above proof.
\end{Rmk}


\section{\bf The global metaplectic tensor product}


Starting from this section, we will show how to construct the
metaplectic tensor product of unitary automorphic
subrepresentations. Hence all the groups are over the ring of adeles
unless otherwise stated, and it should be recalled here that as in
\eqref{E:star} the group $\GL_r(F)^\ast$ is the image of $\GL_r(F)$ under the
partial map $\s:\GL_r(\A)\rightarrow\GLt_r(\A)$, and we simply write
$\GL_r(F)$ for $\GL_r(F)^\ast$, when there is no danger of
confusion. Also throughout the section the group
$A_{\Mt}(\A)$ is an abelian group that satisfies Hypothesis
($\ast$). 


\subsection{\bf The construction}


The construction is similar to the local case in that first we
consider the restriction to $\GLtn_{r_i}(\A)$, though we need an extra
care to ensure the automorphy.

Let us start with the following.
\begin{Lem}\label{L:unitary_restriction1}
Let $\pi$ be a genuine irreducible automorphic unitary
subrepresentation of $\GLt_r(\A)$. Then the restriction
$\pi|_{\GLtn_r(\A)}$ is completely reducible, namely
\[
\pi|_{\GLtn_r{\A}}=\bigoplus\pin_i
\]
where $\pi_i$ is an irreducible unitary representation of
$\GLtn_r(\A)$. 
\end{Lem}
\begin{proof}
This follows from the admissibility and unitarity of $\pi|_{\GLtn_r{\A}}$.
\end{proof}

The lemma implies that the restriction $\pi_i\|_{\GLtn_{r_i}(\A)}$ is
also completely reducible. (See the notation section for
the notation $\|$.) Hence each irreducible constituent of
$\pi_i\|_{\GLtn_{r_i}(\A)}$ is a subrepresentation. Let
\[
\pin_i\subseteq\pi_i
\]
be an irreducible subrepresentation. Then each vector $f\in\pin_i$ is the restriction to $\GLtn_{r_i}(\A)$
of an automorphic form on $\GLt_{r_i}(\A)$. Hence one can naturally
view each vector $f\in\pin_i$ as a function on the group
\[
H_i:=\GL_{r_i}(F)\GLtn_{r_i}(\A).
\]
Namely the representation $\pin_{i}$ is an
irreducible representation of the group $\GLtn_{r_i}(\A)$ realized in
a space of ``automorphic forms on $H_i$''.

Note that $H_i$ is indeed a group and
moreover it is closed in $\GLt_r(\A)$, which can be shown
by using Lemma \ref{L:discrete_sub}. Also note that each element in
$H_i$ is of the form $(h_i,\xi_i)$ for $h_i\in\GL_{r_i}(F)\GL_{r_i}(\A)$
and $\xi_i\in\mu_n$. By the product formula for the
Hilbert symbol and the block-compatibility of the cocycle $\tau_M$, we
have the natural surjection
\begin{equation}\label{E:global_surjection}
H_1\times\cdots\times H_k\rightarrow M(F)\Mtn(\A)
\end{equation}
given by the map $((h_1,\xi_1),\dots,(h_k,\xi_k))\mapsto
(h_1\cdots h_k,\xi_1\cdots\xi_k)$ because $(\det(h_i),\det(h_j))_\A=1$
for all $i, j=1,\dots, k$.

Now we can construct a metaplectic tensor product of $\pin_1,\dots,\pin_k$, which
is an ``automorphic representation'' of $\Mtn(\A)$ realized in a space
of ``automorphic forms on $M(F)\Mtn(\A)$'' as follows.

\begin{Prop}\label{P:tensor_product_Mtn}
Let
\[
V_{\pin_1}\otimes\cdots\otimes V_{\pin_k}
\]
be the space of functions on the direct product
$H_1\times\dots\times H_k$ which gives rise to an irreducible
representation of $\GLtn_{r_1}(\A)\times\dots\times \GLtn_{r_i}(\A)$
which acts by right translation. Then each function in this space can be viewed as a
function on the group
$M(F)\Mtn(\A)$, namely
it factors through the surjection
as in (\ref{E:global_surjection}) and thus gives rise to an
representation of $\Mtn(\A)$, which we
denote by
\[
\pin:=\pin_1\,\otimest\cdots\otimest\,\pin_k.
\]
Moreover each function in $V_{\pin}=V_{\pin_1}\otimes\cdots\otimes
V_{\pin_k}$ is ``automorphic'' in the sense that it is left invariant
on $M(F)$.
\end{Prop}
\begin{proof}
Since $\pi_i$ is
genuine, for each $f_i\in V_{\pin_i}$ and
$g\in H_i$, we have $f_i(g(1,\xi))=f_i((1,\xi)g)=\xi f_i(g)$ for
all $\xi\in\mu_n$. Now the kernel of the  map
(\ref{E:global_surjection}) consists of the elements of the form
$((I_{r_1},\xi_1),\dots,(I_{r_k},\xi_k))$ with
$\xi_1\cdots\xi_k=1$. Hence each $f_1\otimes\cdots\otimes f_k\in
V_{\pin_1}\otimes\cdots\otimes V_{\pin_k}$, viewed as a function on
the direct product $H_1\times\cdots\times H_k$, factors through the
map (\ref{E:global_surjection}), which we denote by
$f_1\otimest\cdots\otimest f_k$. Namely we can naturally define a
function $f_1\otimest\cdots\otimest f_k $ on $M(F)\Mtn(\A)$ by
\[
(f_1\otimest\cdots\otimest f_k)(
\begin{pmatrix}h_1&&\\ &\ddots&\\ &&h_k\end{pmatrix},\xi)
=\xi f_1(h_1,1)\cdots f_k(h_k,1).
\]

One can see each function $f_1\otimest\cdots\otimest f_k$ is ``automorphic'' as follows: For
$\begin{pmatrix}\gamma_1&&\\ &\ddots&\\
  &&\gamma_k\end{pmatrix}\in M(F)$ and $\begin{pmatrix}g_1&&\\
  &\ddots&\\ &&g_k\end{pmatrix}\in M(F)M^{(n)}(\A)$, we have
\allowdisplaybreaks\begin{align*}
&(f_1\otimest\cdots\otimest f_k)\Big(\s\begin{pmatrix}\gamma_1&&\\ &\ddots&\\
  &&\gamma_k\end{pmatrix}
(\begin{pmatrix}g_1&&\\ &\ddots&\\ &&g_k\end{pmatrix},\xi)\Big)\\
=&(f_1\otimest\cdots\otimest f_k)\Big((\begin{pmatrix}\gamma_1&&\\ &\ddots&\\
  &&\gamma_k\end{pmatrix},\prod_{i=1}^ks_{r_i}(\gamma_i)^{-1})
(\begin{pmatrix}g_1&&\\ &\ddots&\\
  &&g_k\end{pmatrix},\xi)\Big)\quad\text{by definition of $\s$}\\
=&(f_1\otimest\cdots\otimest f_k)\Big(\begin{pmatrix}\gamma_1g_1&&\\ &\ddots&\\
  &&\gamma_kg_k\end{pmatrix},\xi\prod_{i=1}^ks_{r_i}(\gamma_i)^{-1}
\tau_M(\begin{pmatrix}\gamma_1&&\\ &\ddots&\\ &&\gamma_k\end{pmatrix},
\begin{pmatrix}g_1&&\\ &\ddots&\\ &&g_k\end{pmatrix})\Big)\\
=&(f_1\otimest\cdots\otimest f_k)\Big(\begin{pmatrix}\gamma_1g_1&&\\ &\ddots&\\
  &&\gamma_kg_k\end{pmatrix},\xi\prod_{i=1}^ks_{r_i}(\gamma_i)^{-1}\tau_{r_i}(\gamma_i,g_i)\Big)
\quad\text{by block-compatibility of $\tau_M$}\\
=&\left(\xi\prod_{i=1}^ks_{r_i}(\gamma_i)^{-1}\tau_{r_i}(\gamma_i,g_i)\right)
\left(\prod_{i=1}^k f_i(\gamma_ig_i,1)\right)
\quad\text{by definition of $f_1\otimest\cdots\otimest f_k$}\\
=&\xi\prod_{i=1}^kf_i\Big(\gamma_ig_i,s_{r_i}(\gamma_i)^{-1}\tau_{r_i}(\gamma_i,g_i)\Big)
\quad\text{because each $f_i$ is genuine}\\
=&\xi\prod_{i=1}^kf_i\Big((\gamma_i, s_{r_i}(\gamma_i)^{-1})(g_i,1)\Big)
\quad\text{by definition of $\tau_{r_i}$}\\
=&\xi\prod_{i=1}^kf_i(\s_{r_i}(\gamma_i)(g_i,1))
\quad\text{by definition of $\s_{r_i}$}\\
=&\xi\prod_{i=1}^kf_i(g_i,1)
\quad\text{by automorphy of $f_i$}\\
=&(f_1\otimest\cdots\otimest f_k) (\begin{pmatrix}g_1&&\\ &\ddots&\\
  &&g_k\end{pmatrix},\xi)\quad\text{by definition of $f_1\otimest\cdots\otimest f_k$} .
\end{align*}
\end{proof}

Just like the local case, we would like to extend the representation
$\pin$ to a representation of $A_{\Mt}(\A)\Mtn(\A)$ by letting
$A_{\Mt}(\A)$ act as a character. This is certainly possible by
choosing an appropriate character because
$A_{\Mt}(\A)\cap\Mtn(\A)$ is in the center of $\Mtn(\A)$. To
ensure the resulting representation is automorphic, however, one needs
extra steps to do it. 

For this purpose, let us, first, define
\begin{equation}\label{E:F_points_of_AM}
A_{\Mt}\Mtn(F):=A_{\Mt}(\A)\Mtn(\A)\cap\s(M(F)).
\end{equation}
Note that this is not necessarily the same as
$A_{\Mt}(F)\Mtn(F)$. Also let
\[
H:=A_{\Mt}\Mtn(F)\Mtn(\A).
\]
By Proposition \ref{P:M(F)_is_discrete}, the
image of $\s(M(F))$ (and hence $A_{\Mt}\Mtn(F)$) in the quotient
$\Mtn(\A)\backslash\Mt(\A)$ is discrete. Hence  $H$ is a
closed (and hence locally compact) subgroup of $\Mt(\A)$ by using Lemma
\ref{L:discrete_sub} with $G=\Mt(\A), Y=\Mtn(\A)$ and
$\Gamma=A_{\Mt}\Mtn(F)$. Also note
that the group $A_{\Mt}(\A)$ commutes pointwise with the group
$H$ by Proposition \ref{P:Z_M_commute_Mtn} and hence $A_{\Mt}(\A)\cap
H$ is in the center of $H$.

We need the following subtle but important lemma.
\begin{Lem}\label{H:central_character_on_H}
There exists
a character $\chi$ on the center $Z_H$ of $H$ such that
$f(ah)=\chi(a)f(h)$ for $a\in Z_H$, $h\in H$ and $f\in\pin$. (Note
that each $f\in\pin$ is a function on $M(F)\Mtn(\A)$ and hence can be
viewed as a function on $H$.)
\end{Lem}
\begin{proof}
Let $\pin_{H_i}$ be an irreducible subrepresentation of
$\pi_i\|_{H_i}$ such that
$\pin_i\subseteq\pin_{H_i}\|_{\GLtn_{r_i}(\A)}$. Analogously to the
construction of $\pin=\pin_1\otimest\cdots\otimest\pin_k$, one can
construct the representation
$\pin_{H_1}\otimest\cdots\otimest\pin_{H_k}$ of $M(F)\Mtn(\A)$. (The
space of this representation is again a space of ``automorphic forms
on $M(F)\Mtn(\A)$'' but this time it is an irreducible representation
of the group $M(F)\Mtn(\A)$, rather than just $\Mtn(\A)$. The
construction is completely the same as $\pin$ and one can just modify
the proof of Proposition \ref{P:tensor_product_Mtn}.) Then one can see
\[
V_{\pin}\subseteq V_{\pin_{H_1}\otimest\cdots\otimest\pin_{H_k}},
\]
and
\[
(\pin_{H_1}\otimest\cdots\otimest\pin_{H_k})\|_{\Mtn(\A)}\cong
(\pin_{H_1}\otimest\cdots\otimest\pin_{H_k})|_{\Mtn(\A)}.
\]
Let $\pin_H$ be an irreducible subrepresentation of
$(\pin_{H_1}\otimest\cdots\otimest\pin_{H_k})|_H$ such that
\[
V_{\pin}\subseteq V_{\pin_H},
\]
where both sides are spaces of functions on $M(F)\Mtn(\A)$.
Such $\pin_H$ certainly exists, since each $\pi_i$ is unitary and the
unitary structure descends to
$\pin_{H_1}\otimest\cdots\otimest\pin_{H_k}$ making it
unitary. Now since $\pin_H$ is unitary and $H$ is locally
compact, $\pin_H$ admits a central character $\chi$. Thus for each $f\in
V_{\pin_H}$ and {\it a fortiori} each $f\in V_{\pin}$, we have
$f(ah)=\chi(a)f(h)$ for $a\in Z_H$ and $h\in H$.
\end{proof}

In the above lemma, if $a\in Z_H\cap\s(M(F))$, we have $\chi (a)=1$ by
the automorphy of $f$, namely $\chi$ is a ``Hecke character on $Z_H$''.

Now define a character $\omega$ on $A_{\Mt}(\A)$ such that $\omega$ is
trivial on $A_{\Mt}(F)$ and 
\[
\omega|_{A_{\Mt}(\A)\cap H}=\chi|_{A_{\Mt}(\A)\cap H}.
\]
Such $\omega$ certainly exists because $\chi|_{A_{\Mt}(\A)\cap H}$ is
viewed as a character on the group $\s(M(F))\cap(
A_{\Mt}(\A)\cap H)\backslash A_{\Mt}(\A)\cap H$, which is a locally compact
abelian group naturally viewed as a closed subgroup of
the locally compact abelian group $A_{\Mt}(F)\backslash A_{\Mt}(\A)$,
and thus it can be extended to $A_{\Mt}(F)\backslash A_{\Mt}(\A)$.

For each $f\in {\pin}$, viewed as a function on
$H=A_{\Mt}\Mtn(F)\Mtn(\A)$, we extend it to a function
$f_\omega:A_{\Mt}(\A)H\rightarrow\C$ by
\[
f_\omega(ah)=\omega(a)f(h),\quad\text{for all $a\in A_{\Mt}(\A)$ and
  $h\in H$}.
\]
This is well-defined because of our choice of $\omega$, and 
\begin{Lem}\label{L:automorphy_of_f_omega}
The function $f_\omega$ is a function on
$A_{\Mt}(\A)\Mtn(\A)$ such that
\[
f_\omega(\gamma m)=f_\omega(m)
\]
for all $\gamma\in A_{\Mt}\Mtn(F)$ and $m\in\Mtn(\A)$. Namely
$f_\omega$ is an ``automorphic form on $A_{\Mt}(\A)\Mtn(\A)$''.
\end{Lem}
\begin{proof}
The lemma follows from the definition of $f_\omega$ and the obvious
equality $A_{\Mt}(\A)H=A_{\Mt}(\A)\Mtn(\A)$.
\end{proof}

The group $A_{\Mt}(\A)\Mtn(\A) $ acts on the space of functions of the
form $f_\omega$, giving rise to an ``automorphic
representation'' $\pin_\omega$ of $A_{\Mt}(\A)\Mtn(\A) $, namely
\[
V_{\pin_\omega}:=\{f_\omega:f\in\pin\}
\]
and $A_{\Mt}(\A)$ acts as the character $\omega$. As abstract
representations, we have
\begin{equation}\label{E:pin_omega}
\pin_\omega\cong\omega\cdot\pin
\end{equation}
where by $\omega\cdot\pin$ is the representation of the group
$A_{\Mt}(\A)\Mtn(\A)$ extended from $\pin$ by letting $A_{\Mt}(\A)$
act via the character $\omega$.

\quad

We need to make sure the relation between $\pin_\omega$ and its local
analogue we constructed in the previous section. For this, let us start with

\begin{Lem}\label{L:local_global}
Let $\pi\cong\otimest'_v\pi_v$ be a genuine admissible representation of
$\Mt(\A)$. Let $\pin$ be an irreducible quotient of the restriction
$\pi|_{\Mtn(\A)}$. If we write
\[
\pin\cong\underset{v}{\otimest}'\pin_v,
\]
then each $\pin_v$ is an irreducible constituent of the restriction $\pi_v|_{\Mtn_r(F_v)}$.
\end{Lem}
\begin{proof}
Since $\pin$ is an irreducible quotient, there is a surjective
$\Mtn(\A)$ map
\[
T:{\underset{v}{\otimest'}\pi_v}\rightarrow {\underset{v}{\otimest}'\pin_v}.
\]
Fix a place $v_0$. Since $T\neq0$, there exists a pure tensor
$\otimes w_v\in {{\otimest}'\pi_v}$ such that $T(\otimes w_v)\neq
0$. (Note that, as we have seen, the space of ${\otimest'\pi_v}$ is
the space of the usual restricted tensor product $\otimes'_v\pi_v$.) Define
\[
i:{\pi_{v_0}}\rightarrow {\otimest' \pi_v}
\]
by
\[
i(w)=w\otimes(\otimes_{v\neq v_0} w_v)
\]
for $w\in V_{\pi_{v_0}}$. Then the composite $T\circ i:
{\pi_{v_0}}\rightarrow {\otimes'_v\pin_v}$ is a non-zero
$\Mtn(F_{v_0})$ intertwining. Let $w\in {\pi_{v_0}}$ be such that
$T\circ i(w)\neq 0$. Then $T\circ i(w)$ is a finite linear combination
of pure tensors, and indeed it is written as 
\[
T\circ i(w)=x_1\otimes y_1+\cdots+x_t\otimes y_t,
\]
where $x_i\in {\pin_{v_0}}$ and $y_i\in\otimes'_{v\neq v_0}
{\pin_v}$. Here one can assume that $y_1,\dots,y_t$ are linearly
independent. Let $\lambda: \otimes_{v\neq v_0}
{\pin_v}\rightarrow\C$ be a linear functional such that $\lambda
(y_1)\neq 0$ and $\lambda (y_2)=\cdots=\lambda (y_t)=0$. (Such
$\lambda$ certainly exits because $y_1,\dots,y_t$ are linearly
independent.) Consider the map
\[
U:{\otimest'\pin_v}\rightarrow {\pin_{v_0}}
\]
defined on pure tensors by
\[
U(\otimes x_v)=\lambda (\otimes_{v\neq v_0} x)x_{v_0}.
\]
This is a non-zero $\Mtn(F_v)$ intertwining map. Moreover the
composite $U\circ T\circ i$ gives a non-zero $\Mtn(F_v)$ intertwining
map from ${\pi_{v_0}}$ to ${\pin_{v_0}}$. Hence $\pin_{v_0}$ is an
irreducible constituent of the restriction $\pi_{v_0}|_{\Mtn(F_{v_0})}$.
\end{proof}

By taking $k=1$ in the above lemma, one can see that if one writes
\[
\pin_i\cong\underset{v}{\otimest'}\pin_{i,v}
\]
then each local component $\pin_{i,v}$ is an irreducible constituent
of $\pi_{i,v}|_{\GLt_{r_i}(F_v)}$ where $\pi_{i,v}$ is the
$v$-component of $\pi_i\cong\underset{v}{\otimest'}\pi_{i,v}$. Then one
can see that for $\pin=\pin_1\otimest\cdots\otimest\pin_k$, if we write
$\pin\cong\underset{v}{\otimest'}\pin_v$, we have
\[
\pin_v\cong\pin_{1,v}\otimest\cdots\otimest\pin_{k,v}
\]
where the right hand side is the local metaplectic tensor product
representation of $\Mtn(F_v)$. Also one can see that the character
$\omega$ decomposes as $\omega=\underset{v}{\otimest'}\omega_v$ where
$\omega_v$ is a character on $A_{\Mt(F_v)}$. Hence by
(\ref{E:pin_omega}) we have

\begin{Prop}\label{P:global_local_pin_omega}
As abstract representations of $A_{\Mt}(\A)\Mtn(\A)$, we have
\[
\pin_\omega\cong\underset{v}{\otimest'}\pin_{\omega_v},
\]
where
\[
\pin_{\omega_v}=\omega_v\cdot \pin_{1,v}\otimest\cdots\otimest\pin_{k,v}
\]
is the representation of $A_{\Mt}(F_v)\Mtn(F_v)$ as defined in the previous section.
\end{Prop}

\quad

Now that we have constructed the representation $\pin_\omega$ of
$A_{\Mt}(\A)\Mtn(\A)$, we can
construct an automorphic representation of $\Mt(\A)$ analogously to
the local case by inducing it to $\Mt(\A)$, though we need extra care
for the global case. First consider the compactly induced representation
\[
\cInd_{A_{\Mt}(\A)\Mtn(\A)}^{\Mt(\A)}\pin_\omega
=\{\varphi:\Mt(\A)\rightarrow\pin_\omega\}
\]
where $\varphi$ is such that $\varphi(hm)=\pin_\omega(h)\varphi(m)$ for all $
h\in A_{\Mt}(\A)\Mtn(\A)$ and $m\in\Mt(\A)$, and the map $m\mapsto
\varphi(m;1)$ is a smooth function on $\Mt(\A)$ whose support is
compact modulo $A_{\Mt}(\A)\Mtn(\A)$. (Note here that for each
$\varphi\in\Ind_{A_{\Mt}(\A)\Mtn(\A)}^{\Mt(\A)}\pin_\omega$ and
$m\in\Mt(\A)$, $\varphi(m)\in V_{\pin_\omega}$ is a function on
$A_{\Mt}(\A)\Mtn(\A)$. For $m'\in A_{\Mt}(\A)\Mtn(\A)$, we use the
notation $\varphi(m;m')$  for the value of
$\varphi(m)$ at $m'$ instead of writing
$\varphi(m)(m')$.)
Also consider the metaplectic restricted tensor product
\[
\underset{v}{\otimest'}\Ind_{A_{\Mt}(F_v)\Mtn(F_v)}^{\Mt(F_v)}\pin_{\omega_v},
\]
where for almost all $v$ at which all the data defining
$\Ind_{A_{\Mt}(F_v)\Mtn(F_v)}^{\Mt(F_v)}\pin_{\omega_v}$ are unramified,
we choose the spherical vector $\varphi_v^\circ\in
\Ind_{A_{\Mt}(F_v)\Mtn(F_v)}^{\Mt(F_v)}\pin_{\omega_v}$ to be the one
defined by
\[
\varphi_v^\circ(m)=\begin{cases}\pin_{\omega_v}(h)f_v^\circ&\text{if $m=h(k,1)$ for $h\in
    A_{\Mt}(F_v)\Mtn(F_v)$ and
    $(k,1)\in\Mt(\OFv)$};\\
0&\text{otherwise},
\end{cases}
\]
where $f_v^\circ\in\pin_{\omega_v}$ is the spherical vector defining the restricted
metaplectic tensor product
$\pin_\omega=\otimest_v'\pin_{\omega_v}$. (Let us mention that we do not
know if the dimension of the spherical vectors in
$\Ind_{A_{\Mt}(F_v)\Mtn(F_v)}^{\Mt(F_v)}\pin_{\omega_v}$ is one or not.) One has the  injection
\[
T:\underset{v}{\otimest'}\Ind_{A_{\Mt}(F_v)\Mtn(F_v)}^{\Mt(F_v)}\pin_{\omega_v}
\longhookrightarrow \cInd_{A_{\Mt}(\A)\Mtn(\A)}^{\Mt(\A)}\pin_\omega
\]
given by $T(\otimes_v\varphi_v)(m)=\otimes_v
\varphi_v(m_v)\in\otimest_v'\pin_{\omega_v}$. The reason the image of $T$
lies in the compactly induced space is because for almost all $v$, the
support of $\varphi^\circ$ is $A_{\Mt}(F_v)\Mtn(F_v)\Mt(\OFv)$ and for
all $v$ the index of $A_{\Mt}(F_v)\Mtn(F_v)$ in $\Mt(F_v)$ is finite
by (\ref{E:finite_quotient}). (Indeed, the support property and the
finiteness of this index imply that $T$ is actually onto as well,
though we do not use this fact.)

Let 
\[
V(\pin_\omega)=
T\left(\underset{v}{\otimest'}\Ind_{A_{\Mt}(F_v)\Mtn(F_v)}^{\Mt(F_v)}\pin_{\omega_v}\right),
\]
namely $V(\pin_\omega) $ is the image of $T$. For each
$\varphi\in V(\pin_\omega)$, define $\varphit:\Mt(\A)\rightarrow\C$ by
\begin{equation}\label{E:definition}
\varphit(m)=\sum_{\gamma\in  A_{M}M^{(n)}(F)\backslash M(F)}\varphi(\s(\gamma) m;1).
\end{equation}
Let us note that by $A_{M}M^{(n)}(F)$ we mean $p(A_{\Mt}\Mtn(F))$,
which is not necessarily the same as $A_{M}(F)\Mn(F)$, and
\[
\s(A_{M}M^{(n)}(F))=A_{\Mt}\Mtn(F)\subseteq A_{\Mt}(\A)\Mtn(\A).
\]
By the automorphy of $\pin_\omega$, $\varphi$ is left invariant on
$\s(A_{M}M^{(n)}(F))$ and hence the sum is well-defined. Also note
that for each fixed $m\in\Mt(\A)$ the map $m'\mapsto\varphi(m'm;1)$ is
compactly supported modulo
$A_{\Mt}(\A)\Mtn(\A)$. Now by our assumption on $A_{\Mt}$ (Hypothesis
($\ast$)), the image of $M(F)$ is discrete in
$A_M(\A)\Mn(\A)\backslash M(\A)$, and hence the group
$A_{M}M^{(n)}(F)\backslash M(F)$ naturally viewed as a subgroup of
$A_M(\A)\Mn(\A)\backslash M(\A)$ is discrete. Now a discrete subgroup
is always closed by \cite[Lemma 9.1.3 (b)]{Deitmar}. Thus the above
sum is a finite sum, and in particular the sum is convergent. Moreover
one can find $\varphi$ with the property that  the support of the map
$m'\mapsto\varphi(m';1)$ is small enough so that if $\gamma\in
A_{M}\Mn(F)\backslash M(F)$, then $\varphi(\gamma;1)\neq 0$ only at
$\gamma=1$. Thus the map $\varphi\mapsto\varphit$ is not identically
zero.

\begin{Rmk}\label{R:hypothesis}
It should be mentioned here that Hypothesis ($\ast$) is
needed to make sure that the sum in \eqref{E:definition}) is
convergent and not identically zero. The author suspects that either
one can always find $A_{\Mt}$ so that Hypothesis ($\ast$)
is satisfied (which is the case if $n=2$), or even without Hypothesis
($\ast$) one can show that the sum in \eqref{E:definition}
is convergent and not identically zero. But the thrust of this paper
is our application to symmetric square $L$-functions (\cite{Takeda1,
  Takeda2}) for which we only need the case for $n=2$.
\end{Rmk}

One can verify that $\varphit$ is a smooth automorphic form on $\Mt(\A)$: The
automorphy is clear. The smoothness and $K_f$-finiteness follows from the fact that
at each non-archimedean $v$, the induced representation
$\Ind_{A_{\Mt}(F_v)\Mtn(F_v)}^{\Mt(F_v)}\pin_{\omega_v}$ is
smooth and admissible. That $\varphit$ is $\ZZ$-finiteness and of uniform moderate
growth follows from the analogous property of
$\varphi(\s(\gamma)m)$, because the Lie algebra of $\Mt(F_v)$ at
archimedean $v$ is the same as that of $\Mtn(F_v)$.

As we mentioned, the sum in \eqref{E:definition} is finite but  which
$\gamma$ contributes to the sum depends on $m$. Yet, we have
\begin{Lem}\label{L:contribute}
For each $\varphi\in V(\pin_\omega)$, there exists a finite set $S$ of
places containing all the archimedean places such that those $\gamma$'s that contribute
to the sum in \eqref{E:definition} depend only on the classes in 
\[
\Mt(\A)\slash\Mtn(\A)\kappa(M(\mathcal{O}_S)),
\] 
where $\mathcal{O}_S=\prod_{v\notin S}\mathcal{O}_{F_v}$ and $\kappa:
  M(\A)\rightarrow\Mt(\A)$ is the section $m\mapsto(m,1)$.
\end{Lem}
\begin{proof}
By smoothness of $\varphi$ at the non-archimedean places, there exists a finite
set $S$ of places such that for all $k\in\kappa(M(\mathcal{O}_S))$, we have
$k\cdot\varphi=\varphi$. Hence one can see that
\[
\supp(\varphi)=\supp(m\cdot\varphi)
\]
for all $m\in \Mtn(\A)\kappa(M(\mathcal{O}_S))$ because $\Mtn(\A)$ is a normal
subgroup. This proves the lemma.
\end{proof}

\begin{Thm}\label{T:main}
Let
\[
\Vt(\pin_\omega)=\{\varphit:\varphi\in V(\pin_\omega)\}
\]
and $\pi_\omega$ an irreducible constituent of
$\Vt(\pin_\omega)$. Then it is an irreducible automorphic
representation of $\Mt(\A)$ and 
\[
\pi_\omega\cong\underset{v}{\otimest}'\pi_{\omega_v}, 
\]
where $\pi_{\omega_v} $ is the local metaplectic tensor product of Mezo.
Also the isomorphism class of $\pi_\omega$ depends only on the choice of the character
$\omega|_{Z_{\GLt_r}(\A)}$. 
\end{Thm}
\begin{proof}
Since the map $\varphi\mapsto\varphit$ is $\Mt(\A)$-intertwining, the
space $\Vt(\pin_\omega)$ provides a space of (possibly reducible)
automorphic representation of $\Mt(\A)$. Hence $\pi_\omega$ is an
automorphic representation of $\Mt(\A)$.

Since each $\pi_i$ is unitary, so is each $\pin_i$, from
which one can see that $\pin_\omega$ is unitary. Since
$V(\pin_\omega)$ is a subrepresentation of the compactly induced
representation induced from the unitary $\pin_\omega$,
$V(\pin_\omega)$ is unitary. Hence $\pi_\omega$, which is a
subquotient of
$V(\pin_\omega)\cong\underset{v}{\otimest}'\Ind_{A_{\Mt}(F_v)\Mtn(F_v)}^{\Mt(F_v)}\pin_{\omega_v}$,
is actually a quotient of
$\underset{v}{\otimest}'\Ind_{A_{\Mt}(F_v)\Mtn(F_v)}^{\Mt(F_v)}\pin_{\omega_v}$
by admissibility.
With this said, one can derive the isomorphism
$\pi_\omega\cong\underset{v}{\otimest}'\pi_{\omega_v}$ from Lemma
\ref{L:local_global}. Since the local $\pi_{\omega_v}$
depends only on the choice of $\omega_v|_{Z_{\GLt_r}(F_v)}$, the global $\pi_\omega$
depends only on $\omega|_{Z_{\GLt_r}(\A)}$ up to equivalence.
\end{proof}

We call the above constructed $\pi_\omega$ the global metaplectic
tensor product of $\pi_1,\dots,\pi_k$ (with respect to $\omega$) and
write
\[
\pi_\omega=(\pi_1\,\otimest\cdots\otimest\,\pi_k)_\omega.
\]

\begin{Rmk}
We do not know if the multiplicity one theorem holds for the group
$\Mt(\A)$, and hence do not know if the space $\Vt(\pin_\omega)$ has
only one irreducible constituent. In this sense, the definition of
$\pi_\omega$ depends on the choice of the irreducible constituent. For
this reason, the metaplectic tensor product should be
construed as an equivalence class of automorphic representations,
although we know more or less explicit way of expressing automorphic
forms in $\pi_\omega$.
\end{Rmk}

\quad


\subsection{\bf The uniqueness}


Just like the local case, the metaplectic tensor product of
automorphic representations is unique up to twist.

\begin{Prop}\label{P:global_uniqueness}
Let $\pi_1,\dots,\pi_k$ and $\pi'_1,\dots,\pi'_k$ be unitary automorphic subrepresentations
of $\GLt_{r_1}(\A),\dots,\GLt_{r_k}(\A)$. They give rise to isomorphic
metaplectic tensor products with a character $\omega$, \ie
\[
(\pi_1\,\otimest\cdots\otimest\,\pi_k)_\omega\cong
(\pi'_1\,\otimest\cdots\otimest\,\pi'_k)_\omega,
\]
if and only if for each $i$ there exists an automorphic character $\omega_i$ of
$\GLt_{r_i}(\A)$ trivial on $\GLtn_{r_i}(\A)$ such that
$\pi_i\cong\omega_i\otimes\pi'_i$.
\end{Prop}
\begin{proof}
By Theorem \ref{T:main}, the global metaplectic tensor product is
written as the metaplectic restricted tensor product of the local
metaplectic tensor products of Mezo. Hence by Proposition
\ref{P:local_uniqueness} for each $i$ and each place $v$, there is a
character $\omega_{i,v}$ on $\GLt_{r_i}(F_v)$ trivial on $\GLtn_{r_i}(F_v)$ such that
$\pi_{i, v}\cong\omega_{i, v}\otimes\pi'_{i, v}$. Let
$\omega_i=\otimest_v'\omega_{i,v}$. Then
$\pi_i\cong\omega_i\otimes\pi'_i$. The automorphy of $\omega$ follows
from that of $\pi_i$ and $\pi'_i$. This proves the only if part. The
if part follows similarly.
\end{proof}


\subsection{\bf Cuspidality and square-integrability}\label{S:cuspidality}


In this subsection, we will show that the cuspidality and
square-integrability are preserved for the metaplectic tensor product.

\begin{Thm}\label{T:cuspidal}
Assume $\pi_1,\dots,\pi_k$ are all cuspidal. Then the metaplectic tensor product
$\pi_\omega=(\pi_1\,\otimest\cdots\otimest\,\pi_k)_\omega$ is cuspidal.
\end{Thm}
\begin{proof}
Assume $\pi_1,\dots,\pi_k$ are all
cuspidal. It suffices to show that for each $\varphi\in
V(\pin_\omega)$
\[
\int_{U(F)\backslash U(\A)}\varphit(\s(u))\,du=0
\]
for all unipotent radical $U$ of the standard proper parabolic subgroup of $M$, where
recall from Proposition \ref{P:s_split_M} that the partial set
theoretic section $\s:M(\A)\rightarrow\Mt(\A)$ is defined (and a group
homomorphism) on the groups $M(F)$ and $U(\A)$. Note that by
definition of $\varphit$, we have
\begin{equation}\label{E:constant_term}
\int_{U(F)\backslash U(\A)}\varphit(\s(u))\,du=
\int_{U(F)\backslash U(\A)}
\sum_{\gamma\in A_{M}M^{(n)}(F)\backslash M(F)}\varphi(\s(\gamma)\s(u))\,du.
\end{equation}
Here we may assume $\gamma\in M(F)$ is a diagonal matrix, because for
each $\gamma=\diag(\gamma_1,\dots,\gamma_k)$ with
$\gamma_i\in\GL_{r_i}(F)$, we have
\[
\gamma_i=\gamma_i\begin{pmatrix}\det(\gamma_i)^{n-1}&\\
&I_{r_i-1}\end{pmatrix}\begin{pmatrix}\det(\gamma_i)^{-n+1}&\\
&I_{r_i-1}\end{pmatrix},
\]
where $\gamma_i\begin{pmatrix}\det(\gamma_i)^{n-1}&\\
&I_{r_i-1}\end{pmatrix}\in \GL_r^{(n)}(F)$. So for each $u\in U(F)$,
we have $\gamma u\gamma^{-1}\in U(F)$. Thus by the automophy of
$\varphi(\s(u);-)$ for each $u\in U(\A)$, one can see that the map
$u\mapsto \varphi(\s(\gamma)\s(u);1)$ is left invariant on
$U(F)$. Hence for the right hand side of \eqref{E:constant_term}, one
can change the sum and integral. So it suffices to
show $\int_{U(F)\backslash U(\A)}\varphi(\s(\gamma)\s(u);1)\,du=0$. 

 Since we are assuming $\gamma$ is a diagonal matrix, we have $\gamma
 u\gamma^{-1}\in U(\A)$ for all $u\in U(\A)$. Then
\begin{align*}\allowdisplaybreaks
&\int_{U(F)\backslash U(\A)}\varphi(\s(\gamma)\s(u);1)\,du\\
=&\int_{U(F)\backslash U(\A)}\varphi(\s(\gamma)
\s(u)\s(\gamma^{-1})\s(\gamma);1)\,du\\
=&\int_{U(F)\backslash U(\A)}\varphi(\s(\gamma);\s(\gamma)\s(
u)\s(\gamma^{-1}))\,du\quad\text{because $\s(\gamma)\s(
u)\s(\gamma^{-1})\in\widetilde{U}(\A)$}\\
=&\int_{U(F)\backslash U(\A)}\varphi(\s(\gamma);\s(\gamma)\s(
u\gamma^{-1}))\,du\quad\text{by Proposition \ref{P:s_split_M}}\\
=&\int_{U(F)\backslash U(\A)}\varphi(\s(\gamma);\s(\gamma)\s(
\gamma^{-1}u))\,du\quad\text{by change of variables $\gamma
  u\gamma^{-1}\mapsto u$}\\
=&\int_{U(F)\backslash U(\A)}\varphi(\s(\gamma);\s(\gamma)\s(
\gamma^{-1})\s(u))\,du\quad\text{by  Proposition \ref{P:s_split_M}}\\
=&\int_{U(F)\backslash U(\A)}\varphi(\s(\gamma);\s(u))\,du
\quad\text{by  Proposition \ref{P:s_split_M}}.
\end{align*}
We would like to show this is equal to zero. For this purpose, recall that
for each $\gamma$, $\varphi(\s(\gamma))$ is in the space
$V_{\pin_\omega}$ and hence is (a finite sum of functions) of
the form $f_1\otimest\cdots\otimest f_k$ with $f_i\in V_{\pi_i}$ and each
$f_i$ is a cusp form. We may assume $\varphi(\s(\gamma))$ is a simple
tensor $f_1\otimest\cdots\otimest f_k$. Now we can write $U=U_1\times\cdots\times U_k$,
where each $U_i$ is a unipotent subgroup of $\GL_{r_i}$ with at least
one of $U_i$ non-trivial, and
accordingly we denote each element $u\in U$ by
$u=\diag(u_1,\dots,u_k)$. Then by definition of $\s$, we have
\[
\s(u)=(u, \prod_is_{r_i}(u_i)^{-1}),
\]
and
\allowdisplaybreaks\begin{align*}
\varphi(\s(\gamma);\s(u))&=(f_1\otimest\cdots\otimest f_k)(\s(u))\\
&=(f_1\otimest\cdots\otimest f_k)(u, \prod_is_{r_i}(u_i)^{-1})\\
&=\left(\prod_is_{r_i}(u_i)^{-1}\right)f_1(u_1,1)\cdots f_k(u_k,1)\quad\text{by
  definition of $f_1\otimest\cdots\otimest f_k$}\\
&=f_1(u_1,s_{r_i}(u_i)^{-1})\cdots f_k(u_k,s_{r_k}(u_k)^{-1})\quad
\text{because each $f_i$ is genuine}\\
&=f_1(\s_{r_1}(u_1))\cdots f_k(\s_{r_k}(u_k))\quad
\text{by definition of $\s_{r_i}$}.
\end{align*}
Hence
\[
\int_{U(F)\backslash U(\A)}\varphi(\s(\gamma);\s(u))\,du
=\prod_{i=1}^k\int_{U_i(F)\backslash U_i(\A)}f_i(\s_{r_i}(u_i))\,du_i.
\]
This is equal to zero because each $f_i$ is cuspidal and at least one
of $U_i$ is non-trivial.
\end{proof}

Next let us take care of the square-integrability.

\begin{Thm}\label{T:square_integrable}
Assume $\pi_1,\dots,\pi_k$ are all square-integrable modulo center.
Then the metaplectic tensor product
$\pi_\omega=(\pi_1\,\otimest\cdots\otimest\,\pi_k)_\omega$ is
square-integrable modulo center.
\end{Thm}

We need a few lemmas for the proof of this theorem.

\begin{Lem}\label{L:strong_approximation}
Let $S$ be a finite set of places including all the infinite places and let
$\mathcal{O}_S=\prod_{v\notin S}\mathcal{O}_{F_v}$. Then the group
\[
F^\times\A^{\times n}\mathcal{O}_S^\times\backslash\A^\times
\]
is finite.
\end{Lem}
\begin{proof}
Let $F_S:=\prod_{v\in S}F_v$. It suffices to show the subgroup
$F^\times F_S^{\times n}\mathcal{O}_S^\times\subseteq
F^\times\A^{\times n}\mathcal{O}_S^\times$ has a finite index in
$\A^\times$. But it is well-known that the group $F^\times
F_S^\times\mathcal{O}_S^\times$ has a finite index in
$\A^\times$. (Indeed, if $S=S_\infty$, the quotient $F^\times
F_S^{\times}\mathcal{O}_S^\times\backslash\A^\times$ is isomorphic to
the class group of $F$, and hence for general $S$, the group $F^\times
F_S^{\times}\mathcal{O}_S^\times\backslash \A^\times$ is a quotient of
the class group.) Also $F^\times F_S^{\times n}\mathcal{O}_S^\times$ has a
finite index in $F^\times F_S^{\times}\mathcal{O}_S^\times$, because
$F_S^{\times n}$ has a finite index in $F_S^\times$. Hence $F^\times
F_S^{\times n}\mathcal{O}_S^\times$ has a finite index in $\A^\times$.
\end{proof}

\begin{Rmk}
One can show that the group $F^\times\A^{\times
  n}\mathcal{O}_S^\times\backslash\A^\times$ surjects onto
$Cl(F)/Cl(F)^n$, where $Cl(F)$ is the class group of $F$. (See
\cite[Proposition 1, Appendix]{Kable}.) Hence this quotient group is
not trivial in general. Occasionally, however,
it can be shown to be the trivial group depending on $F$ and $n$. This
is the case for example if $n=2$ and $F=\Q$. An interested reader might
want to look at \cite[Appendix]{Kable}.
\end{Rmk}

\begin{Lem}\label{L:quotient_measure}
Let $G$ be a locally compact group and $H, N\subset G$ be closed
subgroups such that $NH$ is a closed subgroup. Further assume that the
quotient measures for $N\backslash G, H\backslash NH$ and $NH\backslash
G$ all exist. (Recall that in general the quotient measure for $N\backslash G$
exists if the modular characters of $G$ and $N$ agree on
$N$.) Then
\begin{align*}
\int_{N\backslash G}f(g)\;dg&=\int_{NH\backslash
G}\int_{N\backslash NH}\;f(hg)\;dh\;dg\\
&=\int_{NH\backslash
G}\int_{N\cap H\backslash H}\;f(hg)\;dh\;dg.
\end{align*}
for all $f\in L^1(N\backslash G)$.
\end{Lem}
\begin{proof}
The first equality is \cite[Cor. 1  VII 47]{Bourbaki}, and the second
equality follows from the natural identification $N\backslash NH\cong
N\cap H\backslash H$.
\end{proof}

Now let $f:\Mt(\A)\rightarrow\C$ be any function. Then the absolute
value $|f|$ is
non-genuine in the sense that it factors through $M(\A)$. Also we let
\[
Z^{(n)}_M(\A):=
\{\begin{pmatrix}a_1^nI_{r_1}&&\\ &\ddots&\\ &&a_k^nI_{r_k}\end{pmatrix}:a_i\in\A^\times\}.
\]
This is a closed subgroup by Lemma \ref{L:closed_subgroup_local} and
\ref{L:closed_subgroup_local_global}.
Note the inclusions
\[
Z_M^{(n)}(\A)\subseteq p(Z_{\Mt}(\A))\subseteq Z_M(\A),
\]
where all the groups are closed subgroups of $M(\A)$. Then we have
\begin{Lem}\label{L:M_square_integrable}
Let $f:M(F)\backslash \Mt(\A)\rightarrow\C$ be an automorphic
form with a unitary central character. Then $f$ is square-integrable
modulo the center $Z_{\Mt}(\A)$ if and only if
$|f|\in L^2( Z^{(n)}_{M}(\A) M(F)\backslash M(\A))$ where $|f|$ is viewed as
a function on $M(\A)$ as noted above.
\end{Lem}
\begin{proof}
Let $f$ be an automorphic form on $\Mt(\A)$ with a unitary central
character. Since $|f|$ is non-genuine, we have
\[
\int_{Z_{\Mt}(\A) M(F)\backslash\Mt(\A)}|f(\tilde{m})|^2\;d\tilde{m}
=\int_{p(Z_{\Mt}(\A)) M(F)\backslash M(\A)}|f(\kappa(m))|^2\;dm,
\]
where recall that $p:\Mt(\A)\rightarrow M(\A)$ is the canonical
projection. Note that the quotient measure on the right hand side
exists because the group $p(Z_{\Mt}(\A))M(F)$ is closed by
\cite[Lemma I.1.5, p.8]{MW} and is unimodular because $p(Z_{\Mt}(\A))$
is unimodular and $M(F)$ is discrete and countable.
By Lemma \ref{L:quotient_measure}, we have
\[
\int_{Z^{(n)}_{M}(\A)  M(F)\backslash M(\A)}|f(\kappa(m))|^2\;dm
=\int_{p(Z_{\Mt}(\A)) M(F)\backslash M(\A)}
\int_{Z_M^{(n)}(\A) p(Z_{\Mt}(F))\backslash p(Z_{\Mt}(\A))}|f(\kappa(zm))|^2\;dz\;dm.
\]
Since for each fixed $m\in M(\A)$, the function $z\mapsto f(\kappa(zm))$ is
a smooth function on $p(Z_{\Mt}(\A))$, there exists a finite set $S$ of places such that
for all $z'\in p(Z_{\Mt}(\mathcal{O}_S))=Z_M(\mathcal{O}_S)\cap
p(Z_{\Mt}(\A))$ we have  $f(\kappa(z'zm))=f(\kappa(zm))$. Hence the inner integral
of the above integral is written as
\begin{equation}\label{E:square_integrable}
\int_{Z^{(n)}_M(\A)p(Z_{\Mt}(\mathcal{O}_S)) p(Z_{\Mt}(F))\backslash p(Z_{\Mt}(\A))}|f(\kappa(zm))|^2\;dz.
\end{equation}
Note that we have the inclusion
\[
Z^{(n)}_M(\A)p(Z_{\Mt}(\mathcal{O}_S)) p(Z_{\Mt}(F))\backslash p(Z_{\Mt}(\A))
\subseteq
Z^{(n)}_M(\A)Z_M(\mathcal{O}_S)Z_M(F)\backslash Z_M(\A),
\]
because $p(Z_{\Mt}(\mathcal{O}_S))\cap
p(Z_{\Mt}(F))=Z_M(\mathcal{O}_S)\cap Z_M(F)=1$,
and note that $Z^{(n)}_M(\A)Z_M(\mathcal{O}_S)Z_M(F)\backslash Z_M(\A)$  
can be identified with the product of $k$ copies of
\[
F^\times\A^{\times n}\;\mathcal{O}_S^\times\backslash\A^\times.
\]
By Lemma \ref{L:strong_approximation}, we know that this is a finite
group, and hence the integral in (\ref{E:square_integrable}) is just a finite
sum. Thus for some finite $z_1,\dots,z_N\in p(Z_{\Mt}(\A))$, we have
\begin{align*}
\int_{Z^{(n)}_{M}(\A)  M(F)\backslash M(\A)}|f(\kappa(m))|^2\;dm
&=\int_{p(Z_{\Mt}(\A)) M(F)\backslash M(\A)}
\sum_{i=1}^N |f(\kappa(z_im))|^2\;dm\\
&=\sum_{i=1}^N \int_{p(Z_{\Mt}(\A)) M(F)\backslash M(\A)}
|f(\kappa(m))|^2\;dm\\
&=N\int_{p(Z_{\Mt}(\A)) M(F)\backslash M(\A)}
|f(\kappa(m))|^2\;dm,
\end{align*}
where for the second equality we used
\[
|f(\kappa(z_im))|=|f((\kappa(z_i)\kappa(m)))|
=|\omega(\kappa(z_1))||f(\kappa(m))|=|f(\kappa(m))|
\]
where $\omega$ is the central character of $f$ which is assumed to be
unitary. The lemma follows from this.
\end{proof}

\begin{Lem}\label{L:pin_square_integrable}
Assume $\pi_1,\dots,\pi_k$ are as in Theorem
\ref{T:square_integrable}. Let $\varphi_i\in\pin_i$ for $i=1,\dots,k$
and $\varphi=\varphi_1\otimest\cdots\otimest\varphi_k\in\pin$, which is
a function on $\Mtn(\A)$. Then
\[
\int_{Z_M^{(n)}(\A)M^{(n)}(F)\backslash M^{(n)}(\A)}|\varphi(\kappa(m))|^2\;dm<\infty.
\]
\end{Lem}
\begin{proof}
Write each element $m\in M(\A)$ as $m=\diag(g_1,\dots,g_k)$ where
$g_i\in \GL_{r_i}(\A)$. Then $\diag(g_1,\dots,g_k)\in M^{(n)}(\A)$ if and
only if $g_i\in \GL_{r_i}^{(n)}(\A)$ for all $i$. Hence the integral in the lemma is the
product of integrals 
\[
\int_{Z_{\GL_{r_i}}^{(n)}(\A)\GL_{r_i}^{(n)}(F)\backslash
  \GL_{r_i}^{(n)}(\A)}|\varphi_i(\kappa(g_i))|^2\;dg_i,
\]
where $Z_{\GL_{r_i}}^{(n)}(\A)$ consists of the elements of the form
$a_iI_{r_i}$ with $a_i\in\A^{\times n}$. So we have to show that this
integral converges. But with Lemma \ref{L:M_square_integrable} applied
to $M=\GL_{r_i}$, we know
\[
\int_{Z^{(n)}_{\GL_{r_i}}(\A)\GL_{r_i}(F)\backslash\GL_{r_i}(\A)}|\varphi_i(\kappa(g_i))|^2\;dg_i<\infty,
\]
because each $\varphi_i$ is square-integrable modulo center. By Lemmas
\ref{L:quotient_measure} and \ref{L:discrete_sub}, this is written as
\[
\int_{Z^{(n)}_{\GL_{r_i}}(\A)\GL^{(n)}_{r_i}(\A)\GL_{r_i}(F)\backslash\GL_{r_i}(\A)}
\int_{Z^{(n)}_{\GL_{r_i}}(\A)\GL^{(n)}_{r_i}(F)\backslash\GL^{(n)}_{r_i}(\A)}
|\varphi_i(\kappa(m_i'm_i))|^2\;dm_i'\;dm_i<\infty.
\]
In particular the inner integral converges, which proves the lemma.
\end{proof}

Now we are ready to prove Theorem \ref{T:square_integrable}.

\begin{proof}[Proof of Theorem \ref{T:square_integrable}]
By Lemma \ref{L:M_square_integrable}, we have only to show
\[
\int_{Z_M^{(n)}(\A)M(F)\backslash M(\A)}|\varphit(\kappa(m))|^2\;dm<\infty.
\]
By Lemma \ref{L:quotient_measure}, we have
\begin{align}\label{E:outer_integral}
\notag&\int_{Z_M^{(n)}(\A)M(F)\backslash M(\A)}|\varphit(\kappa(m))|^2\;dm\\
 =&\int_{Z_M^{(n)}(\A)M^{(n)}(\A)M(F)\backslash
  M(\A)}\int_{Z_M^{(n)}(\A)M^{(n)}(F)\backslash
  M^{(n)}(\A)}|\varphit(\kappa(m'm))|^2\;dm'\;dm\\
\notag =&\int_{Z_M^{(n)}(\A)M^{(n)}(\A)M(F)\backslash
  M(\A)}\int_{Z_M^{(n)}(\A)M^{(n)}(F)\backslash
  M^{(n)}(\A)}\left|\sum_{\gamma}\varphi(\kappa(\gamma
  m'm);1)\right|^2\;dm'\;dm\\
\notag =&\int_{Z_M^{(n)}(\A)M^{(n)}(\A)M(F)\backslash
  M(\A)}\int_{Z_M^{(n)}(\A)M^{(n)}(F)\backslash
  M^{(n)}(\A)}\left|\sum_{\gamma}\varphi(\kappa(\gamma
  m);\kappa(\gamma m'\gamma^{-1}))\right|^2\;dm'\;dm.
\end{align}

Let us show that the inner integral converges. Note that
\begin{align*}
&\int_{Z_M^{(n)}(\A)M^{(n)}(F)\backslash
  M^{(n)}(\A)}\left|\sum_{\gamma}\varphi(\kappa(\gamma
  m);\kappa(\gamma m'\gamma^{-1}))\right|^2\;dm'\\
\leq&
\int_{Z_M^{(n)}(\A)M^{(n)}(F)\backslash
  M^{(n)}(\A)}\sum_{\gamma}\left|\varphi(\kappa(\gamma
  m);\kappa(\gamma m'\gamma^{-1}))\right|^2\;dm',
\end{align*}
and the map $m'\mapsto \left|\varphi(\kappa(\gamma
  m);\kappa(\gamma m'\gamma^{-1}))\right|^2$ is invariant under
$Z_M^{(n)}(\A)M^{(n)}(F)$ on the left. Hence to show the inner
integral converges, it suffices to show the integral
\[
\int_{Z_M^{(n)}(\A)M^{(n)}(F)\backslash M^{(n)}(\A)}
\left|\varphi(\kappa(\gamma  m);\kappa(\gamma m'\gamma^{-1}))\right|^2\;dm'
\]
converges. But this follows from Lemma \ref{L:pin_square_integrable}.

To show the outer integral converges, note that the map $m\mapsto
|\varphit(\kappa(m'm))|^2$ is smooth and hence there
exists a finite set of places $S$ so that $\varphit(\kappa(m'm
k))=\varphit(\kappa(m'm k))$ for all $k\in M(\mathcal{O}_S)$. Thus the
integral in \eqref{E:outer_integral} is (a scalar multiple of)
\[
\int_{Z_M^{(n)}(\A)M^{(n)}(\A)M(F)\backslash
  M(\A)/M(\mathcal{O}_S)}\int_{Z_M^{(n)}(\A)M^{(n)}(F)\backslash
  M^{(n)}(\A)}|\varphit(\kappa(m'm))|^2\;dm'\;dm.
\]
Now the set theoretic map
\[
F^\times\A^{\times n}\underbrace{\mathcal{O}^\times_S\backslash\A^\times
\times\cdots\times
F^\times}_{\text{$k$ copies}}\A^{\times n}\mathcal{O}^\times_S
\backslash\A^\times
\rightarrow Z_M^{(n)}(\A)M^{(n)}(\A)M(F)\backslash
  M(\A)/M(\mathcal{O}_S)
\]
given by
\[
(a_1,\dots,a_k)\mapsto\begin{pmatrix}\iota_1(a_1)&&\\ &\ddots&\\ &&\iota_k(a_k)\end{pmatrix}
\]
where $\iota_i$ is as in (\ref{E:iota}) is a well-defined
surjection. Hence Lemma \ref{L:strong_approximation} implies that the
set 
\[
Z_M^{(n)}(\A)M^{(n)}(\A)M(F)\backslash
  M(\A)/M(\mathcal{O}_S)
\]
is a finite set. Therefore the outer integral of the above integral
is a finite sum and hence converges. This completes the proof.
\end{proof}


\subsection{\bf Twists by Weyl group elements}
\label{S:Weyl_group_global}


Just as we saw in Section \ref{S:Weyl_group_global} for the local
case, the global metaplectic tensor product behaves in the expected
way under the action of the Weyl group elements in $W_M$. Namely

\begin{Thm}\label{T:Weyl_group_global}
Let $w\in W_M$ be such that $^w(\GL_{r_1}\times\cdots\times\GL_{r_k})
=\GL_{r_{\sigma(1)}}\times\cdots\times\GL_{r_{\sigma(k)}}$. Then
we have
\[
^w(\pi_1\,\otimest\cdots\otimest\,\pi_k)_\omega
\cong(\pi_{\sigma(1)}\,\otimest\cdots\otimest\,\pi_{\sigma(k)})_\omega,
\]
where $w$ is viewed as an element in $\GL_r(F)$.
\end{Thm}
\begin{proof}
Note that each $\s(w)\in\GLt_r(\A)$ is written as $\prod_v(w,
s_{r,v}(w))$, where we view $(w, s_{r,v}(w))\in\GLt_r(F_v)$ as an
element of $\GLt_r(\A)$ view the natural embedding
$\GLt_r(F_v)\hookrightarrow\GLt_r(\A)$, and the product $\prod_v$ is
literally the product inside $\GLt_r(\A)$. Then one can see that
\[
^w(\pi_1\,\otimest\cdots\otimest\,\pi_k)_\omega=\otimest'_v\,
^w(\pi_{1,v}\,\otimest\cdots\otimest\,\pi_{k,v})_{\omega_v}.
\]
Hence the theorem follows from the local counter part (Theorem
\ref{T:Weyl_group_local}).
\end{proof}

The following is immediate:
\begin{Prop}
Let $\pi_\omega=(\pi_1\,\otimest\cdots\otimest\,\pi_k)_\omega$. For
$w$ as in the theorem and each automorphic form
$\varphit\in\pi_\omega$, define $^w\varphit:\;^w\Mt(A)\rightarrow\C$ by
\[
^w\varphit(m)=\varphit(\s(w)^{-1}m\s(w))
\]
for $m\in\;^w\Mt(\A)$. Then the representation $^w\pi_\omega$ is realized
in the space 
\[
\{^w\varphit:\varphit\in V_{\pi_\omega}\}.
\]
\end{Prop}

Let us mention the following subtle point. Here we have (at least) two different
realizations of $^w\pi_\omega$ in a space of automorphic
forms on $^w\Mt(\A)$, the one is in the space $\{^w\varphit:\varphit\in
V_{\pi_\omega}\}$  as in the proposition and the other as in the
definition of the metaplectic tensor product
$(\pi_{\sigma(1)}\,\otimest\cdots\otimest\,\pi_{\sigma(k)})_\omega$ by
choosing an appropriate $A_{^w\Mt}$ that satisfies Hypothesis
($\ast$) with respect to the Levi $^w\Mt$ (if possible at
all). Without the multiplicity one property for the group $^w\Mt$, we
do not know if they coincide. But one can see that if $A_{\Mt}$
satisfies Hypothesis ($\ast$) with respect to $\Mt$, then
the group $^wA_{\Mt}:=wA_{\Mt}w^{-1}$ satisfies Hypothesis
($\ast$) with respect to $^w\Mt$. Then if we define
$(\pi_{\sigma(1)}\,\otimest\cdots\otimest\,\pi_{\sigma(k)})_\omega$ by
choosing $A_{^w\Mt}=\,^wA_{\Mt}$, one can see from the construction of
our metaplectic tensor product that the space of
$(\pi_{\sigma(1)}\,\otimest\cdots\otimest\,\pi_{\sigma(k)})_\omega$ is
indeed a space of automorphic forms of the form $^w\varphit$ for
$\varphit\in V_{\pi_\omega}$.


\subsection{\bf Compatibility with parabolic induction}
\label{S:parabolic_induction_global}


Just as the local case, we have the compatibility with parabolic
inductions. But before stating the theorem, let us mention
\begin{Lem}
Let $P=MN$ be the standard parabolic subgroup of $\GL_r$. Then
$\Mt(\A)$ normalizes $N(\A)^\ast$, where $N(\A)^\ast$ is the image of
$N(\A)$ under the partial section $\s:\GL_r(\A)\rightarrow\GLt_r(\A)$.
\end{Lem}
\begin{proof}\label{L:normalizer_global}
One can prove it by using the local analogue (Lemma
\ref{L:normalizer_local}). Namely let $\mt=(m,1)\in\Mt(\A)$,
so $\mt^{-1}=(m^{-1},\tau_r(m,m^{-1})^{-1})$. Also let
$n^\ast=(n,s_r(n)^{-1})\in N(\A)^\ast$. Then
\begin{align*}
\mt n^\ast\mt^{-1}&=(m,1)(n,s_r(n)^{-1})(m^{-1},\tau_r(m,m^{-1})^{-1})\\
&=(mnm^{-1}, s_r(n)^{-1}\tau_r(m,n)\tau(m,m^{-1})^{-1}\tau_r(mn,m^{-1})).
\end{align*}
Then one needs to show
\[
s_r(n)^{-1}\tau_r(m,n)\tau_r(m,m^{-1})^{-1}\tau_r(mn,m^{-1})
=s(mnm^{-1})^{-1},
\]
so that $\mt n^\ast\mt^{-1}=(mnm^{-1})^\ast\in N(\A)^\ast$. But one
can show it by arguing ``semi-locally''. Namely for a sufficiently large
finite set $S$ of places, we have
\begin{align*}
&s_r(n)^{-1}\tau_r(m,n)\tau_r(m,m^{-1})^{-1}\tau_r(mn,m^{-1})\\
=&\prod_{v\in
  S}s_r(n_v)^{-1}\tau_r(m_v,n_v)\tau_r(m_v,m_v^{-1})^{-1}\tau_r(m_vn_v,m_v^{-1})\\
=&\prod_{v\in
  S}s_r(n_v)^{-1}\sigma_r(m_v,n_v)\frac{s_r(m_v)s_r(n_v)}{s_r(m_vn_v)}\\
&\qquad\qquad\cdot\sigma_r(m_v,m_v^{-1})^{-1}\frac{s_r(m_vm_v^{-1})}{s_r(m_v)s_r(m_v^{-1})}
\sigma_r(m_vn_v,m_v^{-1})\frac{s_r(m_vn_v)s_r(m_v^{-1})}{s_r(m_vn_vm_v^{-1})}\\
=&\prod_{v\in S}s_r(m_vn_vm_v^{-1})^{-1}\\
=&s_r(mnm^{-1})^{-1},
\end{align*}
where for the second equality we used (\ref{E:tau_sigma}), for the
third equality we used the same cocycle computation as in the proof of
Lemma \ref{L:normalizer_local} and finally for the last equality we
used $s_r(m_vn_vm_v^{-1})=1$ for all $v\notin S$.
\end{proof}

Let us mention that for the case at hand one can prove this lemma as
we did here. However this lemma holds not just for our $\GLt_r(\A)$
but for covering groups in general . (See \cite[I.1.3(4), p.4]{MW}.) 

At any rate, this lemma allows one to form the global induced
representation
\[
\Ind_{\Mt(\A)N(\A)^\ast}^{\GLt_r(\A)}\pi
\]
for an automorphic representation $\pi$ of $\Mt(\A)$, and hence one
can form Eisenstein series on $\GLt_r(\A)$ just like the
non-metaplectic case. 

With this said, we have

\begin{Thm}\label{T:induction_global}
Let $P=MN\subseteq\GL_r$ be the standard parabolic subgroup whose Levi
part is $M=\GL_{r_1}\times\cdots\times\GL_{r_k}$. Further for each
$i=1,\dots,k$ let $P_i=M_iN_i\subseteq\GL_{r_i}$ be the standard
parabolic of $\GL_{r_i}$ whose Levi part is
$M_i=\GL_{r_{i,1}}\times\cdots\times\GL_{r_{i,l_i}}$. For each $i$,
assume we can find $A_{\Mt_i}$ that satisfies Hypothesis
($\ast$) with respect to $M_i$ (which is always the case if
$n=2$), and we are
given an automorphic representation 
\[
\sigma_i:=(\tau_{i,1}\,\otimest\cdots\otimest\,\tau_{i,l_i})_{\omega_i}
\]
of $\Mt_i(\A)$, which is given as the metaplectic tensor product of the
unitary automorphic subrepresentations $\tau_{i,1},\dots,\tau_{i,l_i}$ of
$\GLt_{r_{i,1}}(\A),\dots,\GLt_{r_{i,l_i}}(\A)$, respectively. Assume that $\pi_i$ is an
irreducible constituent of the induced representation
$\Ind_{\Pt_i(\A)}^{\GLt_{r_i}(\A)}\sigma_i$ and is realized as an
automorphic subrepresentation. Then the metaplectic tensor
product
\[
\pi_\omega:=(\pi_1\,\otimest\cdots\otimest\,\pi_k)_\omega
\]
is an irreducible constituent of the induced representation
\[
\Ind_{\Qt(\A)}^{\Mt(\A)}(\tau_{1,1}\,\otimest\cdots\otimest\,\tau_{1, l_1}\,\otimest\cdots\otimest\,
\tau_{k,1}\,\otimest\cdots\otimest\,\tau_{k, l_k})_\omega,
\]
where $Q$ is the standard parabolic of $M$ whose Levi part is
$M_1\times\cdots\times M_k$, where $M_i\subseteq\GL_{r_i}$ for each $i$.
\end{Thm}
\begin{proof}
This follows from its local analogue (Theorem \ref{T:induction_local})
and the local-global compatibility of the metaplectic tensor product
$\pi_\omega\cong\otimest'\pi_{\omega_v}$.
\end{proof}

\begin{Rmk}
Just as we mentioned in Remark \ref{R:induction_local} for the local
case, in the above theorem one may replace ``constituent'' by
``irreducible subrepresentation'' or ``irreducible quotient'', and the
analogous statement still holds.
\end{Rmk}


\subsection{\bf Restriction to a smaller Levi}\label{S:restriction}


As the last thing in this paper, let us mention an important property
of the metaplectic tensor product which one needs when one computes
constant terms of metaplectic Eisenstein series. (See \cite{Takeda2}.) 

Both locally and globally, let 
\[
M_2=\GL_{r_2}\times\cdots\times\GL_{r_k}
=\{\begin{pmatrix}I_{r_1}&&&\\ &g_2&&\\ &&\ddots&\\
  &&&g_k\end{pmatrix}\in M: g_i\in\GL_{r_i}\}
\]
be viewed as a subgroup of $M$ in the obvious way. We view $\GL_{r-r_1}$
as a subgroup of $\GL_r$ embedded in the right lower corner, and so
$M_2$ can be also viewed as a Levi subgroup of $\GL_{r-r_1}$ embedded
in this way. 

Both locally and globally, we let
\[
\tau_{M_2}:M_2\times M_2\rightarrow\mu_n
\]
be the block-compatible 2-cocycle on $M_2$ defined analogously to $\tau_M$. One
can see that the block-compatibility of $\tau_M$ and $\tau_{M_2}$
implies
\begin{equation}\label{E:tau_M2}
\tau_{M_2}={\tau_M}|_{M_2\times M_2},
\end{equation}
which gives the embeddings
\begin{align*}
\Mt_2\subseteq\Mt\hookrightarrow\GLt_r.
\end{align*}
(Note that the last map is not the natural inclusion because here $\Mt$
is actually $\cMt$, and
that is why we use $\hookrightarrow$ instead of $\subseteq$.) 

For each automorphic form $\varphit\in V_{\pi_\omega}$ in the space of
the metaplectic tensor product, one would like to know which
space the restriction $\varphit|_{\Mt_2(\A)}$ belongs to. Just like
the non-metaplectic case, it would
be nice if this restriction is simply in the space of the metaplectic
tensor product of $\pi_2,\dots,\pi_k$ with respect to the character
$\omega$ restricted to, say, $A_{\Mt}\cap\Mt_2$. But as we will see, this is not necessarily the
case. The metaplectic tensor product is more subtle.

Let us first introduce the subgroup $A_{\Mt_2}$ of $\Mt_2$ which
plays the role analogous to that of $A_{\Mt}$:
\[
A_{\Mt_2}(R):=\{(\begin{pmatrix}I_{r_1}&\\ &A_2\end{pmatrix},\xi):
(\begin{pmatrix}a_1I_{r_1}&\\ &A_2\end{pmatrix},\xi)\in
A_{\Mt}(R)  \text{ for some $a_1\in R^{\times n}$}\}.
\]
Note that $A_{\Mt}(R)\cap\Mt_2(R) \subseteq A_{\Mt_2}(R)$, but the equality
might not hold in general. Also note $A_{\Mt_2}(R)\subseteq
A_{\Mt}(R)$. The following lemma implies that $A_{\Mt_2}$ is abelian.
\begin{Lem}
Let $(\begin{pmatrix}I_{r_1}&\\
  &A_2\end{pmatrix}, \xi), (\begin{pmatrix}I_{r_1}&\\
  &A'_2\end{pmatrix}, \xi')\in A_{\Mt_2}(R)$. Then
\[
\tau_{M_2}(A_2,A'_2)=\tau_{M_2}(A'_2,A_2).
\]
\end{Lem}
\begin{proof}
This follows by the block-compatibility of $\tau_M$ and the fact that
$A_{\Mt}(R)$ is abelian.
\end{proof}
Also one can see that the image of $A_{\Mt_2}(R)$ under the canonical
projection is closed, and hence $A_{\Mt_2}(R)$ is closed.

Another property to be mentioned is
\begin{Lem}\label{L:M_2}
For $R=\A$ or $F_v$, we have
\[
A_{\Mt_2}(R)\Mtn_2(R)=A_{\Mt}(R)\Mtn(R)\cap \Mt_2(R).
\]
Also for global $F$ we have
\[
A_{\Mt_2}\Mtn_2(F)=A_{\Mt}\Mtn(F)\cap \s(M_2(F)),
\]
where by definition
\[
A_{\Mt_2}\Mtn_2(F):=A_{\Mt_2}(\A)\Mtn_2(\A)\cap\s(M(F)),
\]
which is not necessarily the same as $A_{\Mt_2}(F)\Mtn_2(F)$.
\end{Lem}
\begin{proof}
The lemma can be verified by direct computations. Note that for both
cases, the inclusion $\subseteq$ is immediate. To show the reverse
inclusion, we need that  if $a\in A_{\Mt}(R)$ and $m\in \Mtn(R)$ are
such that $am\in A_{\Mt}(R)\Mtn(R)\cap \Mt_2(R)$, one can always write
$a=a_2a_1$ with $a_2\in A_{\Mt_2}(R)$ such that $a_1m\in\Mtn_2(R)$,
and hence $am=a_2(a_1m)\in A_{\Mt_2}(R)\Mtn_2(R)\subseteq A_{\Mt_2}\Mtn_2(F)$.
\end{proof}

\quad

Now assume that our group $A_{\Mt}$ satisfies
the following:
\begin{hypo2}
The group $A_{\Mt}$ satisfies:
\begin{enumerate}
\item[(0)] $A_{\Mt}$ satisfies Hypothesis ($\ast$)
\item $A_{\Mt_2}$ as defined above contains the
center $Z_{\GLt_{r-r_1}}$.
\item $A_{\Mt_2}$ satisfies Hypothesis ($\ast$)
with respect to $\Mt_2$.
\end{enumerate}
\end{hypo2}

As an example of $A_{\Mt}$ satisfying the above hypothesis, we have
\begin{Lem}
If $n=2$, the choice of $A_{\Mt}$ as in Proposition
\ref{P:A_M_for_n=2} satisfies this hypothesis. Moreover, one has
\[
A_{\Mt_2}=A_{\Mt}\cap\Mt_2
\]
both locally and globally.
\end{Lem}
\begin{proof}
This can be merely checked case-by-case.
\end{proof}

\quad

Next for each
$\delta\in\GL_{r_1}(F)$, define
$\omega_{\delta}:A_{\Mt_2}(F)\backslash A_{\Mt_2}(\A)\rightarrow\C^1$ by 
\[
\omega_{\delta}(a)=\omega(\s(\delta) a \s(\delta^{-1})).
\]
Since $\s(\delta) A_{\Mt_2}(\A)\s(\delta^{-1})=A_{\Mt_2}(\A)$ and
$A_{\Mt_2}(\A)\subseteq A_{\Mt}(\A)$, this is
well-defined, and since $\s$ is a homomorphism on $M(F)$,
$\omega_\delta$ is a character. Indeed, one can compute
\begin{equation}\label{E:omega_delta}
\omega_{\delta}(a)=(\det\delta,\det a)^{1+2c}\omega(a)
\end{equation}
because one can see
\[
\s(\delta) a \s(\delta^{-1})=(1,(\det\delta,\det a)^{1+2c})a
\]
and $\omega$ is genuine. Hence 
for each $a\in
A_{\Mt_2}(\A)\cap A_{\Mt_2}\Mtn_2(F)\Mtn_2(\A)$ we have $\omega_\delta(a)=\omega(a)$
because $(\det\delta,\det a)=1$, namely
\[
\omega_\delta|_{A_{\Mt_2}(\A)\cap A_{\Mt_2}\Mtn_2(F)\Mtn_2(\A)}
=\omega|_{A_{\Mt_2}(\A)\cap A_{\Mt_2}\Mtn_2(F)\Mtn_2(\A)}.
\]
Therefore using $\pi_2,\dots,\pi_k$ and $\omega_\delta$, one can construct the metaplectic
tensor product representation of $\Mt_2(\A)$ with respect to $A_{\Mt_2}$, namely
\begin{equation}\label{E:pi_restricted_to_M_2}
\pi_{\omega_\delta}:=(\pi_2\,\otimest\cdots\otimest\,\pi_k)_{\omega_\delta}.
\end{equation}

Then we have
\begin{Prop}\label{P:restriction}
Assume $A_{\Mt}$ satisfies Hypothesis ($\ast\ast$). 
For each $\varphit\in\pi_\omega=(\pi_1\otimest\cdots\otimest\pi_k)_{\omega}$, 
\[
\varphit|_{\Mt_2(\A)}\in\bigoplus_\delta m_\delta\pi_{\omega_\delta},
\]
where
$\pi_{\omega_\delta}=(\pi_2\otimest\cdots\otimest\pi_k)_{\omega_\delta}$
as in (\ref{E:pi_restricted_to_M_2})
and $\delta$ runs through a finite subset of $\GL_{r_1}(F)$, and
$m_\delta\in\Z^{>0}$ is a multiplicity. (Note
that which
$\delta$ appears in the sum could depend on $\varphi$.)
\end{Prop}
\begin{proof}
Recall that
\[
\varphit(m)=\sum_{\gamma\in  A_{M}M^{(n)}(F)\backslash M(F)}\varphi(\s(\gamma) m;1),
\]
where the sum is finite but by Lemma \ref{L:contribute} we know that
which $\gamma$ contributes to the sum depends only on the class in
$\Mt(\A)\slash\Mtn(\A)\kappa(M(\mathcal{O}_S))$ for some finite set $S$ of places. 
Note that $A_M\Mn(F)$ is a
normal subgroup of $M(F)$, and hence $A_{M}\Mn(F)\backslash
M(F)$ is a group. (This is actually an abelian group because it is a
subgroup of the abelian group $A_{M}(\A)\Mn(\A)\backslash M(\A)$.) By
Lemma \ref{L:M_2} we have the inclusion 
\[
A_{M_2}\Mn_2(F)\backslash
M_2(F) \hookrightarrow A_{M}\Mn(F)\backslash
M(F).
\]
Hence we have
\begin{align*}
\varphit(m)&=\sum_{\gamma\in  A_{M}\Mn(F)\backslash
  M(F)}\varphi(\s(\gamma) m;1)\\
&=\sum_{\delta\in M_2(F)A_{M}\Mn(F)\backslash M(F)}\;
\sum_{\mu\in A_{M_2}\Mn_2(F)\backslash
  M_2(F)}\varphi(\s(\mu)\s(\delta) m;1).
\end{align*}
By using Lemma \ref{L:M_2}, one
can see that the map on $\Mt_2(\A)$ defined by $m_2\mapsto
\varphi(m_2\s(\delta)m)$ is in the induced space
$\cInd_{A_{\Mt_2}(\A)\Mtn_2(\A)}^{\Mt_2(\A)}\pin_{\omega, 2}$,
where
$\pin_{\omega,2}:=\omega(\pin_2\,\otimest\cdots\otimest\,\pin_k)$ and
$\omega$ is actually the restriction of $\omega$ to $A_{\Mt_2}(\A)$. Now
since we are assuming that $A_{\Mt}$ satisfies Hypothesis
($\ast\ast$), the inner sum is finite. Since the sum over $\gamma\in  A_{M}\Mn(F)\backslash
  M(F)$ is finite, the outer sum is also finite. 

Since $\delta\in  M_2(F)A_{M}\Mn(F)\backslash M(F)$ can be chosen to
be in $\GL_{r_1}(F)$, we have $\s(\mu)\s(\delta)=\s(\delta)\s(\mu)$. So we
can write
\[
\varphit(m)=\sum_{\delta\in M_2(F)A_{M}\Mn(F)\backslash M(F)}\;
\sum_{\mu\in A_{M_2}\Mn_2(F)\backslash
  M_2(F)}\varphi(\s(\delta)\s(\mu) m;1).
\]
One can see by using Lemma \ref{L:M_2} that for each $\delta$ 
the map on $\Mt_2(\A)$ defined by
\[
m_2\mapsto \varphi(\s(\delta)m_2;1)
\]
is in the induced space
$\cInd_{A_{\Mt_2}(\A)\Mtn_2(\A)}^{\Mt_2(\A)}\pin_{\omega_{\delta}}$,
where
$\pin_{\omega_{\delta}}=\omega_{\delta}(\pin_2\,\otimest\cdots\otimest\,\pin_k)$.
Hence the function on $\Mt_2(\A)$ defined by 
\begin{equation*}
\varphit_\delta: m_2\mapsto \sum_{\mu\in A_{M_2}M_2^{(n)}(F)\backslash
  M_2(F)}\varphi(\s(\delta)\s(\mu) m_2;1)
\end{equation*}
belongs to a space of $\pi_{\omega_{\delta}}$. Hence we may write
\begin{equation}\label{E:varphit_gamma}
\varphit(m_2)=\sum_{\delta\in M_2(F)A_{M}\Mn(F)\backslash M(F)}\varphit_\delta(m_2).
\end{equation}
for all $m_2\in\Mt_2(\A)$. 

Now we will show that this sum can be written as a finite sum
independently of $m_2$. First as we noted above the $\delta$'s that
contribute to the sum depend only on the classes in
$\Mt(\A)\slash\Mtn(\A)\kappa(M(\mathcal{O}_S))$. Hence for each
coset in $\Mt_2(\A)\slash\Mtn_2(\A)\kappa(M_2(\mathcal{O}_S))$ the
$\delta$'s that contribute to the sum are all equal. Also since
$\varphit_\delta$ is left invariant on $\s(M_2(F))$, the $\delta$'s
that contribute to the sum in \eqref{E:varphit_gamma} depend only on
the double cosets in
\[
\s(M_2(F))\backslash \Mt_2(\A)\slash\Mtn_2(\A)\kappa(M_2(\mathcal{O}_S)).
\]
But one can see that this double coset space can be identified with the product of
$k-1$ copies of
\[
F^\times\backslash\A^\times\slash \A^{\times n}\;\mathcal{O}_S^\times
=F^\times\A^{\times n}\;\mathcal{O}_S^\times\backslash\A^\times,
\]
which is finite by Lemma \ref{L:strong_approximation}. Hence there are
only finitely many $\delta$'s such that $\varphit_\delta(m_2)\neq 0$
at least for some $m_2$. 

Hence there exists finitely many
$\delta_1,\dots,\delta_N\in M_2(F)A_{M}\Mn(F)\backslash M(F)$ such that
\[
\varphit|_{\Mt_2(\A)}=\sum_{i=1}^N\varphit_{\delta_i}.
\]
Since we do not know the
multiplicity one property for the group $\Mt_2$, we might have a
possible multiplicity $m_\delta$. This completes the proof.
\end{proof}

\begin{Thm}\label{T:restriction}
Assume that the metaplectic tensor product
$(\pi_1\otimest\cdots\otimest\pi_k)_{\omega} $ is realized with the
group $A_{\Mt}$ which satisfies Hypothesis ($\ast\ast$). Then we have
\[
(\pi_1\otimest\cdots\otimest\pi_k)_{\omega}\|_{\Mt_2(\A)}\subseteq
\bigoplus_{\delta\in\GL_{r_1}(F)} m_\delta(\pi_2\otimest\cdots\otimest\pi_k)_{\omega_\delta},
\]
where $m_\delta\in \Z^{\geq 0}$
\end{Thm}
\begin{proof}
This is immediate from the above proposition. 
\end{proof}

Now we can restrict the metaplectic tensor product ``from the
bottom'', and get the same result. Let
\[
M_{k-1}=\GL_{r_1}\times\GL_{r_{k-1}}=
\{\begin{pmatrix}g_1&&&\\ &\ddots&&\\ &&g_{k-1}&\\
  &&&I_{r_k}\end{pmatrix}\in M: g_i\in\GL_{r_i}\},
\]
and embed $M_{k-1}$ in $\GL_r$ in the upper left corner. Then define
$A_{\Mt_{k-1}}$ and the character $\omega_\delta$ analogously. Also
consider the analogue of Hypothesis ($\ast\ast$), namely
\begin{hypo3}
The group $A_{\Mt}$ satisfies:
\begin{enumerate}
\item[(0)] $A_{\Mt}$ satisfies Hypothesis ($\ast$)
\item $A_{\Mt_{k-1}}$ as defined above contains the
center $Z_{\GLt_{r-r_k}}$.
\item $A_{\Mt_{k-1}}$ satisfies Hypothesis ($\ast$)
with respect to $\Mt_{k-1}$.
\end{enumerate}
\end{hypo3}
Then we have
\begin{Thm}\label{T:restriction}
Assume that the metaplectic tensor product
$(\pi_1\otimest\cdots\otimest\pi_k)_{\omega} $ is realized with the
group $A_{\Mt}$ which satisfies Hypothesis ($\ast\ast\ast$). Then we have
\[
(\pi_1\otimest\cdots\otimest\pi_k)_{\omega}\|_{\Mt_{k-1}(\A)}\subseteq
\bigoplus_{\delta\in\GL_{r_k}(F)}
m_\delta(\pi_1\otimest\cdots\otimest\pi_{k-1})_{\omega_\delta},
\]
where $m_\delta\in \Z^{>0}$
\end{Thm}
\begin{proof}
The proof is essentially the same as the case for the restriction to
$\Mt_2$. We will leave the verification to the reader.
\end{proof}

Also for the case $n=2$, we can do even better.
\begin{Thm}\label{T:restriction_n=2}
Assume $n=2$. 
\begin{enumerate}[(a)]
\item Choose $A_{\Mt}$ to be as in Proposition
\ref{P:A_M_for_n=2}. For $j=2,\dots,k$, let
$M_j=\GL_{r_j}\times\cdots\times\GL_{r_k}\subseteq M$ embedded into
the right lower corner. Then 
\[
(\pi_1\otimest\cdots\otimest\pi_k)_{\omega}\|_{\Mt_j(\A)}\subseteq
\bigoplus_{\omega'} m_{\omega'}(\pi_j\otimest\cdots\otimest\pi_k)_{\omega'},
\]
where $\omega'$ runs through a countable number of characters on
$A_{\Mt_j}=A_{\Mt}\cap\Mt_j$.

\item Choose $A_{\Mt}$ to be as in Proposition
\ref{P:A_M_for_n=2_2}. For $j=1,\dots,k-1$, let
$M_{k-j}=\GL_{r_1}\times\cdots\times\GL_{r_{k-j}}\subseteq M$ embedded into
the left upper corner. Then 
\[
(\pi_1\otimest\cdots\otimest\pi_k)_{\omega}\|_{\Mt_{k-j}(\A)}\subseteq
\bigoplus_{\omega'} m_{\omega'}(\pi_1\otimest\cdots\otimest\pi_{k-j})_{\omega'},
\]
where $\omega'$ runs through a countable number of characters on
$A_{\Mt_{k-j}}=A_{\Mt}\cap\Mt_{k-j}$.
\end{enumerate}
\end{Thm}
\begin{proof}
For (a), one can inductively show that $A_{\Mt_j}=A_{\Mt_{j-1}}\cap\Mt_{j-1}$
satisfies both Hypotheses ($\ast$) and
($\ast\ast$) for the Levi $M_j$. Thus one can successively
apply the above theorem for $j=2,\dots,k$, which proves the
theorem. The case (b) can be treated similarly.
\end{proof}

\begin{Rmk}
In the above theorem, we choose different $A_{\Mt}$ for the two
cases to define $(\pi_1\otimest\cdots\otimest\pi_k)_{\omega}$. They
are, however, equivalent,
because, though the character $\omega$ is a character on $A_{\Mt}$,
the metaplectic tensor product is dependent only on the restriction
$\omega|_{Z_{\GLt_r}}$ to the center.
\end{Rmk}


\appendix\section{\bf On the discreteness of the group
  $A_{M} M^{(n)}(F)\backslash M (F)$}\label{A:topology}


In this appendix, we will discuss the issue of when $A_{\Mt}$ can be chosen so that
the group $A_{M} M^{(n)}(F)\backslash M (F)$ is a discrete subgroup of
$A_{\Mt}(\A)\Mtn(\A)\backslash\Mt(\A)$, and hence the metaplectic
tensor product can be defined. In particular, we will show that if
$n=2$, one can always choose such $A_{\Mt}$, and
hence all the global results hold without any condition. If $n>2$, the
author does not know if it is always possible to chose such nice
$A_{\Mt}$, though he suspects that this is always the case.

Throughout this appendix the field $F$ is a number field. Also for
topological groups $H\subseteq G$, we always assume $H\backslash G$ is
equipped with the quotient topology.

The crucial fact is
\begin{Prop}\label{P:F_is_discrete}
For any positive integer $m$, the image of $F^\times$ in
$\A^{\times m}\backslash \A^\times$ is discrete in the quotient topology.
\end{Prop}
\begin{proof}
Let $K=\prod_{v}K_v\subseteq\A^\times$ be the open neighborhood of the
identity defined by $K_v=\mathcal{O}_{F_v}^\times$ for all finite $v$
and $K_v=F_v^\times$ for all infinite $v$. To show the discreteness of
the image of $F^\times$,
it suffices to show that the set $\A^{\times m} K\cap \A^{\times
  m}F^\times$ has only finitely many points modulo $A^{\times
  m}$. This is because the image of $F^\times$ in $\A^{\times m}\backslash
\A^\times$ will then have an open neighborhood of the identity in the
subspace topology for $\A^{\times
  m}\backslash \A^{\times m}F^\times$ containing finitely many points,
and the quotient $\A^{\times m}\backslash\A^\times$ is Hausdorf since $\A^{\times
  m}$ is closed.

Now let $a^m\in\A^{\times m}$ and $u\in F^\times$ be such that
$a^mu\in \A^{\times m} K\cap \A^{\times m}F^\times$. Then
$u\in\A^{\times m}K$, and so for each finite $v$, we have $u_v\in
F_v^{\times m}K_v$, which implies the fractional ideal $(u)$ generated
by $u$ is $m^{\text{th}}$ power in the group $I_F$ of fractional
ideals of $F$. Namely $(u)\in P_F\cap I_F^m$, where $P_F$ is the
group of principal fractional ideals. On the other hand for any
$(u)\in P_F\cap I_F^m$, one can see that $u\in\A^{\times m}K$. 

Accordingly, if we define
\[
G:=\{u\in F^\times: (u)\in  P_F\cap I_F^m\},
\]
we have the surjection
\[
F^{\times m}\backslash G\rightarrow \A^{\times m}\backslash( \A^{\times
  m} K\cap \A^{\times m}F^\times),
\]
given by $u\mapsto\A^{\times m}u$. So we have only to show that the group $F^{\times m}\backslash G$ is
finite. But note that the map $u\mapsto (u)$ gives rise to the short exact sequence
\[
0\rightarrow U_F^{m}\backslash U_F\rightarrow F^{\times m}\backslash
G\rightarrow P_F^m\backslash P_F\cap I_F^m\rightarrow 0,
\]
where $U_K$ is the group of units for $F$. Now the group
$U_F^{m}\backslash U_F$ is finite by Dirichlet's unit
theorem. The group $P_F^m\backslash P_F\cap I_F^m$ is
isomorphic to the group of $m$-torsions in the class group of $F$ via
the map
\[
P_F^m\backslash P_F\cap I_F^m\rightarrow P_F\backslash I_F,\quad
\mathfrak{A}^m\mapsto\mathfrak{A}
\]
for each fractional ideal $\mathfrak{A}^m \in I_F^m$, and hence
finite. Therefore $F^{\times m}\backslash G$ is finite.
\end{proof}

As a first consequence of this, we have
\begin{Prop}\label{P:M(F)_is_discrete}
The image of $M(F)$ in $M^{(n)}(\A)\backslash M(\A)$ is discrete.
\end{Prop}
\begin{proof}
Let
\[
\Det_M:M(\A)\rightarrow \underbrace{\A^{\times
    n}\backslash\A^\times\times\cdots\times\A^{\times n}\backslash \A^\times}_{k-\text{times}}
\]
be the map defined by
$\Det_M(\diag(g_1,\dots,g_k))=(\det(g_1),\dots,\det(g_k))$. Then
$\ker(\Det_M)=M^{(n)}(\A)$. Moreover the map $\Det_M$ is
continuous. Hence we have a continuous group isomorphism
\[
M^{(n)}(\A)\backslash M(\A)\rightarrow \A^{\times
  n}\backslash\A^\times\times\cdots\times\A^{\times n}\backslash
\A^\times.
\]
Moreover, one can construct the continuous inverse by sending each
$a_i\in \A^{\times  n}\backslash\A^\times$ to the first entry of the
$i^{\text{th}}$ block $\GL_{r_i}(\A)$. But the image of $M(F)$ in $\A^{\times
  n}\backslash\A^\times\times\cdots\times\A^{\times n}\backslash
\A^\times$ under $\Det_M$ is discrete by the above proposition. The
proposition follows.
\end{proof}

As a corollary,

\begin{Cor}\label{C:GCD}
If the center $Z_{\GLt_r}(\A)$ is contained in $\Mtn(\A)$, which is
the case if $n$ divides $nr_i/d$ for all
$i=1,\dots,k$ where $d=\gcd(n, r-1+2cr)$, then Hypothesis
($\ast$) is satisfied and the metaplectic tensor
product can be defined.
\end{Cor}
\begin{proof}
If the center is already in $\Mtn(\A)$, one can choose
$A_{\Mt}(\A)=Z_{\GLt_r}(\A)$ and then $A_{\Mt}(\A)\Mtn(\A)=\Mtn(\A)$,
and so $A_MM^{(n)}(F)=M^{(n)}(F)$. Then by the above proposition,
$A_MM^{(n)}(F)\backslash M(F)$ is discrete in $A_{\Mt}(\A)\Mtn(\A)\backslash\Mt(\A)$.
\end{proof}

Proposition \ref{P:M(F)_is_discrete} also implies
\begin{Prop}
The group $M(F)M^{(n)}(\A)$ (resp. $M(F)^\ast\Mtn(\A)$) is a closed
subgroup of $M(\A)$ (resp. $\Mt(\A)$). 
\end{Prop}
\begin{proof}
It suffices to show it for $M(F)M^{(n)}(\A)$ because the canonical
projection is continuous. But for this, one can apply the following lemma
with $G=M(\A), Y=M^{(n)}(\A)$ and $\Gamma=M(F)$, which will complete
the proof.
\end{proof}

\begin{Lem}\label{L:discrete_sub}
Let $G$ be a Hausdorf topological group. If $\Gamma\subset G$ is a discrete
subgroup and $Y\subset G$ a closed normal subgroup such that the image of $\Gamma$
in $G\slash Y$ is discrete in the quotient topology, then the group $\Gamma Y$ is closed in
$G$.
\end{Lem}
\begin{proof}
Let $p:G\rightarrow G/Y$ be the canonical projection. By our
assumption, the image $p(\Gamma)$ of $\Gamma$ is discrete in the quotient
topology. Now since $Y$ is closed, the quotient $G/Y$ is a Hausdorf
topological group. Hence $p(\Gamma)$ is closed by \cite[Lemma 9.1.3
(b)]{Deitmar}. To show $\Gamma Y$ is closed, it suffices to show every
net $\{\gamma_iy_i\}_{i\in I}$ that converges in $G$, where
$\gamma_i\in\Gamma$ and $y_i\in Y$, converges in $\Gamma Y$. But since
$p$ is continuous, the net $\{p(\gamma_iy_i)\}$ converges in $G/Y$. But
$p(\gamma_iy_i)=p(\gamma_i)$ and $p(\gamma_i)\in p(\Gamma)$. Since
$p(\Gamma)$ is closed and discrete, in order for the net $\{p(\gamma_i)\}$ to
converge, there exists $\gamma\in\Gamma$ such that
$p(\gamma_i)=p(\gamma)$ for all sufficiently large $i\in I$, namely,
the net $\{p(\gamma_i)\}$ is eventually constant. Hence for
sufficiently large $i$, we have $\gamma_iy_i=\gamma y_i'$ for
some $y_i'\in Y$. This means that the net $\{\gamma_iy_i\}$ is
eventually in the set $\gamma Y$. But since $Y$ is closed, so is
$\gamma Y$, which implies that the net  $\{\gamma_iy_i\}$ converges in
$\gamma Y\subset \Gamma Y$.
\end{proof}

Finally in this appendix, we will show that if $n=2$, one can always
choose $A_{\Mt}$ so that the group $A_MM^{(n)}(F)\backslash M(F)$ is
discrete and hence the metaplectic tensor product is defined, and
moreover the metaplectic tensor product can be realized in such a way
that it behaves nicely with the restriction to the smaller rank groups.

First let us note that for any $r$, the center $Z_{\GLt_r}(\A)$ is
given by
\[
Z_{\GLt_r}(\A)=\{(aI_r, \xi): a\in \A^{\times \varepsilon}\},\quad
\varepsilon=
\begin{cases} 
1\quad\text{if $r$ is odd};\\
2\quad\text{if $r$ is even}.
\end{cases}\\
\]
Accordingly, one can see
\[
Z_{\GLt_r}(\A)\GLtt_r(\A)=
\begin{cases}
\GLt_r(\A)\quad\text{if $r$ is odd};\\
\GLtt_r(\A)\quad\text{if $r$ is even}.
\end{cases}
\]

With this said, one can see
\begin{Prop}\label{P:A_M_for_n=2}
Assume $n=2$. Let 
\[
\Zt_i(\A)=Z_{\GLt_{r_i+\cdots+r_k}}(\A)\subseteq
\GLt_{r_i}(\A)\timest\cdots\timest\GLt_{r_k}(\A)
\subseteq\Mt(\A),
\]
and 
\[
A_{\Mt}(\A)=\Zt_1(\A)\Zt_2(\A)\cdots\Zt_k(\A).
\]
Then $A_{\Mt}(\A)$ is a closed abelian subgroup of
$\widetilde{Z_M}(\A)$ and further the
group $A_{\Mt}(\A)\Mtt(\A)$ is closed and the
image of $M(F)$ in $A_M(\A)M^{(2)}(\A)\backslash M(\A)$ as well as in
$A_{\Mt}(\A)\Mtt(\A)\backslash\Mt(\A)$ is discrete.
\end{Prop}
\begin{proof}
It is clear that
$A_{\Mt}(\A)$ is abelian since for each $i=1,\dots, k$, $\Zt_i$ is the
center of $\GLt_{r_i+\cdots+r_k}(\A)$, and hence commutes pointwise
with $\Zt_j(\A)\subseteq \GLt_{r_i+\cdots+r_k}(\A)$ for all $j\geq
i$. To show $A_{\Mt}(\A)$ is closed, it suffices to show
$A_{M}(\A):=p(A_{\Mt}(\A))$ is closed. Now one can write
$A_{M}(\A)=\prod_v'A_M(F_v)$, where $A_M(F_v)$ is defined analogously
to the global case. Then one can see that $Z^{(2)}_M(F_v)\subseteq
A_M(F_v)\subseteq Z_M(F_v)$, and since $Z^{(2)}_M(F_v)$ is closed and
of finite index in $Z_M(F_v)$, so is $A_M(F_v)$. But $Z_M(F_v)$ is
closed in $M(F_v)$ and so $A_M(F_v)$ is closed in $M(F_v)$. Then one
can show that $A_M(\A)$ is closed in $M(\A)$ by Lemma
\ref{L:closed_subgroup_local_global}.

Now one can show by induction on $k$ that the group $A_{M}(\A)M^{(2)}(\A)$ is
the kernel of the map
\[
\Det_M:M(\A)\rightarrow \A^{\times
  \varepsilon_1}\backslash\A^\times\times\cdots\times\A^{\times
  \varepsilon_k}\backslash \A^\times,
\]
where $\varepsilon_i$ is either $1$ or $2$. Hence one has a continuous
group isomorphism
\[
A_{\Mt}(\A)\Mtt(\A)\backslash\Mt(\A)\rightarrow \A^{\times
  \varepsilon_1}\backslash\A^\times\times\cdots\times\A^{\times
  \varepsilon_k}\backslash \A^\times,
\]
where the space on the right is Hausdorff. Hence the space on the left
is Hausdorf as well, which shows $A_{\Mt}(\A)\Mtt(\A)$ is closed. Also
one can show that the image of $M(F)$ is discrete as we did for
Proposition \ref{P:M(F)_is_discrete}.
\end{proof}

\begin{Prop}\label{P:A_M_for_n=2_2}
Assume $n=2$. Let 
\[
\Zt_j(\A)=Z_{\GLt_{r_1+\cdots+r_{k-j}}}(\A)\subseteq
\GLt_{r_1}(\A)\timest\cdots\timest\GLt_{r_{k-j}}(\A)
\subseteq\Mt(\A),
\]
and 
\[
A_{\Mt}(\A)=\Zt_1(\A)\Zt_2(\A)\cdots\Zt_k(\A).
\]
Then $A_{\Mt}(\A)$ is a closed abelian subgroup of
$\widetilde{Z_M}(\A)$ and further the
group $A_{\Mt}(\A)\Mtt(\A)$ is closed and the
image of $M(F)$ in $A_M(\A)M^{(2)}(\A)\backslash M(\A)$ as well as in
$A_{\Mt}(\A)\Mtt(\A)\backslash\Mt(\A)$ is discrete.
\end{Prop}
\begin{proof}
Identical to the previous proposition.
\end{proof}

Let us make the following final remark.
\begin{Rmk}
The above proposition and Corollary \ref{C:GCD} imply Proposition
\ref{P:hypothesis}.  Also for $n>2$, if $n$ and $r=r_1+\cdots+r_k$ are
such that $n$ divides $nr_i/d$ for all  $i=1\cdots k$  where $d=\gcd(n, r-1+2cr)$
and $n$ divides $ nr_i/d_2$ for all $i=2\cdots k$  where $d_2=\gcd(n,
r-r_1-1+2c(r-r_2))$, then $A_{\Mt}=Z_{\GLt_r}$ satisfies Hypothesis
($\ast\ast$), and hence one has the restriction property to the
smaller rank group. Moreover this is always the case, for example, if
$\gcd(n, r-1+2cr)=\gcd(n, r-r_1-1+2c(r-r_1))=1$. Similarly one can
satisfy Hypothesis ($\ast\ast\ast$) if
$n$ divides $nr_i/d$ for all  $i=1\cdots k$  and divides $
nr_i/d_{k-1}$ for all $i=1\cdots k-1$  where $d_{k-1}=\gcd(n,
r-r_{k-1}-1+2c(r-r_{k-1}))$. Those conditions are indeed often satisfied
especially when $n$ is a prime.
\end{Rmk}

\end{document}